\documentclass[reqno]{amsart}

\usepackage[latin1]{inputenc}
\usepackage[T1]{fontenc}
\usepackage[english]{babel}
\usepackage{latexsym}
\usepackage{amssymb}
\usepackage[backend = bibtex, sorting = nty, maxnames = 99]{biblatex}
\usepackage{enumitem}
\usepackage{mathtools}
\usepackage[hyperindex,breaklinks]{hyperref}
\mathtoolsset{centercolon}

\numberwithin{equation}{section} 

\setenumerate{leftmargin=*}

\allowdisplaybreaks

\newtheorem{theorem}{Theorem}[section]
\newtheorem{lem}[theorem]{Lemma}
\newtheorem{cor}[theorem]{Corollary}
\newtheorem{prop}[theorem]{Proposition}

\theoremstyle{remark}
\newtheorem{rem}[theorem]{Remark}

\newcommand{\R}{\mathbb{R}}

\newcommand{\N}{\mathbb{N}}

\newcommand{\Ltwoa}{L^2(\Omega)}

\newcommand{\Ltwoh}{L^2(\R^3_+)}
\newcommand{\Ltwohdom}{L^2(G)}
\newcommand{\Ltwohn}[1]{\|#1\|_{\Ltwoh}}

\newcommand{\Ltwoan}[1]{\|#1\|_{L^2_\gamma(\Omega)}}
\newcommand{\Ltwoanwgamma}[1]{\|#1\|_{L^2(\Omega)}}
\newcommand{\Ltwohnt}[1]{\Gnorm{0}{#1}}

\newcommand{\Ltwodomh}{L^2(G)}
\newcommand{\Ltwodomhn}[1]{\|#1\|_{\Ltwodomh}}

\newcommand{\Hh}[1]{H^{#1}(\R^3_+)}
\newcommand{\Hhdom}[1]{H^{#1}(G)}

\newcommand{\Ha}[1]{H^{#1}(\Omega)}
\newcommand{\Hadom}[1]{H^{#1}(J \times G)}
\newcommand{\Hagamma}[1]{H^{#1}_\gamma(\Omega)}

\newcommand{\Hhta}[1]{H_{\operatorname{ta}}^{#1}(\R^3_+)}

\newcommand{\Hata}[1]{H_{\operatorname{ta}}^{#1}(\Omega)}

\newcommand{\Hatan}[2]{\|#2\|_{H_{\operatorname{ta},\gamma}^{#1}(\Omega)}}

\newcommand{\Hhtan}[2]{\|#2\|_{H_{\operatorname{ta}}^{#1}(\R^3_+)}}

\newcommand{\Hhn}[2]{\|#2\|_{\Hh{#1}}}
\newcommand{\Hhndom}[2]{\|#2\|_{\Hhdom{#1}}}

\newcommand{\Hangamma}[2]{\|#2\|_{\Hagamma{#1}}}
\newcommand{\Hangammadom}[2]{\|#2\|_{H^{#1}_\gamma(J \times G)}}

\newcommand{\G}[1]{G_{#1}(\Omega)}

\newcommand{\GPrime}[1]{G_{#1}(\Omega')}
\newcommand{\Gvar}[1]{\tilde{G}_{#1}(\Omega)}

\newcommand{\Ggamma}[1]{G_{#1, \gamma}(\Omega)}
\newcommand{\Gnorm}[2]{\|#2\|_{G_{#1, \gamma}(\Omega)}}
\newcommand{\Gnormdom}[2]{\|#2\|_{G_{#1, \gamma}(J \times G)}}
\newcommand{\Gnormdomwg}[2]{\|#2\|_{G_{#1}(J \times G)}}

\newcommand{\GnormPrime}[2]{\|#2\|_{G_{#1, \gamma}(\Omega')}}

\newcommand{\Gdom}[1]{G_{#1}(J \times G)}

\newcommand{\Gdomvar}[1]{\tilde{G}_{#1}(J \times G)}

\newcommand{\Gdomnorm}[2]{\|#2\|_{G_{#1, \gamma}(J \times G)}}
\newcommand{\Gdomnormwg}[2]{\|#2\|_{G_{#1}(J \times G)}}

\newcommand{\F}[1]{F_{#1}(\Omega)}

\newcommand{\Fpdw}[2]{F_{#1,#2}(\Omega)}

\newcommand{\Fuwl}[2]{F_{#1}^{\operatorname{#2}}(\Omega)}

\newcommand{\Fupdwl}[3]{F_{#1,#3}^{\operatorname{#2}}(\Omega)}
\newcommand{\Fk}[2]{F_{#1,#2}(\Omega)}
\newcommand{\Fpdk}[3]{F_{#1,#3,#2}(\Omega)}

\newcommand{\Fuwlk}[3]{F_{#1,#3}^{\operatorname{#2}}(\Omega)}

\newcommand{\Fnorm}[2]{\|#2\|_{F_{#1}(\Omega)}}
\newcommand{\Fdom}[1]{F_{#1}(J \times G)}
\newcommand{\Fdomk}[2]{F_{#1,#2}(J \times G)}

\newcommand{\Fdompdw}[2]{F_{#1,#2}(J \times G)}

\newcommand{\Fdomuwl}[2]{F_{#1}^{\operatorname{#2}}(J \times G)}

\newcommand{\Fdomupdwl}[3]{F_{#1,#3}^{\operatorname{#2}}(J \times G)}
\newcommand{\Fdompdk}[3]{F_{#1,#3,#2}(J \times G)}

\newcommand{\Fdomuwlk}[3]{F_{#1,#3}^{\operatorname{#2}}(J \times G)}

\newcommand{\Fdomupdwlk}[4]{F_{#1,#4,#3}^{\operatorname{#2}}(J \times G)}
\newcommand{\Fdomnorm}[2]{\|#2\|_{F_{#1}(J \times G)}}

\newcommand{\Fvardom}[1]{F^0_{#1}(G)}

\newcommand{\Fvarwkdom}[2]{F^0_{#1,#2}(G)}
\newcommand{\Fcoeff}[2]{F_{#1, \operatorname{coeff}}^{\operatorname{#2}}(\R^3_+)}

\newcommand{\Fvarnorm}[2]{\|#2\|_{F^0_{#1}(\R^3_+)}}
\newcommand{\Fvarnormdom}[2]{\|#2\|_{F^0_{#1}(G)}}
\newcommand{\Fnormdom}[2]{\|#2\|_{F_{#1}(J \times G)}}
\newcommand{\E}[1]{E_{#1}(J \times \partial \R^3_+)}
\newcommand{\Edom}[1]{E_{#1}(J \times \partial G)}

\newcommand{\Enorm}[2]{\|#2\|_{E_{#1, \gamma}(J \times \partial \R^3_+)}}
\newcommand{\Enormwgdom}[2]{\|#2\|_{E_{#1}(J \times \partial G)}}
\newcommand{\Enormdom}[2]{\|#2\|_{E_{#1, \gamma}(J \times \partial G)}}

\newcommand{\curl}{\operatorname{curl}}
\newcommand{\Div}{\operatorname{Div}}
\renewcommand{\div}{\operatorname{div}}

\newcommand{\Tr}{\operatorname{Tr}}
\newcommand{\tr}{\operatorname{tr}}

\newcommand{\Hhtanw}[3]{\|#2\|_{H^{#1}_{\operatorname{ta}, #3}(\R^3_+)}}
\newcommand{\Hhtanwt}[3]{\|#2\|_{L^2_\gamma(J,H^{#1}_{\operatorname{ta}, #3}(\R^3_+))}}

\newcommand{\clJ}{\overline{J}}

\newcommand{\dist}{\operatorname{dist}}

\newcommand{\supp}{\operatorname{supp}}

\renewcommand{\epsilon}{\varepsilon}
\renewcommand{\phi}{\varphi}

\bibliography{LiteratureRegularityForLinearMaxwell}

\title[Regularity theory for Maxwell equations]{Regularity theory for nonautonomous Maxwell equations with perfectly conducting boundary conditions}
\author{Martin Spitz}
\address[Martin Spitz]{Department of Mathematics, Karlsruhe Institute of Technology, Englerstr. 2, 76131 Karlsruhe}
\email{martin.spitz@kit.edu}
\subjclass[2010]{35L50, 35Q61}
\keywords{Maxwell equations, perfectly conducting boundary conditions, hyperbolic system, initial boundary 
value problem, characteristic 
boundary, a priori estimates, regularity theory}

\begin{document}
\begin{abstract}
 In this work we study linear Maxwell equations with time- and space-dependent matrix-valued 
 permittivity and permeability
 on domains with a perfectly conducting boundary. This leads to an initial boundary value problem 
 for a first order hyperbolic system with characteristic boundary. We prove a priori estimates for solutions in 
 $H^m$. Moreover, we show the existence of a unique $H^m$-solution if the coefficients 
 and the data are accordingly regular and
 satisfy certain compatibility conditions. Since the boundary is characteristic for the 
 Maxwell system, we have to exploit the divergence conditions in the Maxwell equations in order 
 to derive the energy-type $H^m$-estimates. The combination of these estimates with several regularization 
 techniques then yields the existence of solutions in $H^m$.
\end{abstract}
 \maketitle

\section{Introduction and main result}
\label{Introduction}
The Maxwell equations are the mathematical formulation of the theory of electromagnetism and therefore 
one of the most significant partial differential equations in physics.
In this paper we establish a detailed regularity theory in the case of nonautonomous linear material 
laws and perfectly conducting boundary conditions. 
Such results are known in the autonomous case, where e.g. semigroup methods can be applied. For the 
nonautonomous problem one only has satisfactory results in the full space case~\cite{KatoLinearI, KatoLinearII} 
or for other (absorbing) boundary conditions~\cite{CagnolEller, MajdaOsher, PicardZajaczkowski}. 
The general theory of symmetric hyperbolic systems merely yields partial regularity 
results~\cite{Gues, MajdaOsher, Rauch85}. In this article we obtain a full regularity theory 
using the special structure of Maxwell's equations. Based on these results, in the companion paper~\cite{SpitzQuasilinearMaxwell} 
we develop a complete local wellposedness theory for quasilinear Maxwell equations in~$H^3$, which 
so far was only known for the full space case, see~\cite{KatoQuasilinear}.

In the presence of a 
linear heterogeneous anisotropic medium, the macroscopic Maxwell equations in a domain $G$ read 
\begin{align}
\label{EquationMaxwellSystem}
\begin{aligned}
 \partial_t (\epsilon \boldsymbol{E}) &= \curl \boldsymbol{H} - (\sigma \boldsymbol{E} + \boldsymbol{J}), \qquad  &&\text{for } x \in G, \quad  &&t \in (t_0,T), \\
 \partial_t (\mu \boldsymbol{H}) &= -\curl \boldsymbol{E},  	 &&\text{for } x \in G, &&t \in (t_0,T),  \\
 \div (\epsilon \boldsymbol{E}) &= \rho,  \quad \div (\mu \boldsymbol{H}) = 0, &&\text{for } x \in G, &&t \in (t_0,T),  \\
 \boldsymbol{E} \times \nu &= 0, \quad (\mu \boldsymbol{H}) \cdot \nu = 0, &&\text{for } x \in \partial G, &&t \in (t_0,T),  \\
 \boldsymbol{E}(t_0) &= \boldsymbol{E}_0, \quad \boldsymbol{H}(t_0) = \boldsymbol{H}_0, &&\text{for } x \in G,
\end{aligned}
 \end{align}
for an initial time $t_0 \in \R$. Here $\boldsymbol{E}(t,x) \in \R^3$ and $\boldsymbol{H}(t,x) \in \R^3$ denote the electric 
respectively magnetic field. The conductivity $\sigma(t,x) \in \R^{3 \times 3}$ and current density
$\boldsymbol{J}(t,x) \in \R^3$ are given. The charge density $\rho(t,x)$ depends on the current and 
the electric field via
\begin{equation*}
 \rho(t) = \div(\epsilon(t_0) \boldsymbol{E}_0) - \int_{t_0}^t \div(\sigma \boldsymbol{E} + \boldsymbol{J})(s) ds
\end{equation*}
for all $t \geq t_0$. We further assume that the permittivity $\epsilon(t,x) \in \R^{3 \times 3}$ 
and permeability $\mu(t,x) \in \R^{3 \times 3}$ are symmetric and uniformly positive definite on $(t_0,T) \times G$.
In~\eqref{EquationMaxwellSystem} we have equipped the Maxwell system with the boundary conditions of a 
perfect conductor, where $\nu$ denotes the outer normal unit vector of $G$. 
In order to write the Maxwell equations~\eqref{EquationMaxwellSystem} in the standard form of first 
order systems, we introduce the matrices
\begin{align*}
 J_1 = \begin{pmatrix}
	  0 &0 &0 \\
	  0 &0 &-1 \\
	  0 &1 &0
       \end{pmatrix},
        \quad
 J_2 =  \begin{pmatrix}
         0 &0 &1 \\
         0 &0 &0 \\
         -1 &0 &0
        \end{pmatrix},
        \quad
 J_3 = \begin{pmatrix}
        0 &-1 &0 \\
        1 &0 &0 \\
        0 &0 &0
       \end{pmatrix}
\end{align*}
and
\begin{align}
\label{EquationDefinitionOfAj}
 A_j^{\operatorname{co}} = \begin{pmatrix}
        0 & -J_j \\
        J_j &0 
       \end{pmatrix}
\end{align}
for $j = 1, 2, 3$. Note that $\sum_{j = 1}^3 J_j \partial_j = \curl$. Setting
\begin{equation}
\label{EquationFormForCoefficientsAndInhomogeneityForMaxwell}
 A_0 = \begin{pmatrix}
        \epsilon &0 \\
        0 &\mu
       \end{pmatrix},
\quad 
 D = \begin{pmatrix}
      \partial_t \epsilon + \sigma &0 \\
      0 &\partial_t \mu
     \end{pmatrix}, 
    \quad
 f = \begin{pmatrix}
	- \boldsymbol{J} \\ 0  
     \end{pmatrix} 
\end{equation}
and introducing $u = (\boldsymbol{E}, \boldsymbol{H})$ as new variable, we can write the evolutionary part 
of the Maxwell equations~\eqref{EquationMaxwellSystem} as
\begin{equation}
\label{EquationFirstOrderSystemEvolutionPart}
 A_0 \partial_t u + \sum_{j = 1}^3 A_j \partial_j u + D u = f.
\end{equation}
Under mild regularity conditions on the fields and the coefficients, e.g. $\epsilon \boldsymbol{E}, 
\mu \boldsymbol{H} \in C([t_0,T], \Hhdom{1}) \cap C^1([t_0,T], \Ltwohdom)$ and 
$\operatorname{div}(\sigma \boldsymbol{E} + \boldsymbol{J}) \in L^1((t_0,T), \Ltwoh)$, a solution 
$u = (\boldsymbol{E}, \boldsymbol{H})$ of~\eqref{EquationFirstOrderSystemEvolutionPart}
satisfies the divergence conditions in~\eqref{EquationMaxwellSystem} if they hold at the initial time
$t = t_0$. Similarly, the second part of the boundary conditions, i.e., $(\mu \boldsymbol{H}) \cdot \nu = 0$ on 
$(t_0,T) \times \partial G$ is true if $\boldsymbol{E} \times \nu = 0$ on $(t_0,T) \times \partial G$ 
and $(\mu \boldsymbol{H})(t_0) \cdot \nu = 0$ on $\partial G$. We refer to~\cite[Lemma~7.25]{SpitzDissertation} 
for details. Defining the matrix
\begin{equation*}
 B = \begin{pmatrix}
      0 &\nu_3 &-\nu_2 &0 &0 &0 \\
      -\nu_3 &0 &\nu_1 &0 &0 &0 \\
      \nu_2 &-\nu_1 &0 &0 &0 &0 
     \end{pmatrix},
\end{equation*}
we can thus cast the Maxwell system~\eqref{EquationMaxwellSystem} into the first order linear 
initial boundary value problem
\begin{equation}
  \label{IBVPDomain}
\left\{\begin{aligned}
   A_0 \partial_t u + \sum_{j=1}^3 A_j^{\operatorname{co}} \partial_j u + D u  &= f, \quad &&x \in G, \quad &t \in (t_0,T); \\
   B u &= g, \quad &&x \in \partial G, &t \in (t_0,T); \\
   u(t_0) &= u_0, \quad &&x \in G,
\end{aligned}\right.
\end{equation}
with additional conditions for the initial value.
We also consider inhomogeneous boundary conditions here. Besides being of mathematical interest, 
inhomogeneous boundary conditions for the perfect conductor also have physical relevance, see~\cite{DautrayLionsI}.

The goal of this article is to prove a priori estimates for and the existence of regular 
solutions of~\eqref{IBVPDomain}
provided that the coefficients and the data fulfill suitable regularity and
compatibility conditions, see~\eqref{EquationCompatibilityConditionPrecised} below.  Our main Theorem~\ref{TheoremMainResultOnDomain} 
describes quite detailed how the constants in the a priori estimates depend on suitable norms of 
the coefficients. This precise information is crucial for the nonlinear results in~\cite{SpitzQuasilinearMaxwell}.
In view of 
the above observations our results for problem~\eqref{IBVPDomain} directly
transfer to~\eqref{EquationMaxwellSystem}.

Problem~\eqref{IBVPDomain} is a symmetric hyperbolic system with conservative boundary conditions.
Since the classical work of Friedrichs~\cite{Friedrichs} and Lax-Phillips~\cite{LaxPhillips} on 
symmetric hyperbolic boundary value problems with dissipative boundary conditions, a lot of 
progress has been made. We refer to~\cite{BenzoniGavage} and~\cite{Eller} for an overview of 
the state of the art for hyperbolic systems. 

For Lipschitz coefficients, it is known that the system~\eqref{IBVPDomain} has a unique solution in 
$L^2(J \times G)$ if the data satisfy $u_0 \in \Ltwohdom$, $g \in L^2(J, H^{1/2}(\partial G))$, 
and $f \in L^2(J \times G)$, 
see~\cite{Eller}. Here we set $J = (t_0,T)$. Moreover, one has the basic $L^2$-estimate~\eqref{EquationEllerEstimateInL2} 
for the solutions. We start from this result and use a classical strategy. A 
localization procedure transforms the problem to the half-space. In order to derive a 
priori estimates for more regular solutions, one then differentiates in tangential directions 
and applies the basic $L^2$-estimate to these derivatives, as they again solve an initial boundary 
value problem with known initial value, boundary value, and inhomogeneity, see~\cite{Rauch, Rauch85,RauchMassey}.
But this 
procedure does not work for the derivative in normal direction since we cannot control its 
behavior at the boundary. If the boundary 
matrix $A(\nu) = \sum_{j = 1}^d A_j \nu_j$ is regular, one can express the normal derivative of the solution via the equation 
by tangential derivatives of the solution and lower order terms and thus obtains the desired 
full regularity. Even if $A(\nu)$ is singular (the characteristic case), one can recover normal 
from tangential regularity under certain structural conditions on the problem, see e.g.~\cite{MajdaOsher, OhkuboLinear}.
However, these conditions fail for the Maxwell system~\eqref{IBVPDomain} (which is characteristic 
as the boundary matrix $\sum_{j = 1}^3 A_j^{\operatorname{co}} \nu_j$ is singular) 
with perfectly conducting boundary conditions, cf.~\cite{MajdaOsher}. It also seems that 
Kato's approach from~\cite{KatoLinearI, KatoLinearII} cannot be applied here. On the other hand, 
for general symmetric hyperbolic systems a loss of derivatives in normal direction may occur, 
see e.g.~\cite{Gues,MajdaOsher}.

In our paper we use the structure of Maxwell's equations to prove the full regularity of solutions 
of~\eqref{IBVPDomain}. We proceed as indicated above and focus on the  half-space problem on 
$\R^3_+ = \{x \in \R^3 \colon x_3 > 0\}$. 
The main difficulty is to control the derivative in normal direction $\partial_3 u$. 
Although the boundary matrix is not invertible, using the equation we can bound four components
of $\partial_3 u$ by $\partial_1 u$, $\partial_2 u$, $\partial_t u$, and $f$.
The key step is then to prove that the structure of the Maxwell operator allows us to estimate 
the remaining two components. Here we exploit the divergence conditions for the Maxwell operator, 
respectively for a generalized variable coefficient Maxwell operator which arises due to the 
localization. By means of a Gronwall argument, we can then control these two components. 

Also in the regularization procedure the characteristic boundary poses several challenges. 
It is no longer sufficient (as in the noncharacteristic case) to regularize only in tangential 
directions. However, applying a mollifier in normal direction leads to a loss of derivatives across the boundary. 
We overcome this problem by studying a family of spatially restricted problems. The regularity of the 
corresponding solutions then implies the smoothness of the solution of the original problem. 
To derive the regularity in tangential directions we apply classical techniques from~\cite{Hoermander}. 
Here we rely on the structure of the Maxwell operator which allows us to transform the half-space problem 
to an equivalent one with $A_3 = A_3^{\operatorname{co}}$ so that no commutator terms between 
mollifier and $A_3$ appear. For the differentiability in time yet another regularization technique 
is necessary as the a priori estimates do not allow for a mollifier in time. Moreover, these three 
regularization steps have to be subtly intertwined to retrieve the full regularity of the solution. 

Already for the wellposedness in $L^2$ in~\cite{Eller}
coefficients in $W^{1,\infty}(J \times G)$ (constant outside of a compact set) are needed. 
To treat the initial boundary value problem in higher regularity, we require the coefficients 
$A_0$ and $D$ to belong to 
\begin{align*}
 &\Fdomk{m}{k} = \{A \in W^{1,\infty}(J \times G)^{k \times k} \colon  \partial^\alpha A \in L^\infty(J , \Ltwohdom)^{k \times k} 
 \text{ for all } \alpha \in \N_0^4 \\
 &\hspace{18em} \text{with } 1 \leq |\alpha| \leq m\}, \\
 &\Fdomnorm{m}{A} = \max\{\|A\|_{W^{1,\infty}(J \times G)}, \max_{1 \leq |\alpha| \leq m} \|\partial^\alpha A\|_{L^\infty(J, \Ltwohdom)}\},
\end{align*}
where $m \in \N_0$, see Remark~\ref{RemarkAboutMainTheorem} below for the motivation of this particular space. 
The smoothness of time evaluations of these coefficients will be measured in
\begin{align*}
 &\Fvarwkdom{m}{k} = \{A \in L^\infty(G)^{k \times k} \colon \partial^\alpha A \in \Ltwodomh^{k \times k} \text{ for all } \alpha \in \N_0^3  \text{ with }  1\leq |\alpha| \leq m\}, \\
  &\Fvarnormdom{m}{A} = \max\{\|A\|_{L^\infty(G)}, \max_{1 \leq |\alpha| \leq m}\|\partial^\alpha A\|_{\Ltwohdom}\}.
\end{align*}
By $\Fdompdk{m}{\eta}{k}$ we mean those functions $A$ from $\Fdomk{m}{k}$ with 
$A(t,x)^T = A(t,x) \geq \eta$ for all $(t,x) \in J \times G$, by $\Fdomupdwl{m}{cp}{k}$ those which 
are constant outside of a compact subset of $\overline{J \times G}$, and by $\Fdomupdwl{m}{c}{k}$ those 
which have a limit as $|(t,x)| \rightarrow \infty$. Finally, $\Fdomupdwlk{m}{cp}{\eta}{k}$ and $\Fdomupdwlk{m}{c}{\eta}{k}$ 
are defined as the intersection of $\Fdompdk{m}{\eta}{k}$ with $\Fdomupdwl{m}{cp}{k}$ respectively $\Fdomupdwl{m}{c}{k}$.
We only use the parameters $k = 1$ and $k = 6$ in the following.
As it will be clear from the context which 
parameter we consider, we usually drop it from our notation.

The analysis in~\cite{Eller} requires that the boundary values~$g$ belong to $L^2(J, \! H^{1/2}(\partial G))$. 
In higher regularity we thus take $g$ from the spaces
\begin{align*}
 &\Edom{m} = \bigcap_{j = 0}^m H^j(J, H^{m + \frac{1}{2} - j}(\partial G)), \\
 &\Enormwgdom{m}{g} = \max_{0 \leq j \leq m} \|\partial_t^j g\|_{L^2(J, H^{m + 1/2 - j}(\partial G))}.
\end{align*}
We want to show that under suitable assumptions the solutions of~\eqref{IBVPDomain}
belong to
\begin{equation*}
 \Gdom{m} = \bigcap_{j = 0}^m C^j(\clJ, \Hhdom{m-j}).
\end{equation*}
We equip these spaces with the family of time-weighted norms
\begin{equation*}
  \Gnormdom{m}{v} = \max_{0 \leq j \leq m} \| e_{-\gamma} \partial_t^j v \|_{L^\infty(J, \Hhdom{m-j})}
\end{equation*}
for all $\gamma \geq 0$, where $e_{-\gamma}$ denotes the exponential function $t \mapsto e^{-\gamma t}$.
If $\gamma = 0$, we also abbreviate $\|v\|_{G_{m,0}(J \times G)}$ by $\Gnormdomwg{m}{v}$.
Analogously, any time-space norm indexed by $\gamma$ means the usual norm complemented by
the time weight $e_{-\gamma}$.

As soon as we look for solutions in $\Gdom{m}$ (with data $u_0 \in \Hhdom{m}$, $g \in \Edom{m}$, 
and $f \in \Hadom{m}$) 
with $m \geq 1$, we have to note that the time evaluation of $u$ still has a trace on $\partial G$ which equals 
the time evaluation of the trace of $u$ on $\partial G$. In the case $m = 1$, we thus obtain 
$B u_0 = g(t_0)$ as necessary condition for the existence of a $\Gdom{m}$-solution.
For $m > 1$ there are more of these so called compatibility conditions which have to be satisfied. 
We discuss them in detail in Section~\ref{SectionLocalizationAndFunctionSpaces} below.
We can now state our main result.
\begin{theorem}
  \label{TheoremMainResultOnDomain}
  Let $\eta > 0$, $m \in \N_0$, and $\tilde{m} = \max\{m,3\}$. Fix $r \geq  r_0 >0$. Take a domain $G$ 
  with compact $C^{\tilde{m} + 2}$-boundary.
  Choose $t_0 \in \R$, $T' > 0$ and $T \in (0, T')$ and set 
  $J = (t_0, t_0 + T)$. Take coefficients $A_0 \in \Fdomupdwl{\tilde{m}}{c}{\eta}$ and  
  $D \in \Fdomuwl{\tilde{m}}{c}$ with
  \begin{align*}
    &\Fnormdom{\tilde{m}}{A_0} \leq r, \quad \Fnormdom{\tilde{m}}{D} \leq r, \\
    &\max \{\Fvarnormdom{\tilde{m}-1}{A_0(t_0)},\max_{1 \leq j \leq \tilde{m}-1} \Hhndom{\tilde{m}-j-1}{\partial_t^j A_0(t_0)}\} \leq r_0, \\
    &\max \{\Fvarnormdom{\tilde{m}-1}{D(t_0)},\max_{1 \leq j \leq \tilde{m}-1} \Hhndom{\tilde{m}-j-1}{\partial_t^j D(t_0)}\} \leq r_0.
  \end{align*}
  Choose data $f \in \Hadom{m}$, $g \in \Edom{m}$, and $u_0 \in \Hhdom{m}$ 
  such that the tuple $(t_0, A_0, A_1^{\operatorname{co}}, A_2^{\operatorname{co}}, A_3^{\operatorname{co}}, D, B, f, g,u_0)$
  fulfills the compatibility conditions~\eqref{EquationCompatibilityConditionPrecised} of 
  order $m$. 
   
  Then the linear initial boundary value problem~\eqref{IBVPDomain} has a unique solution $u$ in 
  $\Gdom{m}$. Moreover,  there is a number
 $\gamma_m = \gamma_m(\eta, r, T') \geq 1$ such that
 \begin{align}
  &\Gdomnorm{m}{u}^2  \leq (C_{m,0} + T C_m) e^{m C_1 T} \Big(  \sum_{j = 0}^{m-1} \Hhndom{m-1-j}{\partial_t^j f(t_0)}^2 + \Enormdom{m}{g}^2  \nonumber\\
      &\hspace{18em} + \Hhndom{m}{u_0}^2 \Big) + \frac{C_m}{\gamma}  \Hangammadom{m}{f}^2    \nonumber
 \end{align}
 for all $\gamma \geq \gamma_m$, where $C_i = C_i(\eta,  r,T') \geq 1$ and $C_{i,0} = C_{i,0}(\eta,r_0) \geq 1$ 
 for $i \in \{1,m\}$.
 \end{theorem}
 Several remarks are in order.
 \begin{rem}
  \label{RemarkAboutMainTheorem}
  \begin{enumerate}
   \item Since in the $L^2$-result in~\cite{Eller} the coefficients belong to $W^{1,\infty}(J \times G)$,
   one may ask why we use coefficients from $\Fdom{m}$ and not from 
   $W^{m,\infty}(J \times G)$. The reason is that the space $\Fdom{m}$ occurs naturally if one applies the above result to 
   quasilinear problems, cf.~\cite{SpitzQuasilinearMaxwell}.
   
   In fact, in Theorem~\ref{TheoremMainResultOnDomain} and the other results in this article 
   we could replace $\Fdom{m}$ by the space consisting of those $W^{1,\infty}(J \times G)$-functions whose 
   derivatives up to order $m$ belong to $L^\infty(J, \Ltwohdom) + L^\infty(J \times G)$. Since the part of 
   the derivatives in 
   $L^\infty(J \times G)$ is much easier to treat, we however concentrate on $\Fdom{m}$ here.
   
   \item The assumption that $G$ has a compact boundary is not necessary. We can also treat more general domains 
   with a 
   uniform $C^{\tilde{m} + 2}$-boundary satisfying some extra properties. See~\cite[Chapter~5]{SpitzDissertation} 
   for details.
   
   \item In~\cite{SpitzQuasilinearMaxwell} we also show the finite speed of propagation of solutions of~\eqref{IBVPDomain}.
  \end{enumerate}
 \end{rem}
 
 The proof of Theorem~\ref{TheoremMainResultOnDomain} proceeds in several steps. In Section~\ref{SectionLocalizationAndFunctionSpaces} 
 we describe the localization procedure which transforms~\eqref{IBVPDomain} into a half-space problem. 
 By an additional transformation we manage to keep the matrix $A_3 = A_3^{\operatorname{co}}$ 
 unchanged, which is a crucial ingredient in the regularization procedure.
 In this preliminary section we further define the aforementioned compatibility conditions 
 and provide crucial properties of the function spaces $\Fdom{m}$. We then start to derive a priori estimates 
 in Section~\ref{SectionAPrioriEstimates}. Differentiating in tangential directions and 
 applying the basic $L^2$-estimate, we obtain bounds for the tangential derivatives of the solution.
 As explained above, the crucial step is to derive an a priori bound for the derivative in normal direction 
 via the properties of the Maxwell operator, which is done in Proposition~\ref{PropositionCentralEstimateInNormalDirection}. 
 An iteration argument then yields the a priori estimates of higher order. In Section~\ref{SectionRegularity} 
 we show that the solution of the initial boundary value problem has the same level of regularity as
 the data and the coefficients, roughly speaking. Analogous to the derivation of 
 the a priori estimates, also the regularization procedure is more difficult than in the 
 noncharacteristic case. We use three different 
 regularization techniques in normal, tangential, and time direction,  which also have to be subtly intertwined, 
 see Lemma~\ref{LemmaRegularityInNormalDirection}, Lemma~\ref{LemmaRegularityInSpaceTangential}, 
 and Lemma~\ref{LemmaBasicRegularityInTime}. For several arguments there we need more regular coefficients. 
 However, approximating only the coefficients violates the compatibility conditions.
 We therefore have to construct smooth coefficients and data which approximate the original ones in 
 suitable spaces and which satisfy the compatibility conditions, see~Lemma~\ref{LemmaExistenceOfApproximatingSequence}. 
 Combining these ingredients we finally obtain the claimed regularity of the solution.
 
 \textbf{Notation:} By $m$ we always mean a nonnegative integer. We further denote 
 the differential operator $A_0 \partial_t + \sum_{j = 1}^3 A_j \partial_j + D$ by 
 $L(A_0, \ldots, A_3, D)$. If it is clear from the context which coefficients we consider, we 
 often suppress the argument and only write $L$ for the differential operator. We also set 
 $\partial_0 = \partial_t$.
 
  We further fix a number $T' > t_0$ and take $T \in (t_0, T']$. We set 
 $J = (t_0,T)$ and $\Omega = J \times \R^3_+$. Due to the translation invariance 
 of~\eqref{IBVPDomain}, we often assume without loss of 
 generality $t_0 = 0$.


\section{Compatibility conditions, function spaces, and localization}
\label{SectionLocalizationAndFunctionSpaces}
In this section we denote by $G$ a domain with a compact $C^{m+2}$-boundary or 
the half-space.
With $\Fdom{m}$ and $\Fvardom{m}$ respectively $\Gdom{m}$ we have identified the function spaces
for the coefficients and the solution of~\eqref{IBVPDomain}. In view of the strategy 
described in the introduction, it is clear that we will need corresponding bilinear estimates
for functions from these spaces. Also the spaces $\Gdomvar{m}$, consisting of those functions $v$ 
with $\partial^\alpha v \in L^\infty(J, \Ltwohdom)$ for all $\alpha \in \N_0^4$ with 
$0 \leq |\alpha| \leq m$, will be useful for the fixed point argument of the nonlinear problem,
cf.~\cite{SpitzQuasilinearMaxwell}.
\begin{lem}
 \label{LemmaRegularityForA0}
  Take  $m_1, m_2 \in \N$ with $m_1 \geq m_2$ and $m_1 \geq 2$ and a parameter $\gamma \geq 0$.
 \begin{enumerate}
  \item \label{ItemProductInG0} Let $k \in \{0,\ldots,m_1\}$, $f \in  \Gdomvar{m_1 - k}$,
      and $g \in \Gdomvar{k}$. Then $f  g \in \Gdomvar{0}$ and
      \begin{align*}
       \Gdomnorm{0}{ f g} \leq C \Gdomnormwg{m_1 - k}{f} \Gdomnorm{k}{g}.
      \end{align*}

  \item \label{ItemProductHkappa}  Let $f \in \Gdomvar{m_1}$ and 
      $g \in \Gdomvar{m_2}$. Then $f g \in \Gdomvar{m_2}$ and 
      \begin{align*}
       \Gdomnorm{m_2}{f g} \leq C \min\{&\Gdomnormwg{m_1}{f} \Gdomnorm{m_2}{g}, \\
       &\Gdomnorm{m_1}{f} \Gdomnormwg{m_2}{g}\}.
      \end{align*}
      The result remains true if we replace $\Gdomvar{m_1}$ by $\Fdom{m_1}$ and if 
      we replace both $\Gdomvar{m_1}$ and $\Gdomvar{m_2}$ 
      by $\Fdom{m_1}$ and $\Fdom{m_2}$.
    
  \item \label{ItemProductInL2}  Let $k \in \{0,\ldots,m_1\}$, $f \in  \Hhdom{m_1 -k}$,
      and $g \in \Hhdom{k}$. Then $f  g \in \Ltwodomh$ and
      \begin{align*}
       \Ltwodomhn{ f g} \leq C \Hhndom{m_1 - k}{f} \Hhndom{k}{g}.
      \end{align*}
      
  \item \label{ItemProductIn0} Let $f \in \Hhdom{m_1}$ and $g \in \Hhdom{m_2}$. 
	 Then $f g \in \Hhdom{m_2}$ and
	 \begin{align*}
	  \Hhndom{m_2}{f g} \leq C \Hhndom{m_1}{f}\Hhndom{m_2}{g}.
	 \end{align*}
	 The result remains true if we replace $\Hhdom{m_1}$ by $\Fvardom{m_1}$.
 \end{enumerate}
 The assertions in~\ref{ItemProductInG0} and~\ref{ItemProductHkappa} remain true if we remove the 
 tildes there.
\end{lem}
The proof relies on a term by term analysis of the derivatives of the products combined 
with an appropriate application of the Sobolev embedding theorem and H{\"o}lder's inequality, 
see~\cite[Lemma~2.22]{SpitzDissertation}.

For the regularity results in Section~\ref{SectionRegularity} we have to apply techniques which only 
work for smooth coefficients. We then need an approximating sequence for 
coefficients in $\F{m}$ with properties strong enough to transfer regularity.
\begin{lem}
 \label{LemmaApproximationOfCoefficients}
 Let $m \in \N$.
 Choose $A \in \F{m}$.
 Then there exists a family $\{A_\epsilon\}_{\epsilon > 0}$ 
 in $C^\infty(\overline{\Omega})$ with
 \begin{enumerate}
  \item \label{ItemDerivativeInFmkc} 
	$\partial^\alpha A_\epsilon \in \F{m}$ for all $\alpha \in \N_0^4$ and $\epsilon > 0$,
  \item \label{ItemApproximatingSequenceEstimate}
	$\|A_\epsilon\|_{W^{1,\infty}(\Omega)} \leq C \|A\|_{W^{1,\infty}(\Omega)}$ and 
	$\|\partial^\alpha A_\epsilon\|_{L^\infty(J, \Ltwoh)} \leq C \Fnorm{m}{A}$ 
	for all multiindices $1 \leq |\alpha| \leq m$ and $\epsilon > 0$,
  \item \label{ItemApproximatingSequenceConvergence}
	$A_\epsilon \rightarrow A$ in $L^\infty(\Omega)$ as $\epsilon \rightarrow 0$, and
  \item \label{ItemApproximatingSequenceConvergenceInZero}
	$A_\epsilon(0) \rightarrow A(0)$ in $L^\infty(\R^3_+)$ and $\partial^\alpha  A$ and $\partial^\alpha A_\epsilon$
	have a 	representative in the space $C(\overline{J},\Ltwoh)$ with
	$\partial^\alpha  A_\epsilon(0) \rightarrow \partial^\alpha  A(0)$ in $\Ltwoh$ as 
	$\epsilon \rightarrow 0$ for all $\alpha \in \N_0^4$ with $0 < |\alpha| \leq m-1$.
 \end{enumerate}
 If $A$ is independent of time, the same is true for $A_\epsilon$ for all $\epsilon > 0$. If $A$ 
 additionally belongs to $\Fuwl{m}{cp}$, $\Fuwl{m}{c}$,  $\Fpdw{m}{\eta}$ for a number
 $\eta > 0$, or the 
 intersection of two of these spaces, then the same is true for $A_\epsilon$ for all $\epsilon > 0$.
\end{lem}
The proof again follows standard ideas, see~\cite[Lemma~2.21]{SpitzDissertation} for details.

As indicated in the introduction, we will reduce~\eqref{IBVPDomain} via localization 
to a half-space problem with variable coefficients below. In order to discuss the compatibility 
conditions in a unified framework, we thus consider~\eqref{IBVPDomain} with 
variable coefficients $A_1, A_2, A_3 \in \Fdom{m}$ independent of time for a moment. We further 
fix a positive definite coefficient $A_0 \in \Fdom{m}$, as well as
$D \in \Fdom{m}$, $B \in W^{m+1,\infty}(G)$ and data $f \in \Hadom{m}$, 
$g \in \Edom{m}$, and $u_0 \in \Hhdom{m}$. If~\eqref{IBVPDomain}
has a solution $u$ which belongs to $\Gdom{m}$, then we can 
differentiate the differential equation in~\eqref{IBVPDomain} by Lemma~\ref{LemmaRegularityForA0}
up to $m-1$-times in time to 
obtain that
\begin{equation}
\label{EquationTimeDerivativesOfSolution}
 \partial_t^p u(t) = S_{G,m,p}(t,A_0, A_1,A_2,A_3,D,f,u(t)),
\end{equation}
for all $t \in \clJ$ and $p \in \{0, \ldots, m\}$, where $S_{G,m,p} = S_{G,m,p}(t_0,A_0, A_1,A_2,A_3,D,f,u_0)$ is defined by
\begin{align}
\label{EquationDefinitionSmp}
 &S_{G,m,0} = u_0, \nonumber \\
 &S_{G,m,p}= A_0(t_0)^{-1} \Big(\partial_t^{p-1} f(t_0) 
  - \sum_{j = 1}^3 A_j \partial_j S_{G,m, p-1} - \sum_{l=1}^{p-1} \binom{p-1}{l} \partial_t^l A_0(t_0) S_{G,m, p-l} \nonumber \\
  &\hspace{10em}- \sum_{l=0}^{p-1} \binom{p-1}{l} \partial_t^l D(t_0) S_{G,m,p-1-l}\Big), 
\end{align}
for $1 \leq p \leq m$. On the other hand, differentiating the boundary condition 
in~\eqref{IBVPDomain} up to $m-1$-times in time and inserting any $t \in \clJ$, we derive
\begin{equation}
\label{EquationDifferentiatedBoundaryCondition}
 B \partial_t^p u(t) = \partial_t^p g(t)
\end{equation}
on $\partial G$ for all $0 \leq p \leq m-1$ and $t \in \clJ$. Combining~\eqref{EquationTimeDerivativesOfSolution} 
 with~\eqref{EquationDifferentiatedBoundaryCondition} in $t = t_0$, we obtain the compatibility 
 conditions of order $m$ 
 \begin{equation}
 \label{EquationCompatibilityConditionPrecised}
 B S_{G,m,p}(t_0, A_0, \ldots, A_3, D, f, u_0) = \partial_t^p g(t_0) \quad \text{on } \partial G \text{ for } 0 \leq p \leq m-1
\end{equation}
for the coefficients and data. These conditions are thus necessary for the existence of a 
solution in $\Gdom{m}$. We will show in Section~\ref{SectionRegularity} that they are also sufficient.
If it is clear from the context which domain $G$ we consider, we will often suppress it in the notation.

The operators $S_{G,m,p}$ for $0 \leq p \leq m$ appear frequently in the following and 
corresponding estimates are indispensable.
\begin{lem}
 \label{LemmaEstimatesForHigherOrderInitialValues}
 Take $\eta > 0$, $m \in \N$, and set $\tilde{m} := \max\{m,3\}$.  Pick $r_0 > 0$. Choose
 $A_0 \in \Fdompdw{\tilde{m}}{\eta}$,  $A_1, A_2, A_3 \in \F{\tilde{m}}$ independent of time, 
 and $D \in \Fdom{\tilde{m}}$
 with
 \begin{align*}
  &\Fvarnormdom{\tilde{m}-1}{A_i(t_0)} \leq r_0, \quad \Fvarnormdom{\tilde{m}-1}{D(t_0)} \leq r_0, \\
  &\max_{1 \leq j \leq m-1}\Hhndom{\tilde{m}-1-j}{\partial_t^j A_0(t_0)} \leq r_0, \quad
  \max_{1 \leq j \leq m-1}{\Hhndom{\tilde{m}-1-j}{\partial_t^j D(t_0)}} \leq r_0
 \end{align*}
 for all $i \in \{0, \ldots, 3\}$.
 Take $f \in \Hadom{m}$ and $u_0 \in \Hhdom{m}$. 
 Then the function $S_{G,m,p}(t_0,A_0,\ldots,A_3, D, f,u_0)$ is contained in $\Hhdom{m-p}$ for all $p \in \{0,\ldots,m\}$.
 Moreover, there exist constants
 $C_{m,p} = C_{m,p}(\eta,r_0) > 0$ such that
 \begin{align*}
  \Hhndom{m-p}{S_{G,m,p}} \leq C_{m,p} \Big(\sum_{j = 0}^{p-1} \Hhndom{m-1-j}{\partial_t^j f(t_0)} + \Hhndom{m}{u_0} \Big)
 \end{align*}
 for $0 \leq p \leq m$.
\end{lem}
For the proof one applies Lemma~\ref{LemmaRegularityForA0} to the terms appearing in~\eqref{EquationDefinitionSmp},
see~\cite[Lemma~2.33]{SpitzDissertation} for details.

Via localization, we reduce the initial boundary value problem~\eqref{IBVPDomain} on general domains $G$
to the corresponding problem on the half-space. While this procedure simplifies the underlying 
domain, we no longer deal with the curl operators but more general variable coefficient linear first 
order differential operators on the half-space. However, this operator still has 
a structural similarity with the curl operator. Since this structure is utterly 
important in the following, we want at least to indicate how the half-space problem arises from 
the local charts.

To that purpose, assume that we have chosen a finite covering $(U_i)_{i \in \N}$ of $\partial G$ with 
corresponding charts $\phi_i$, which are $C^{m+2}$-diffeomorphisms from $U_i$ to subsets $V_i$ 
of $B(0,1)$. Denoting the composition with $\phi_i^{-1}$ by $\Phi_i$, we consider on the 
half-space the differential operator
\begin{align}
  \label{EquationDerivationTransportedDifferentialOperator}
  L^i v &= \Phi_i \Big( A_0 \partial_t + \sum_{j=1}^3 A_j^{\operatorname{co}} \partial_j + D \Big) \Phi_i^{-1} v \nonumber \\
  &= \Phi_i \Big( A_0 \partial_t v \circ \phi_i + \sum_{j=1}^3 A_j^{\operatorname{co}} \partial_j (v \circ \phi_i) + D v \circ \phi_i \Big) \nonumber \\
  &= \Phi_i A_0 \, \partial_t v + \Phi_i \Big(\sum_{j = 1}^3 \sum_{l = 1}^3 A_j^{\operatorname{co}} \Phi_i^{-1} \partial_l v \,  \partial_j \phi_{i, l} \Big) + \Phi_i D  \,v \nonumber \\
  &= \Phi_i A_0 \, \partial_t v +  \sum_{l = 1}^3 \Big(\sum_{j = 1}^3 A_j^{\operatorname{co}} \Phi_i \partial_j \phi_{i,l}  \Big) \partial_l v + \Phi_i D\, v, 
\end{align}
for $v \in L^2(V_i)$, where $\phi_{i,l}$ denotes the $l$-th component of $\phi_i$ for all $i \in \N$. Extending 
the coefficients of $L^i$ from the bounded set $V_i$ to $\R^3_+$ appropriately, we obtain coefficients
$A_0^i \in \Fuwlk{m}{cp}{\eta}$ and $D^i \in \Fuwl{m}{cp}$ if $A_0$ and $D$ belong to 
$\Fdomuwlk{m}{cp}{\eta}$ respectively $\Fdomuwl{m}{cp}$. Moreover, the other coefficients $A_1^i$, 
$A_2^i$, and $A_3^i$ are contained in
\begin{align*}
 \Fcoeff{m}{cp} := \{&A \in \Fuwlk{m}{cp}{6} \colon \exists \mu_1, \mu_2, \mu_3 \in \Fuwlk{m}{cp}{1} 
 \text{ independent of time, } \\
 &\text{ constant outside of a compact set such that } A = \sum_{j = 1}^3 A_j^{\operatorname{co}} \mu_j\}.
\end{align*}
In fact, they even belong to $W^{m+1,\infty}(\R^3_+)$ as $G$ has a $C^{m+2}$-boundary. Exploiting 
this amount of regularity makes the spatial coefficients easier to treat. However, in order to 
streamline the assumptions in the results below, we treat them in the larger space $\Fcoeff{m}{cp}$.

Since the matrix $B$ is of rank $2$, we eliminate one row of the boundary condition in the 
localization procedure. The boundary condition is then given by a $2 \times 6$ matrix $B^i$. 
On the right-hand sides we obtain localized data $f^i$, $g^i$, and $u_0^i$.
We note that the localized coefficients and data can be estimated by their original counterparts 
in the corresponding norms with a constant only depending on the shape of $\partial G$.

Choosing the covering of $\partial G$ fine enough, there exist numbers $\tau > 0$  
and $k(i) \in \{1,2,3\}$ such that $|\mu^i_{3,k(i)}| \geq \tau$ on $\R^3_+$ for all $i \in \N$, 
see Lemma~5.1 in~\cite{SpitzDissertation}. 
Here we denote by $\mu_{3,j}^i$ the coefficients for which $A_3^i = \sum_{j = 1}^3 A_j^{\operatorname{co}} \mu_{3,j}^i$. 
This property allows us to transform for every $i \in \N$ 
the initial boundary value problem on the half-space into 
the form
\begin{equation}
  \label{LocalizedIBVP}
\left\{\begin{aligned}
   \tilde{A}_0^i \partial_t v + \sum_{j=1}^3 \tilde{A}_j^i \partial_j v + \tilde{D}^i v  &= \tilde{f}^i, \quad &&x \in \R^3_+, \quad &t \in J; \\
   \tilde{B}^i v &= \tilde{g}^i, \quad &&x \in \partial \R^3_+, &t \in J; \\
   v(0) &= \tilde{u}_0^i, \quad &&x \in \R^3_+;
\end{aligned}\right.
\end{equation}
with coefficients $\tilde{A}_0^i \in \Fuwlk{m}{cp}{\eta}$, $\tilde{A}^i_1, \tilde{A}^i_2 \in \Fcoeff{m}{cp}$, $\tilde{A}^i_3 = A_3^{\operatorname{co}}$, 
and $\tilde{B}^i = B^{\operatorname{co}}$, where $B^{\operatorname{co}}$ is a constant  
$2 \times 6$ matrix of rank $2$. (The actual form of $B^{\operatorname{co}}$ 
depends on $k(i)$. E.g., in the case $k(i) = 3$ we have $B^{\operatorname{co}} = (e_2, -e_1)^T$, 
where $e_1$ and $e_2$ are the correspondent unit vectors of $\R^6$.) For the transformed data we have $\tilde{f}^i \in \Ha{m}^6$, $\tilde{g}^i \in \E{m}^2$, and 
$\tilde{u}_0^i \in \Hh{m}^6$ for all $i \in \N$ if $f \in \Hadom{m}^6$, $g \in \Edom{m}^3$, 
and $u_0 \in \Hhdom{m}^6$. The localized and transformed data can be estimated by its original 
counterparts in the corresponding norms with a constant depending on suitable norms of 
the coefficients and the shape of $\partial G$.
The transform is given by
\begin{align*}
   &\tilde{A}_j^i = G_i^{T} A_j^i G_i, \quad \tilde{D}^i = G_i^{T} D^i G_i - \sum_{j = 1}^3 G_i^{T} A_j^i G_i \partial_j G_i^{-1} G_i,  
   \quad \tilde{B}^i = B^i G_i, \\
   &\tilde{f}^i = G_i^T f^i, \quad \tilde{g}^i = g^i, \quad \tilde{u}_0^i = G_i^{-1} u_0^i,
  \end{align*}
for $j \in \{0, \ldots, 3\}$, where
\begin{align*}
  G_i = \begin{pmatrix}
          \hat{G}_i &0 \\
          0 &\hat{G}_i
         \end{pmatrix},
 \quad \text{and} \quad 
 \hat{G}_i = \frac{1}{\sqrt{\mu_{3,3}^i}} \begin{pmatrix}
                               1 &0 &\mu_{1,3}^i \\
                               0 &1 &\mu_{2,3}^i \\
                               0 &0 &\mu_{3,3}^i
                              \end{pmatrix}
\end{align*}
in the case $k(i) = 3$ and $\mu_{3,3}^i \geq \tau$ on $\R^3_+$. In the other cases the matrix 
$\hat{G}_i$ has to be adapted accordingly.  A function $u^i$ then solves the untransformed problem 
if and only if $\tilde{u}^i = G_i^{-1} u^i$ solves the transformed one. Here we employ that 
the coefficients $A_j^i$ belong to $W^{m+1,\infty}(\R^3_+)$ for $j \in \{1,2,3\}$ in order 
to conclude that it is 
enough to study the transformed problem on the half-space in the following.
We refer to~\cite[Theorem~5.6~IV)]{SpitzDissertation} for the details of the localization 
procedure and the subsequent transform.

Finally, we point out that there is a constant $2 \times 6$ matrix $C^{\operatorname{co}}$ such that 
$A_3^{\operatorname{co}} = 1/2 (C^{\operatorname{co}  T} B^{\operatorname{co}} + B^{\operatorname{co}  T} C^{\operatorname{co}})$. 
We further note that $A_3^{\operatorname{co}}$ has exactly two positive and two negative eigenvalues counted with multiplicities.

\begin{rem}
 \label{RemarkAboutLocalization}
 \begin{enumerate}
  \item For the localization procedure we also have to treat the part $U_0$ of $G$ which is not 
	contained in the cover of $\partial G$. On $U_0$ a similar but simpler 
	procedure as described above leads to a full space problem for the localized solution.
	The results in Sections~\ref{SectionAPrioriEstimates} and~\ref{SectionRegularity}
	are also true on the full space and follow by the same strategy, which is 
	much easier in that case. See~\cite[Theorem~5.3]{SpitzDissertation} for a more detailed
	discussion.
  \item The proof that Theorem~\ref{TheoremMainResultOnDomain} indeed follows from the corresponding 
	result for problem~\eqref{LocalizedIBVP} is very technical but follows from standard ideas, 
	see~\cite[Theorem~5.6]{SpitzDissertation} once again.
 \end{enumerate}
\end{rem}


\section{A priori estimates}
\label{SectionAPrioriEstimates}
In the previous section we have reduced~\eqref{IBVP} to the initial boundary value problem
\begin{equation}
  \label{IBVP}
\left\{\begin{aligned}
   A_0 \partial_t u + \sum_{j=1}^3 A_j \partial_j u + D u  &= f, \quad &&x \in \R^3_+, \quad &t \in J; \\
   B u &= g, \quad &&x \in \partial \R^3_+, &t \in J; \\
   u(0) &= u_0, \quad &&x \in \R^3_+;
\end{aligned}\right.
\end{equation}
on the half-space with $A_3 = A_3^{\operatorname{co}}$ and $B = B^{\operatorname{co}}$. In this section we derive a priori estimates for $\G{m}$-solutions of~\eqref{IBVP}.
We note that $L(A_0, \ldots A_3, D) u = f \in \Ltwoa$ implies that the divergence of 
$((A_0 u)_k, \ldots, (A_3 u)_k)$ belongs to $\Ltwoa$ for $1 \leq k \leq 6$, $A_1, \ldots, A_3 \in W^{1,\infty}(\Omega)$,
and $u \in \Ltwoa$. Therefore, this 
vector has a normal trace
$((A_0 u)_k, \ldots, (A_3 u)_k) \cdot \nu$ in $H^{-1/2}(\partial \Omega)$, where 
$\nu$ denotes the normal outer unit of $\partial \Omega$. On $J \times \partial \R^3_+$ this 
product coincides with $- (A_3 u)_k$, which allows us to define a trace for $A_3 u$. One can find a constant 
matrix $M^{\operatorname{co}}$ such that $B^{\operatorname{co}} = M^{\operatorname{co}} A_3^{\operatorname{co}}$ so that we obtain a trace operator $\Tr$ for 
the function $B u$. We refer to Section~2.1 in~\cite{SpitzDissertation} for the details of 
this construction. By a (weak) solution of~\eqref{IBVP} we mean a function $u \in C(\clJ, \Ltwoh)$ 
with $L(A_0, \ldots, A_3, D) u = f$ in the weak sense, $\Tr(B u) = g$ on 
$J \times \partial \R^3_+$, and $u(0) = u_0$.

We first state the fundamental a priori estimate on $L^2$-level which was shown in Proposition~5.1
in~\cite{Eller}. The dependancies of the constants follow from the proof of this result in~\cite{Eller}.
\begin{lem}
 \label{LemmaEllerResult}
 Let $\eta > 0$ and $r \geq r_0 > 0$. Take $A_0 \in \Fupdwl{0}{\operatorname{cp}}{\eta}$, $A_1, A_2 \in \Fcoeff{0}{\operatorname{cp}}$ with $\|A_i\|_{W^{1,\infty}(\Omega)} \leq r$
 and $\|A_i(0)\|_{L^\infty(\R^3_+)} \leq r_0$ for all $i \in \{0, 1, 2\}$, and $A_3 = A_3^{\operatorname{co}}$. Let $D \in L^\infty(\Omega)$ with 
 $\|D\|_{L^\infty(\Omega)} \leq r$ and $B = B^{\operatorname{co}}$.
 Let $f \in \Ltwoa$, $g \in L^2(J, H^{1/2}(\partial \R^3_+))$, and $u_0 \in \Ltwoh$. Then~\eqref{IBVP} has a unique solution $u$ in $C(\clJ, \Ltwoh)$, 
 and there exists a number $\gamma_0 = \gamma_0(\eta,r) \geq 1$ such that
 \begin{align}
 \label{EquationEllerEstimateInL2}
 &\sup_{t \in J}\Ltwohn{e^{-\gamma t}u(t)}^2 + \gamma \Ltwoan{u}^2 \nonumber\\
 &\leq C_{0,0}\Ltwohn{u_0}^2 + C_{0,0} \|g\|_{L^2_\gamma(J, H^{1/2}(\partial \R^3_+))}^2 + C_0  \frac{1}{\gamma} \Ltwoan{f}^2
\end{align}
for all $\gamma \geq \gamma_0$, where $C_0 = C_0(\eta, r)$ 
and $C_{0,0} = C_{0,0}(\eta, r_0)$.
\end{lem}

We now start to derive the desired a priori estimates. In a first step, we give estimates for the  
tangential derivatives of a solution. The proof is classical but since we are interested in the 
particular structure of the constants, cf. the introduction, we provide the details.

We introduce the space $\Hata{m}$ which 
consists of those functions $v \in \Ltwoa$ with $\partial^\alpha v \in \Ltwoa$ for all 
$\alpha \in \N_0^4$ with $|\alpha| \leq m$ and $\alpha_3 = 0$. We equip this space with its natural norm.
\begin{lem}
\label{LemmaCentralEstimateInTangentialDirections}
Let $\eta > 0$ and  $r \geq r_0 > 0$. Pick $m \in \N$ and set $\tilde{m} = \max\{m,3\}$. 
Take $A_0 \in \Fupdwl{\tilde{m}}{\operatorname{cp}}{\eta}$, $A_1, A_2 \in \Fcoeff{\tilde{m}}{\operatorname{cp}}$, 
$A_3 = A_3^{\operatorname{co}}$, $D \in \Fuwl{\tilde{m}}{\operatorname{cp}}$, and $B = B^{\operatorname{co}}$ with
\begin{align*}
	&\Fnorm{\tilde{m}}{A_i} \leq r, \quad \Fnorm{\tilde{m}}{D} \leq r, \\ 
	&\max\{\Fvarnorm{\tilde{m}-1}{A_i(0)}, \max_{1 \leq j \leq m-1} \Hhn{\tilde{m}-1-j}{\partial_t^j A_0(0)}\} \leq r_0, \\
	&\max\{\Fvarnorm{\tilde{m}-1}{D(0)}, \max_{1 \leq j \leq m-1} \Hhn{\tilde{m}-1-j}{\partial_t^j D(0)}\} \leq r_0,
\end{align*}
for all $i \in \{0, 1, 2\}$.
Choose $f \in \Hata{m}$, $g \in \E{m}$, and $u_0 \in \Hh{m}$. 
Assume that the solution $u$ of~\eqref{IBVP} belongs to $\G{m}$. 
Then there exists $\gamma_m = \gamma_m(\eta, r) \geq 1$ such that
\begin{align}
 \label{EquationCentralEstimateInTangentialDirection}
  &\sum_{\substack{|\alpha| \leq m \\ \alpha_3 = 0}} \Gnorm{0}{\partial^\alpha u}^2  + \gamma \Hatan{m}{u}^2  \\
  &\leq  C_{m,0} \Big(\sum_{j = 0}^{m-1} \Hhn{m-1-j}{\partial_t^j f(0)}^2 + \Enorm{m}{g}^2 + \Hhn{m}{u_0}^2 \Big) \nonumber \\
	&\hspace{2em}  +   \frac{C_m}{\gamma} \Big( \Hatan{m}{f}^2 + \Gnorm{m}{u}^2 \Big), \nonumber
\end{align}
for all $\gamma \geq \gamma_0$, where $C_m = C_m(\eta, r, T')$,  
and $C_{m,0} = C_{m,0}(\eta, r_0)$.
\end{lem}
\begin{proof}
Let $\alpha \in \N_0^4$ with $|\alpha| \leq m$ and $\alpha_3 = 0$. It is straightforward to show 
that $\partial^\alpha u$ solves the initial boundary value problem
\begin{equation}
\label{EquationOneTimesDiff}
  \left\{\begin{aligned}
 L(A_0,\ldots, A_3, D) v &=  f_{\alpha},  \quad &&x \in \R^3_+, \quad &t \in J;\\
 B  v &= \partial^\alpha g,   &&x \in \partial \R^3_+, \quad &t \in J; \\
  v(0) &= u_{0,\alpha}, && x \in \R^3_+,
  \end{aligned}\right.
\end{equation}
where
\begin{align*}
  f_{\alpha} &= \partial^\alpha f - \sum_{j = 0}^2 \sum_{0 < \beta \leq \alpha} \binom{\alpha}{\beta} \partial^\beta A_j \partial^{\alpha - \beta} \partial_j u
	- \sum_{0 < \beta \leq \alpha} \binom{\alpha}{\beta} \partial^\beta D \partial^{\alpha - \beta} u, \\
  u_{0,\alpha} &= \partial^\alpha u(0) = \partial^{(0,\alpha_1, \alpha_2, 0)} S_{\R^3_+,m,\alpha_0}(0,A_0, \ldots, A_3, D, f, u_0),
\end{align*}
and where we employed that $A_3 = A_3^{\operatorname{co}}$ and $B = B^{\operatorname{co}}$, see Section~3.1 in~\cite{SpitzDissertation} for details.
We note that $f_\alpha$ is an element of $\Ha{m - |\alpha|}$ with
\begin{align}
\label{EquationEstimateForfalpha}
 \Ltwoan{f_\alpha} \leq \Hatan{m}{f} + C(r,T') \Gnorm{m}{u}
\end{align}
by Lemma~3.4 of~\cite{SpitzDissertation}, which uses Lemma~\ref{LemmaRegularityForA0} above. Lemma~\ref{LemmaEstimatesForHigherOrderInitialValues} 
further yields that $u_{0,\alpha}$ belongs to $\Hh{m - |\alpha|}$ and 
\begin{equation}
\label{EquationEstimateForu0alpha}
 \Hhn{m - |\alpha|}{u_{0,\alpha}} \leq C_{\ref{LemmaEstimatesForHigherOrderInitialValues}; m, |\alpha|}\Big(\sum_{k = 0}^{m-1} \Hhn{m-1-k}{\partial_t^k f(0)} + \Hhn{m}{u_0} \Big),
\end{equation}
where $C_{\ref{LemmaEstimatesForHigherOrderInitialValues}; m,|\alpha|} = C_{\ref{LemmaEstimatesForHigherOrderInitialValues}; m,|\alpha|}(\eta,r_0)$ is 
the constant from Lemma~\ref{LemmaEstimatesForHigherOrderInitialValues}.

Since $\partial^\alpha u$ solves the initial boundary value problem~\eqref{EquationOneTimesDiff}, we can apply 
estimate~\eqref{EquationEllerEstimateInL2} to $\partial^\alpha u$ and then insert estimates~\eqref{EquationEstimateForfalpha} 
and~\eqref{EquationEstimateForu0alpha} to deduce
\begin{align*}
 &\Gnorm{0}{\partial^\alpha u}^2 + \gamma \Ltwoan{\partial^\alpha u}^2 \\
 &\leq C_{0,0} \Ltwohn{u_{0,\alpha}}^2 + C_{0,0} \|\partial^\alpha g\|_{L^2_\gamma(J, H^{1/2}(\partial \R^3_+))}^2 + C_0 \frac{1}{\gamma} \Ltwoan{f_\alpha}^2  \\
 &\leq \tilde{C}_{m,0} \Big( \sum_{k = 0}^{m-1} \Hhn{m-1-k}{\partial_t^k f(0)}^2 + \Hhn{m}{u_0}^2 \Big) 
	+ \tilde{C}_{m,0} \Enorm{m}{g}^2 \\
  &\quad  + \tilde{C}_m \frac{1}{\gamma} \Big( \Gnorm{m}{u}^2 +  \Hatan{m}{f}^2 \Big)
\end{align*}
for all $\gamma \geq \gamma_{0}$. Here $\gamma_{0}(\eta,r) = \gamma_{\ref{LemmaEllerResult};0}(\eta,r)$ is 
the corresponding number from Lemma~\ref{LemmaEllerResult} and $\tilde{C}_{m,0} = \tilde{C}_{m,0}(\eta,r_0)$
and $\tilde{C}_m = \tilde{C}_m(\eta,r,T')$ denote constants with the 
described dependancies. Summing over all multiindices $\alpha \in \N_0^4$ with $\alpha_3 = 0$ and $|\alpha| \leq m$, we 
thus obtain the assertion.
\end{proof}


The above procedure only works in tangential directions because differentiation in the normal direction does 
not preserve the boundary condition. Since the boundary matrix $A_3$ is not invertible, we neither obtain
the normal derivative from the equation itself.
Instead, we will use the structure of the Maxwell equations to get an estimate for the normal derivative.
We consider the 
initial value problem
\begin{equation}
  \label{IVP}
 \left\{\begin{aligned}
   L(A_0,\ldots, A_3, D) u &= f, \qquad &&x \in \R^3_+, \qquad &&t \in J; \\
   u(0) &= u_0,  &&x \in \R^3_+.
 \end{aligned} \right.
\end{equation}
We define a solution of~\eqref{IVP} 
to be a function $u \in C(\clJ, \Ltwoh)$ with $u(0) = u_0$ in $\Ltwoh$ and $L u = f$ in $H^{-1}(\Omega)$.
In the iteration and regularization process it will be important that we do not impose a boundary condition 
in~\eqref{IVP} and the next result.

For the formulation of Proposition~\ref{PropositionCentralEstimateInNormalDirection} below we also need the following notion.
Take $A_1, A_2, A_3 \in \Fcoeff{0}{cp}$. The definition of this space then implies that there are functions 
$\mu_{lj} \in \Fuwlk{0}{cp}{1}$ such that
\begin{equation}
\label{EquationFixingmuljForCoefficients}
 A_j = \sum_{l = 1}^3 A_l^{\operatorname{co}} \mu_{lj} \quad \text{for all } j \in \{1, 2, 3\}.
\end{equation}
We set 
\begin{equation}
\label{EquationDefinitionOfTildeMu}
    \tilde{\mu} = \begin{pmatrix}
		  \mu &0 \\
		  0   &\mu
        \end{pmatrix},
 \end{equation}
 where $\mu$ denotes the $3 \times 3$-matrix $(\mu_{lj})_{lj}$, and we define
 \begin{equation}
 \label{EquationDefinitionOfGeneralizedDiv}
  \Div(A_1, A_2, A_3) h = \Big(\sum_{k = 1}^3 (\tilde{\mu}^T \nabla h)_{kk}, \sum_{k = 1}^3 (\tilde{\mu}^T \nabla h)_{(k+3) k} \Big)
 \end{equation}
 for all $h \in \Ltwoh$.

The next proposition is the key step in the derivation of the regularity theory for~\eqref{IBVPDomain}. 
It tells us that the derivative in normal direction can be controlled by the ones in tangential directions 
and the data although the problem is characteristic. In the complement of the kernel of $A_3^{\operatorname{co}}$ 
we can control $\partial_3 u$ via the equation. For the remaining components we exploit that the 
(generalized) divergence $\Div(A_1, A_2, A_3)$ of the (generalized) Maxwell operator $\sum_{j = 1}^3 A_j \partial_j u$ 
only contains first order derivatives of $u$. This cancellation property of the Maxwell system then 
allows us to apply a Gronwall argument.

We point out that we do not assume that $u$ belongs to $\G{1}$. For the derivative in normal 
direction it is enough to demand that it belongs to $L^\infty(J, \Ltwoh)$. While the gain of regularity 
which is thus contained in Proposition~\ref{PropositionCentralEstimateInNormalDirection} is just a byproduct of 
the proof here, the reduced regularity assumption is utterly important for the regularization procedure in 
Section~\ref{SectionRegularity}. Similarly, estimate~\eqref{EquationFirstOrderL2} with its 
less regular right-hand side is a significant tool in Section~\ref{SectionRegularity}.
\begin{prop}
\label{PropositionCentralEstimateInNormalDirection}
Let $T' > 0$, $\eta > 0$, $\gamma \geq 1$, and $r \geq r_0 >0$. 
Take $A_0 \in \Fupdwl{0}{\operatorname{cp}}{\eta}$, $A_1, A_2 \in \Fcoeff{0}{\operatorname{cp}}$, $A_3 = A_3^{\operatorname{co}}$,
and $D \in \Fuwl{0}{\operatorname{cp}}$ with 
\begin{align*}
	&\|A_i\|_{W^{1,\infty}(\Omega)} \leq r, \quad \|D\|_{W^{1,\infty}(\Omega)} \leq r, \\
	&\|A_i(0)\|_{L^\infty(\R^3_+)} \leq r_0,  \quad \|D(0)\|_{L^\infty(\R^3_+)} \leq r_0
\end{align*}
for all $i \in \{0, 1, 2\}$. 
Choose $f \in \G{0}$ with $\Div(A_1, A_2, A_3) f \in \Ltwoa$ and $u_0 \in \Hh{1}$.
Let $u$ solve~\eqref{IVP} with initial
value $u_0$ and inhomogeneity $f$. Assume that $u  \in C^1(\clJ, L^2(\R^3_+)) \cap C(\clJ, \Hhta{1}) \cap L^\infty(J,\Hh{1})$.
Then $u$ belongs to $\G{1}$ and there are constants $C_{1,0} = C_{1,0}(\eta,r_0) \geq 1$ and $C_1 = C_1(\eta, r, T') \geq 1$ 
such that
\begin{align}
 \label{EquationFirstOrderFinal}
    \Gnorm{0}{\nabla u}^2 &\leq e^{C_1 T}  \Big((C_{1,0} + T C_1)\Big(\sum_{j = 0}^2 \Ltwohnt{\partial_j u}^2 
      +   \Ltwohnt{f}^2 + \Hhn{1}{u_0}^2 \Big)\nonumber \\
    & \hspace{4em}  + \frac{C_1}{\gamma} \Ltwoan{\Div(A_1, A_2, A_3) f}^2  \Big).
\end{align}
If $f$ additionally belongs to $\Ha{1}$, we get
\begin{align}
 \label{EquationFirstOrderFinalVariant}
    \Gnorm{0}{\nabla u}^2 &\leq e^{C_1 T}  \Big(\!(C_{1,0} + T C_1)\Big(\sum_{j = 0}^2 \Ltwohnt{\partial_j u}^2  
      +   \Ltwohn{f(0)}^2 + \Hhn{1}{u_0}^2\Big)\nonumber \\
    & \hspace{4em}  + \frac{C_1}{\gamma} \Hangamma{1}{f}^2   \Big).
\end{align}
Finally, if $f$ merely belongs to $\Ltwoa$ with $\Div(A_1, A_2, A_3) f \in \Ltwoa$, we still have
\begin{align}
 \label{EquationFirstOrderL2}
 \Ltwoan{\nabla u}^2 &\leq e^{C_1 T} \Big((C_{1,0} + T C_1)  \Big(  \sum_{j = 0}^2 \Ltwoan{\partial_j u}^2  + \Ltwoan{f}^2 + \Hhn{1}{u_0}^2\Big) \nonumber\\
 &\hspace{4em} + \frac{C_1}{\gamma}  \Ltwoan{\Div(A_1, A_2, A_3) f}^2 \Big).
\end{align}
\end{prop}
\begin{proof}
For the assertion of the lemma it is enough to show that $\partial_3 u$ belongs to $C(\clJ, \Ltwoh)$ and 
that inequalities~\eqref{EquationFirstOrderFinal} to~\eqref{EquationFirstOrderL2} hold.

By the definition of the space $\Fcoeff{0}{\operatorname{cp}}$ 
there are functions $\mu_{lj} \in \Fuwlk{0}{\operatorname{cp}}{1}$ for $l,j \in \{1,2,3\}$
which satisfy~\eqref{EquationFixingmuljForCoefficients} and are bounded by $C(r)$ in $\Fk{0}{1}$.
Since $A_3 = A_3^{\operatorname{co}}$, we have $\mu_{13} = \mu_{23} = 0$ and $\mu_{33} = 1$.
Moreover, 
\begin{equation}
\label{EquationConstantMatricesAndLeviCivita}
 A_l^{\operatorname{co}} = \begin{pmatrix}
                            0 &-J_l \\
                            J_l &0
                           \end{pmatrix}
  \text{ with }  J_{l;mn} = - \epsilon_{lmn}
\end{equation}
for all $l, m, n \in \{1,2,3\}$, where $\epsilon _{lmn}$ denotes the Levi-Civita symbol, i.e.,
\begin{equation*}
	\epsilon_{ijk} =  \left\{\begin{aligned}
 &1  \quad &&\text{if } (i,j,k) \in \{(1,2,3), (2,3,1), (3,1,2)\},\\
 -&1   &&\text{if } (i,j,k) \in \{(3,2,1), (2,1,3), (1,3,2)\},\\
  &0 &&\text{else}.
  \end{aligned}\right.
\end{equation*}
We use the matrix $\tilde{\mu}$ from~\eqref{EquationDefinitionOfTildeMu}.
Since the coefficients are Lipschitz, 
we can take the weak time derivative of  $\tilde{\mu}^T A_0  \nabla u$  componentwise to obtain
\begin{align}
\label{EquationDifferentiatingMTransposedAzeroGradu}
  &\partial_t (\tilde{\mu}^T A_0  \nabla u) =  \tilde{\mu}^T \partial_t A_0 \nabla u + \tilde{\mu}^T A_0 \partial_t \nabla u \nonumber \\
  &=  \tilde{\mu}^T \partial_t A_0 \nabla u + \tilde{\mu}^T A_0 \nabla \Big(A_0^{-1} \Big(f - \sum_{j = 1}^3 A_j \partial_j u - D u \Big)\Big) \nonumber\\
  &=  \tilde{\mu}^T \partial_t A_0 \nabla u  + \tilde{\mu}^T A_0 \nabla A_0^{-1}\Big(f - \sum_{j = 1}^3 A_j \partial_j u - D u \Big)  \nonumber \\
  &\quad + \tilde{\mu}^T \nabla f - \tilde{\mu}^T \sum_{j = 1}^2 \nabla A_j \partial_j u - \tilde{\mu}^T \nabla D u - \tilde{\mu}^T D \nabla u - \tilde{\mu}^T \sum_{j = 1}^3 A_j \nabla \partial_j u
\end{align}
in $L^\infty(J, \Hh{-1})$, also employing~\eqref{IVP} and that
\begin{align*}
 ((\nabla A_0^{-1}) h)_{jk} := \sum_{l=1}^6 \partial_k A_{0;jl}^{-1} h_l
\end{align*}
 and analogously for $A_j$ with $j \in \{1, 2, 3\}$ and $D$.
We abbreviate
\begin{align}
\label{EquationDefinitionOfLambda}
 \Lambda &:= \tilde{\mu}^T \partial_t A_0 \nabla u  + \tilde{\mu}^T A_0 \nabla A_0^{-1}\Big(f - \sum_{j = 1}^3 A_j \partial_j u - D u\Big)  \nonumber \\
  &\qquad + \tilde{\mu}^T \nabla f - \tilde{\mu}^T \sum_{j = 1}^2 \nabla A_j \partial_j u - \tilde{\mu}^T \nabla D u - \tilde{\mu}^T D \nabla u
\end{align}
and note that this sum only contains first order spatial derivatives of $u$. 
We further compute
\begin{align*}
 &\sum_{k = 1}^3 \Big(\tilde{\mu}^T \sum_{j = 1}^3 A_j \nabla \partial_j u \Big)_{kk} 
  = \sum_{j,k = 1}^3 \sum_{l,p = 1}^6 \tilde{\mu}^T_{kl} A_{j;lp} \partial_k \partial_j u_p \\
  &= \sum_{j,k,n = 1}^3 \sum_{l,p = 1}^6 \tilde{\mu}^T_{kl} A_{n;lp}^{\operatorname{co}} \mu_{nj} \partial_k \partial_j u_p 
  = \sum_{j,k,l,n = 1}^3 \sum_{p = 1}^6 \mu_{lk} A_{n;lp}^{\operatorname{co}} \mu_{nj} \partial_k \partial_j u_p,
\end{align*}
using that $\tilde{\mu}_{lk} = 0$ for all $(l,k) \in \{4,5,6\} \times \{1,2,3\}$. Formula~\eqref{EquationConstantMatricesAndLeviCivita} thus leads to
\begin{equation}
\label{EquationTraceInvolvingLeviCivita}
 \sum_{k = 1}^3 \Big(\tilde{\mu}^T \sum_{j = 1}^3 A_j \nabla \partial_j u \Big)_{kk} = \sum_{j,k,l,n,p = 1}^3  \epsilon_{nlp} \mu_{lk}  \mu_{nj} \partial_k \partial_j u_{p+3}.
\end{equation}
Interchanging the indices $l$ and $n$ as well as $k$ and $j$, we arrive at
\begin{align}
 \label{EquationTraceInterchangingIndices}
 \sum_{k = 1}^3 \Big(\tilde{\mu}^T \sum_{j = 1}^3 A_j \nabla \partial_j u \Big)_{kk} &= \sum_{j,k,l,n,p = 1}^3  \epsilon_{lnp} \mu_{nj}  \mu_{lk} \partial_j \partial_k u_{p+3} \nonumber\\
 &= - \sum_{j,k,l,n,p = 1}^3  \epsilon_{nlp} \mu_{lk} \mu_{nj} \partial_k \partial_j u_{p+3}.
\end{align}
Equations~\eqref{EquationTraceInvolvingLeviCivita} and~\eqref{EquationTraceInterchangingIndices} 
thus yield
\begin{equation}
 \label{EquationTraceSmaller3EqualsZero}
 \sum_{k = 1}^3 \Big(\tilde{\mu}^T \sum_{j = 1}^3 A_j \nabla \partial_j u \Big)_{kk} = 0.
\end{equation}
Analogously, we derive
\begin{align}
 \label{EquationTraceLarger3EqualsZero}
 &\sum_{k = 1}^3 \Big(\tilde{\mu}^T \sum_{j = 1}^3 A_j \nabla \partial_j u \Big)_{(k+3)k} = 0.
\end{align}
In view of~\eqref{EquationDefinitionOfLambda}, equation~\eqref{EquationDifferentiatingMTransposedAzeroGradu} 
now implies that
\begin{align*}
 \sum_{k = 1}^3 \partial_t (\tilde{\mu}^T A_0  \nabla u)_{kk} &= \sum_{k = 1}^3 \Lambda_{kk}.
\end{align*}
An integration in $\Hh{-1}$ from $0$ to $t$ then leads to the identity
\begin{align*}
 \sum_{k = 1}^3 (\tilde{\mu}^T A_0  \nabla u)_{kk}(t) &= \sum_{k = 1}^3 (\tilde{\mu}^T A_0  \nabla u)_{kk}(0) +  \sum_{k = 1}^3 \int_0^t\Lambda_{kk}(s) ds
\end{align*}
for all $t \in \clJ$. The integrand on the right-hand side is also integrable with values in $\Ltwoh$, 
implying that the integral exists in $\Ltwoh$ and the equality holds in $\Ltwoh$ for all $t \in \clJ$. 
Starting from~\eqref{EquationTraceLarger3EqualsZero}, we obtain in the same way that
\begin{equation*}
 \sum_{k = 1}^3 (\tilde{\mu}^T A_0  \nabla u)_{(k+3)k}(t) = \sum_{k = 1}^3 (\tilde{\mu}^T A_0  \nabla u)_{(k+3)k}(0) +  \sum_{k = 1}^3 \int_0^t\Lambda_{(k+3)k}(s) ds
\end{equation*}
in $\Ltwoh$ for all $t \in \clJ$.
We denote the $k$-th row respectively the $k$-th column of a matrix 
$N$ by $N_{k \cdot}$ respectively $N_{\cdot k}$ and we set 
\begin{align}
\label{EquationDefinitionF7F8}
 F_7(t) &= \sum_{k = 1}^3 (\tilde{\mu}^T A_0  \nabla u)_{kk}(0) +  \sum_{k = 1}^3 \int_0^t\Lambda_{kk}(s) ds - \sum_{k = 1}^2 (\tilde{\mu}^T A_0)_{k \cdot} \partial_k u(t), \\
 F_8(t) &= \sum_{k = 1}^3 (\tilde{\mu}^T A_0  \nabla u)_{(k+3)k}(0) +  \sum_{k = 1}^3 \int_0^t\Lambda_{(k+3)k}(s) ds - \sum_{k = 1}^2 (\tilde{\mu}^T A_0)_{(k+3) \cdot} \partial_k u(t) \nonumber
\end{align}
for all $t \in \clJ$. Moreover, we define
\begin{align}
\label{EquationDefinitionF1ToF6}
 (F_1, \ldots, F_6)^T = f - \sum_{j = 0}^2 A_j \partial_j u - D u.
\end{align}
The function $F = (F_1, \ldots, F_8)^T$ then belongs to $C(\clJ, \Ltwoh)$. Introducing the matrix
\begin{equation*}
 \hat{\mu} = \begin{pmatrix}
            A_3 \\
            (\tilde{\mu}^T A_0)_{3 \cdot} \\
            (\tilde{\mu}^T A_0)_{6 \cdot}
           \end{pmatrix} \in \F{0}^{8 \times 6},
\end{equation*}
we obtain
\begin{equation}
\label{EquationFullRankEquationforNormalDerivativeOfu}
 \hat{\mu} \partial_3 u = F.
\end{equation}
We multiply $\hat{\mu}$ with the matrix
\begin{equation}
\label{EquationDefinitionOfG2}
 G_1 = \begin{pmatrix}
        1 &0 &0 &0 &0 &0 &0 &0 \\
        0 &-1 &0 &0 &0 &0 &0 &0 \\
        0 &0 &1 &0 &0 &0 &0 &0 \\
        0 &0 &0 &-1 &0 &0 &0 &0 \\
        0 &0 &0 &0 &1 &0 &0 &0 \\
        0 &0 &0 &0 &0 &1 &0 &0 \\
        -\tilde{\mu}^T_{3l} A_{0;l5} & \tilde{\mu}^T_{3l} A_{0; l4} &0 &\tilde{\mu}^T_{3l} A_{0;l2} &-\tilde{\mu}^T_{3l} A_{0;l1} &0 &1 &0 \\
        -\tilde{\mu}^T_{6l} A_{0;l5} & \tilde{\mu}^T_{6l} A_{0; l4} &0 &\tilde{\mu}^T_{6l} A_{0;l2} &-\tilde{\mu}^T_{6l} A_{0;l1} &0 &0 &1
       \end{pmatrix},
\end{equation}
where summation over the index $l$ (from $1$ to $6$) is implicitly assumed.
It follows
\begin{equation*}
 G_1 \hat{\mu} = \begin{pmatrix}
                    0 &0 &0 &0 &1 &0 \\
                    0 &0 &0 &1 &0 &0 \\
                    0 &0 &0 &0 &0 &0 \\
                    0 &1 &0 &0 &0 &0 \\
                    1 &0 &0 &0 &0 &0 \\
                    0 &0 &0 &0 &0 &0 \\
                    0 &0 &\alpha_{33} &0 &0 &\alpha_{36} \\
                    0 &0 &\alpha_{63} &0 &0 &\alpha_{66}
                   \end{pmatrix}
\end{equation*}
with the numbers
\begin{align*}
 \alpha_{kn} &=  \sum_{j = 1}^3 \sum_{l = 1}^6 \tilde{\mu}^T_{kl} A_{0;l(j+n-3)} \mu_{j3} = \tilde{\mu}^T_{k\cdot} A_{0} \tilde{\mu}_{\cdot n} = A_{0; kn}
\end{align*}
for all $k,n \in \{3,6\}$.
We conclude that
\begin{equation}
\label{EquationDefinitionMatrixAlpha}
 \begin{pmatrix}
        \alpha_{33} &\alpha_{36} \\
        \alpha_{63} &\alpha_{66}
       \end{pmatrix}
       = 
       \begin{pmatrix}
        A_{0;33} &A_{0;36} \\
        A_{0;63} &A_{0;66}
       \end{pmatrix}
       \geq \eta.
\end{equation}
Hence, it has an inverse $\beta$ satisfying
\begin{equation*}
 \|\beta\|_{L^\infty(\Omega)} \leq C(\eta).
\end{equation*}
Introducing the matrix
\begin{equation}
\label{EquationDefinitionOfG2AndG3}
 G_2 = \begin{pmatrix}
        I_{6 \times 6} &0 \\
        0 &\beta
       \end{pmatrix}, 
\end{equation}
we compute
\begin{equation}
\label{EquationMassMatrixAfterGaussAlgorithm}
  G_2 G_1 \hat{\mu} = \begin{pmatrix}
			      0 &0 &0 &0 &1 &0 \\
			      0 &0 &0 &1 &0 &0 \\
			      0 &0 &0 &0 &0 &0 \\
			      0 &1 &0 &0 &0 &0 \\
			      1 &0 &0 &0 &0 &0 \\
			      0 &0 &0 &0 &0 &0 \\
			      0 &0 &1 &0 &0 &0 \\
			      0 &0 &0 &0 &0 &1
                           \end{pmatrix} =: \tilde{M}.
\end{equation}
Using also~\eqref{EquationFixingmuljForCoefficients} and~\eqref{EquationDefinitionOfTildeMu}, we see that
\begin{equation*}
 \|G_2 G_1 \|_{L^\infty(\Omega)} \leq C(\eta) (1 + c_0)
\end{equation*}
with the constant
\begin{equation*}
  c_0 = \max\{\max_{j = 0, \ldots, 3} \|A_j \|_{L^\infty(\Omega)}, \|D\|_{L^\infty(\Omega)}\}.
\end{equation*}
Equation~\eqref{EquationFullRankEquationforNormalDerivativeOfu} and~\eqref{EquationMassMatrixAfterGaussAlgorithm} yield
\begin{equation}
\label{EquationFullRankMatrixMtildeInFrontOfPartial3u}
 \tilde{M} \partial_3 u = G_2 G_1 F.
\end{equation}
Since the matrices $G_i$ belong to $C(\overline{J}, L^\infty(\R^3_+))$ and $F$ is contained in $C(\overline{J}, \Ltwoh)$, 
we infer that $\partial_3 u$ is contained in $C(\overline{J}, \Ltwoh)$ and 
\begin{equation}
\label{EquationFirstEstimateForNormalDerivativeOfu}
 \Ltwohn{\partial_3 u(t)} \leq C(\eta)(1 + c_0) \Ltwohn{F(t)}
\end{equation}
for all $t \in \clJ$. To estimate $\Ltwohn{F(t)}$ we first note that
\begin{align}
\label{EquationEstimateForFInLtwoh}
 &\Ltwohn{F(t)} \leq \Ltwohn{(F_1, \ldots, F_6)^T(t)} + \Ltwohn{(F_7, F_8)^T(t)}  \\
 &\leq \Ltwohn{f(t)} + c_0 \sum_{j = 0}^2 \Ltwohn{\partial_j u(t)} + c_0 \Ltwohn{u(t)}  + \Ltwohn{(F_7, F_8)^T(t)} \nonumber
\end{align}
for all $t \in \clJ$. Applying Minkowski's inequality, we further deduce
\begin{align*}
  &\Ltwohn{(F_7, F_8)^T(t)} \leq C(r_0) \Hhn{1}{u_0} + c_0 \sum_{k = 1}^2 \Ltwohn{\partial_k u(t)} \\
  &+ C(\eta,r) \int_0^t (\Ltwohn{\nabla u(s)} + \Ltwohn{u(s)} + \Ltwohn{\Div f(s)} + \Ltwohn{f(s)}) ds \\
\end{align*}
for all $t \in \clJ$, where we abbreviate $\Div(A_1,A_2,A_3)$ by $\Div$.
This estimate,~\eqref{EquationFirstEstimateForNormalDerivativeOfu}, and~\eqref{EquationEstimateForFInLtwoh}, 
lead to the inequality
\begin{align}
\label{EquationPointwiseEstimateForGradient}
 &\Ltwohn{\nabla u(t)} \\
 &\leq C(\eta) (1 + c_0) \Big( \Ltwohn{f(t)} + c_0 \sum_{j = 0}^2 \Ltwohn{\partial_j u(t)} 
 + c_0 \Ltwohn{u(t)}  \nonumber\\
 &\quad + C(r_0) \Hhn{1}{u_0}  + C(\eta,r) \int_0^t (\Ltwohn{\nabla u(s)} + \Ltwohn{u(s)} \nonumber\\
 &\hspace{15em} + \Ltwohn{\Div f(s)} + \Ltwohn{f(s)}) ds\Big) \nonumber
\end{align}
for all $t \in \clJ$. Let $\gamma \geq 1$. Using H{\"o}lder's inequality, we infer
\begin{align*}
 &\Ltwohn{\nabla u(t)} \leq C(\eta) (1 + c_0) (1 + T C(\eta,r)) \Big( e^{\gamma t} \Gnorm{0}{f} + c_0 e^{\gamma t} \Gnorm{0}{u}  \\
  &\hspace{3em} + c_0 e^{\gamma t} \sum_{j = 0}^2 \Gnorm{0}{\partial_j u} 
 + C(r_0) \Hhn{1}{u_0} \Big) \\
 &\hspace{3em}  + C(\eta,r)\Big(\frac{1}{\sqrt{\gamma}} e^{\gamma t} \Ltwoan{\Div f}  +\int_0^t \Ltwohn{\nabla u(s)} ds\Big) \\
 &=: g(t) + C(\eta,r) \int_0^t \Ltwohn{\nabla u(s)} ds
\end{align*}
for all $t \in \clJ$.
 Since the function $g$
increases in $t$, Gronwall's inequality yields
\begin{align}
\label{EquationGronwall}
 \Gnorm{0}{\nabla u} &\leq \Big( C(\eta,r_0) (1 + c_0)^2 (1 + T C(\eta,r)) \Big( \Gnorm{0}{f} +  \Gnorm{0}{u} \nonumber\\
 &\!+  \sum_{j = 0}^2 \Gnorm{0}{\partial_j u}  + \Hhn{1}{u_0} \Big)  +  C(\eta,r)\frac{1}{\sqrt{\gamma}} \Ltwoan{\Div f} \Big) e^{C(\eta,r)T}.
\end{align}
Employing that $\partial_t A_0$ belongs to 
$L^\infty(\Omega)$, we further obtain 
\begin{align*} 
 \|A_0\|_{L^\infty(\Omega)}  \leq \|A_0(0)\|_{L^\infty(\R^3_+)} + T \|A_0\|_{W^{1,\infty}(\Omega)} \leq r_0 + T r.
\end{align*}
We argue analogously for the other coefficients, which yields $c_0 \leq r_0 + T r$.

To conclude~\eqref{EquationFirstOrderFinal}, we write $u$ as
\begin{align*}
  u(t) = u(0) + \int_0^t \partial_t u(s) ds
 \end{align*}
 in $\Ltwoh$ using that $u$ belongs to $C^1(\clJ, \Ltwoh)$. Minkowski's and H{\"o}lder's inequality then imply
 \begin{align}
 \label{EquationSplittingfInLowerOrderNorm}
  \Gnorm{0}{u}^2 \leq 2 \Ltwohn{u_0}^2 + \frac{1}{\gamma} T \Gnorm{0}{\partial_t u}^2.
 \end{align}
 Plugging this inequality into~\eqref{EquationGronwall}, assertion~\eqref{EquationFirstOrderFinal} follows.
 If $f$ additionally belongs to $\Ha{1}$, we argue as in~\eqref{EquationSplittingfInLowerOrderNorm} 
 for the function $f$ to derive~\eqref{EquationFirstOrderFinalVariant}.
 
 Now assume that $f$ only belongs to $\Ltwoa$ with $\Div f \in \Ltwoa$. Then estimate~\eqref{EquationPointwiseEstimateForGradient} 
 is still valid for almost all $t \in J$. We square~\eqref{EquationPointwiseEstimateForGradient}, 
 multiply with the exponential $e_{-2 \gamma}$, and integrate from $0$ to $t$. Applying Gronwall's 
 inequality to the function $t \mapsto \int_0^t e^{-2 \gamma s} \Ltwohn{\nabla u(s)}^2 ds$, 
 we deduce~\eqref{EquationFirstOrderL2} in the same way as we obtained~\eqref{EquationFirstOrderFinal}.
\end{proof}

Combining the a priori estimates in tangential and normal direction with an iteration argument, 
we obtain our first main result. It provides the desired a priori estimates of arbitrary order.
\begin{theorem}
 \label{TheoremAPrioriEstimates}
 Let $T' > 0$, $\eta > 0$, and $r \geq  r_0 >0$. Pick $T \in (0, T']$ and set $J = (0,T)$.
 Let $m \in \N$ and $\tilde{m} = \max\{m,3\}$. 
 Choose $A_0 \in \Fupdwl{\tilde{m}}{\operatorname{cp}}{\eta}$, $A_1, A_2 \in \Fcoeff{\tilde{m}}{\operatorname{cp}}$, 
 $A_3 = A_3^{\operatorname{co}}$,  
 $D \in \Fuwl{\tilde{m}}{\operatorname{cp}}$, and $B = B^{\operatorname{co}}$ with
 \begin{align*}
  &\Fnorm{\tilde{m}}{A_i} \leq r, \quad \Fnorm{\tilde{m}}{D} \leq r, \\
  &\max \{\Fvarnorm{\tilde{m}-1}{A_i(0)},\max_{1 \leq j \leq \tilde{m}-1} \Hhn{\tilde{m}-j-1}{\partial_t^j A_0(0)}\} \leq r_0, \\
  &\max \{\Fvarnorm{\tilde{m}-1}{D(0)},\max_{1 \leq j \leq \tilde{m}-1} \Hhn{\tilde{m}-j-1}{\partial_t^j D(0)}\} \leq r_0
 \end{align*}
 for all $i \in \{0, \ldots, 2\}$.
 Let $f \in \Ha{m}$, $g \in \E{m}$, and $u_0 \in \Hh{m}$. Assume that the solution $u$ 
 of~\eqref{IBVP}  belongs to $\G{m}$. Then there is a number
 $\gamma_m = \gamma_m(\eta, r, T') \geq 1$ such that
 \begin{align}
  &\Gnorm{m}{u}^2  \leq (C_{m,0} + T C_m) e^{m C_1 T} \Big(  \sum_{j = 0}^{m-1} \Hhn{m-1-j}{\partial_t^j f(0)}^2 + \Enorm{m}{g}^2  \nonumber\\
      &\hspace{18em} + \Hhn{m}{u_0}^2 \Big) + \frac{C_m}{\gamma}  \Hangamma{m}{f}^2    \nonumber
 \end{align}
 for all $\gamma \geq \gamma_m$, where $C_m = C_m(\eta, r,T') \geq 1$, $C_{m,0} = C_{m,0}(\eta,r_0) \geq 1$, and 
 $C_1 = C_1(\eta,r,T')$ is a constant independent of $m$.
\end{theorem}
\begin{proof}
  We prove the assertion by induction with respect to $m$. To this purpose we observe that 
 combining Lemma~\ref{LemmaEllerResult}, Lemma~\ref{LemmaCentralEstimateInTangentialDirections}, and 
 Proposition~\ref{PropositionCentralEstimateInNormalDirection}, and choosing $\gamma_1 = \gamma_1(\eta, r, T')$ large enough, 
 we obtain the assertion for $m = 1$. Next assume that $m \geq 2$ and that the assertion has been shown for $m-1$. 
 
 We now take $u$, data, and coefficients as in the statement. Let $p \in \{0,1,2\}$. As in~\eqref{EquationOneTimesDiff} we deduce that $\partial_p u$ solves~\eqref{IBVP} with differential 
 operator $L(A_0,\ldots, A_3, D)$, inhomogeneity $f_{1,p}$, boundary value $\partial_p g$, and initial value $\partial_p u_0$, where
 \begin{align}
 \label{EquationDefinitionf1p}
  f_{1,p} &= \partial_p f - \sum_{i = 0}^2 \partial_p A_i \partial_i u - \partial_p D u, \\
  \partial_0 u_0 &= S_{m,1}(0,A_0,\ldots, A_3, D,f,u_0). \nonumber
 \end{align}
 Note that $f_{1,p}$ belongs to $\Ha{m-1}$ by Lemma~\ref{LemmaRegularityForA0}.
 Since $A_0 \in \Fupdwl{\tilde{m}}{\operatorname{cp}}{\eta}$, 
 $A_1, A_2 \in \Fcoeff{\tilde{m}}{cp}$,
 $D \in \Fuwl{\tilde{m}}{\operatorname{cp}}$, $f \in \Ha{m}$, and $u_0 \in \Hh{m}$, 
 Lemma~\ref{LemmaEstimatesForHigherOrderInitialValues} yields that $S_{m,1}(0,A_0,\ldots, A_3,D,f,u_0)$ is contained in 
 $\Hh{m-1}$. The induction hypothesis therefore gives
 \begin{align}
 \label{EquationEstimatingPartialPu}
  &\Gnorm{m-1}{\partial_p u}^2  \leq (C_{m-1,0} + T C_{m-1}) e^{(m-1) C_1 T} \Big( \sum_{j = 0}^{m-2}\Hhn{m-2-j}{\partial_t^j f_{1,p}(0)}^2  \nonumber\\
	&\hspace{1em} + \Enorm{m-1}{\partial_p g}^2 + \Hhn{m-1}{\partial_p u_0}^2 \Big) + \frac{C_{m-1}}{\gamma}  \Hangamma{m-1}{f_{1,p}}^2    
 \end{align}
 for all $\gamma \geq \gamma_{m-1}$. 
 We next estimate the terms appearing on the right-hand side of~\eqref{EquationEstimatingPartialPu}. To that purpose, 
 let $j \in \{0,\ldots,m-2\}$. We observe that
 \begin{align*}
  \partial_t^j (\partial_p A_0 \partial_t u)(0) = \sum_{l = 0}^j \binom{j}{l} \partial_t^l \partial_p A_0(0) S_{m,j+1-l}(0,A_0,\ldots, A_3, D,f,u_0).
 \end{align*}
 Since the function $\partial_t^l \partial_p A_0(0)$ belongs to $\Hh{\tilde{m}-l-2}$ and $S_{m,j+1-l}$ is contained in $\Hh{m-j+l-1}$ 
 by Lemma~\ref{LemmaEstimatesForHigherOrderInitialValues}, Lemma~\ref{LemmaRegularityForA0}~\ref{ItemProductInL2} 
 in the case $j = m-2$ and Lemma~\ref{LemmaRegularityForA0}~\ref{ItemProductIn0} in the case $j < m-2$ show
 \begin{align*}
  &\Hhn{m-2-j}{\partial_t^l \partial_p A_0(0) S_{m,j+1-l}(0,A_0,\ldots, A_3,D,f,u_0)} \\
  &\leq C_{\ref{LemmaEstimatesForHigherOrderInitialValues}; m,j-l+1}(\eta,r_0) r_0 \Big(\sum_{k=0}^{j - l} \Hhn{m-1-k}{\partial_t^k f(0)} + \Hhn{m}{u_0} \Big),
 \end{align*}
 where we also applied Lemma~\ref{LemmaEstimatesForHigherOrderInitialValues}. Arguing analogously,
 for $A_1$, $A_2$, and $D$, we arrive at
 \begin{align}
  \label{EquationEstimateOffInZero}
  \Hhn{m-2-j}{\partial_t^j f_{1,p}(0)} \leq C(\eta,r_0) \Big(\sum_{k=0}^{m-1} \Hhn{m-1-k}{\partial_t^k f(0)} + \Hhn{m}{u_0} \Big).
 \end{align}
 Lemma~\ref{LemmaRegularityForA0} further yields
 \begin{equation}
 \label{EquationEstimatefInHanmminus1}
  \Hangamma{m-1}{f_{1,p}} \leq \Hangamma{m}{f} + C(m,r) \Gnorm{m}{u}
 \end{equation}
 for all $\gamma > 0$.

We insert the estimates~\eqref{EquationEstimateOffInZero} and~\eqref{EquationEstimatefInHanmminus1} 
into~\eqref{EquationEstimatingPartialPu} and combine it with the induction hypothesis and 
Lemma~\ref{LemmaEstimatesForHigherOrderInitialValues} in the case $p = 0$ to infer
\begin{align}
\label{EquationEstimateForTangentialDerivativesInGmminusone}
 &\Gnorm{m-1}{u}^2 + \sum_{p = 0}^2 \Gnorm{m-1}{\partial_p u}^2  \\
 &\leq (\tilde{C}_{m,0} + T \tilde{C}_{m})e^{(m-1) C_1 T} \Big(\sum_{k=0}^{m-1} \Hhn{m-1-k}{\partial_t^k f(0)}^2 + \Hhn{m}{u_0}^2 \nonumber\\
 &\hspace{2em} + \Enorm{m}{g}^2 \Big)  + \frac{\tilde{C}_{m}}{\gamma} \Big(\Hangamma{m}{f}^2 + \Gnorm{m}{u}^2 \Big) \nonumber
\end{align}
for all $\gamma \geq \gamma_{m-1}$, where $\tilde{C}_{m,0} = \tilde{C}_{m,0}(\eta,  r_0)$ and $\tilde{C}_{m} = \tilde{C}_{m}(\eta,  r, T')$ denote constants which may change from line to line.

 It only remains to control the $\Ggamma{0}$-norm of $\partial_3^m u$. To this purpose, we note 
 that $\partial_3^{m-1} u \in  \G{1}$ solves the initial value problem
 \begin{equation}
 \label{EquationIVPForPartial3mminus1u}
 \left\{\begin{aligned}
   L v &= f_{m,3}, &\quad &x \in \R^3_+, & \quad &t \in J; \\
   v(0) &= \partial_3^{m-1} u_0, & &x \in \R^3_+;& &
\end{aligned}\right.
\end{equation}
where 
\begin{equation*} 
 f_{m,3} = \partial_3^{m-1} f - \sum_{0 < j \leq m-1} \binom{m-1}{j} \Big(\sum_{k = 0}^2 \partial_3^j A_k \, \partial_k \partial_3^{m-1-j} u
    + \partial_3^j D \, \partial_3^{m-1-j} u\Big).
\end{equation*} 
Lemma~\ref{LemmaRegularityForA0}~\ref{ItemProductHkappa} implies that
 $f_{m,3}$ belongs to $\Ha{1}$ with
\begin{align}
\label{EquationEstimateInH1Normfm3}
 \Hangamma{1}{f_{m,3}}^2 &\leq C(r,T') \Big(\Hangamma{m}{f}^2 + \Gnorm{m}{u}^2\Big),
\end{align}
for all $\gamma > 0$.
As $\partial_3^j A_k(0)$ is an element of $\Hh{\tilde{m}-1-j}$ and $\partial_3^{m-1-j} \partial_k u(0)$ of 
$\Hh{j}$, Lemma~\ref{LemmaRegularityForA0}~\ref{ItemProductInL2} yields
\begin{align*}
 &\Ltwohn{\partial_3^j A_k(0) \partial_3^{m-1-j} \partial_k u(0)} \leq C(\eta,r_0) (\Hhn{m-1}{f(0)} + \Hhn{m}{u_0})
\end{align*}
for all $k \in \{0, \ldots, 2\}$ and $j \in \{1, \ldots, m-1\}$, where we employed 
Lemma~\ref{LemmaEstimatesForHigherOrderInitialValues} for the case $k = 0$. 
The same arguments applied to $D$ allow us to conclude
\begin{align}
\label{EquationEstimatefm3InZero}
 \Ltwohn{f_{m,3}(0)}^2 \leq C(\eta,r_0) \Big(\Hhn{m-1}{f(0)}^2 + \Hhn{m}{u_0}^2\Big).
\end{align}

On the other hand Proposition~\ref{PropositionCentralEstimateInNormalDirection} applied to~\eqref{EquationIVPForPartial3mminus1u} tells us that
\begin{align*}
 &\Gnorm{0}{\partial_3^{m} u}^2 \leq    (C_{1,0} + T C_1) e^{C_1 T}
      \Big(\sum_{j=0}^2 \Gnorm{0}{\partial_j \partial_3^{m-1} u}^2  + \Ltwohn{f_{m,3}(0)}^2  \\
    &\hspace{16em} + \Hhn{1}{\partial_3^{m-1} u_0}^2\Big) + \frac{C_1}{\gamma} e^{C_1 T} \Hangamma{1}{f_{m,3}}^2 
\end{align*}
for all $\gamma \geq 1$. Combined with~\eqref{EquationEstimateInH1Normfm3} and~\eqref{EquationEstimatefm3InZero} 
the above inequality implies
\begin{align}
\label{EquationEstimateForPartial3mtimesu}
    &\Gnorm{0}{\partial_3^m u}^2 \leq  (\tilde{C}_{m,0} + T \tilde{C}_m)e^{C_1 T} \Big(\sum_{j=0}^2 \Gnorm{m-1}{\partial_j u}^2  
     +  \Hhn{m-1}{f(0)}^2 \nonumber\\ 
    &\hspace{10em} + \Hhn{m}{u_0}^2 \Big) 
    +  \frac{\tilde{C}_m}{\gamma} \Big(\Hangamma{m}{f}^2 + \Gnorm{m}{u}^2 \Big) 
\end{align}
for all $\gamma \geq 1$. We then use~\eqref{EquationEstimateForPartial3mtimesu} to estimate
\begin{align*}
 &\Gnorm{m}{u}^2 \leq \Gnorm{m-1}{u}^2 + \sum_{j = 0}^2 \Gnorm{m-1}{\partial_j u}^2 + \Gnorm{0}{\partial_3^{m} u}^2 \\
	&\leq  (\tilde{C}_{m,0} + T \tilde{C}_m)e^{C_1 T}  \Big(\sum_{j=0}^2 \Gnorm{m-1}{\partial_j u}^2 + \Gnorm{m-1}{u}^2 \Big) + \frac{\tilde{C}_m}{\gamma} \Gnorm{m}{u}^2 \nonumber \\
    &\quad  +  (\tilde{C}_{m,0} + T \tilde{C}_m)e^{C_1 T} \Big( \Hhn{m-1}{f(0)}^2 + \Hhn{m}{u_0}^2 \Big) +  \frac{\tilde{C}_m}{\gamma} \Hangamma{m}{f}^2 
\end{align*}
for all $\gamma \geq 1$. Together with~\eqref{EquationEstimateForTangentialDerivativesInGmminusone}, it follows
\begin{align}
\label{EquationEstimateForuInGnormmAlmostFinal}
 \Gnorm{m}{u}^2 &\leq (\tilde{C}_{m,0} + T \tilde{C}_{m})e^{m C_1 T} \Big(\sum_{k=0}^{m-1} \Hhn{m-1-k}{\partial_t^k f(0)}^2 + \Enorm{m}{g}^2 \nonumber\\
 &\hspace{2em}  + \Hhn{m}{u_0}^2 \Big) + \frac{\tilde{C}_{m}}{\gamma} \Big( \Hangamma{m}{f}^2 + \Gnorm{m}{u}^2\Big)
\end{align}
for all $\gamma  \geq \gamma_{m-1}$. Choosing $\gamma_m = \gamma_m(\eta,r,T')$ large enough, the assertion follows.
\end{proof}


\section{Regularity of solutions}
\label{SectionRegularity}

In order to prove that the solution of~\eqref{IBVP} belongs to $\G{m}$ if the data and the 
coefficients are accordingly smooth and compatible, we have to apply different
regularizing techniques 
in normal, tangential, and time direction. We start by showing that regularity in time and 
in tangential directions implies regularity in normal direction. The main difficulty here is 
to avoid a loss of regularity across the boundary of $\partial \R^3_+$.
\begin{lem}
 \label{LemmaRegularityInNormalDirection}
Let $\eta > 0$, $m \in \N$, and $\tilde{m} =  \max\{m,3\}$. Take $A_0 \in \Fupdwl{\tilde{m}}{cp}{\eta}$, 
$A_1, A_2 \in \Fcoeff{\tilde{m}}{\operatorname{cp}}$, $A_3 = A_3^{\operatorname{co}}$,
and $D \in \Fuwl{\tilde{m}}{cp}$. Pick $f \in \Ha{m}$, 
and $u_0 \in \Hh{m}$. 
Let $u$ be a solution of the linear initial value problem~\eqref{IVP} with differential operator $L = L(A_0,\ldots, A_3, D)$, inhomogeneity 
$f$, and initial value $u_0$. Assume that $u$ belongs to $\bigcap_{j = 1}^m C^j(\clJ, \Hh{m-j})$.

Take $k \in \{1, \ldots, m\}$ and a multi-index $\alpha \in \N_0^4$ with $|\alpha| = m$, $\alpha_0 = 0$,
and $\alpha_3 = k$. 
Suppose that $\partial^\beta u$ is contained in $\G{0}$ for all $\beta \in \N_0^4$ with $|\beta| = m$ and 
$\beta_3 \leq k-1$.
Then $\partial^\alpha u$ is an element of $\G{0}$.
\end{lem}
\begin{proof} 
 I) We have to start with several preparations.  
 Let $\rho \in C_c^\infty(\R^3)$ be a nonnegative function with $\int_{\R^3} \rho(x) dx = 1$ and $\supp \rho \subseteq B(0,1)$.
 We denote the convolution operator with kernel
 $\rho_{\epsilon} = \epsilon^{-3} \rho(\epsilon^{-1} \cdot)$ by $M_\epsilon$ for all $\epsilon > 0$, where 
 the convolution is taken over $\R^3$. We further define the translation operator
 \begin{align}
 \label{EquationDefinitionTranslation}
  T_\tau v(x) = v(x_1,x_2, x_3 + \tau)
 \end{align}
 for all $v \in L^1_{\operatorname{loc}}(\R^3_+)$, $\tau > 0$, and for almost all $x \in \R^2 \times (-\tau, \infty)$.
 Clearly, $T_\tau$ maps $W^{l,p}(\R^3_+)$ continuously into 
 $W^{l,p}(\R^2 \times (-\tau, \infty))$ and $\partial^{\tilde{\alpha}} T_\tau v = T_\tau \partial^{\tilde{\alpha}} v$
 for all $v \in W^{l,p}(\R^3_+)$, $\tilde{\alpha} \in \N_0^4$ with $|\tilde{\alpha}| \leq l$, $l \in \N_0$, $1 \leq p \leq \infty$,
 and $\tau > 0$. 
 If $v \in L^1_{\operatorname{loc}}(\R^3)$, we further define $T_{\tau} v$ by formula~\eqref{EquationDefinitionTranslation} 
 for all $\tau \in \R$.
 
 Functions which are only defined on a subset of $\R^3$ will be identified with their zero-extensions 
 in the following.
 We extend the translations $T_\delta$ to continuous operators on $H^{-1}(\R^3_+)$ by 
 setting
 \begin{align*}
    &\langle T_\delta v, \psi \rangle_{H^{-1}(\R^3_+) \times H^1_0(\R^3_+)} = \langle v, T_{-\delta} \psi \rangle_{H^{-1}(\R^3_+) \times H^1_0(\R^3_+)} 
 \end{align*}
 for all  $\psi \in H^1_0(\R^3_+)$ and $\delta > 0$. Since partial derivatives commute with $T_{-\delta}$ and 
 the zero extension on $H^1_0(\R^3_+)$, we thus also have
 \begin{align}
 \label{EquationDerivativeAndTranslationCommuteOnHMinusOne}
  \partial_j T_\delta v  = T_\delta \partial_j v
 \end{align}
 for all $v \in \Ltwoh$ and $\delta > 0$.
 
 We next take a closer look on the convolution operator $M_\epsilon$, which is defined for functions in 
 $L^1_{\operatorname{loc}}(\R^3)$. We want to extend this operator in a sense to functions in 
 $L^1_{\operatorname{loc}}(\R^3_+)$ without obtaining singularities at the boundary.
 To that purpose,
 take $0 < \epsilon < \delta$. For functions $v$ in $L^1_{\operatorname{loc}}(\R^3_+)$ 
 we will employ the regularization $M_\epsilon T_\delta v$ and restrict it to $\R^3_+$.
 As it will be clear from the context on which domain we consider $M_\epsilon T_\delta v$, we will not 
 write this restriction explicitly.
 It is easy to see that if $v$ has a weak derivative in $\R^3_+$, then also $M_\epsilon T_\delta v$ has a weak derivative
 in $\R^3_+$ and
 \begin{align*}
  \partial_j M_\epsilon T_\delta  v = M_\epsilon T_\delta  \partial_j v
 \end{align*}
 for all $j \in \{1,2,3\}$.
 
 We define $\tilde{\rho}$ by $\tilde{\rho}(x) = \rho(-x)$ for all $x \in \R^3$. The 
 convolution operator with kernel $\tilde{\rho}_\epsilon$ is denoted by $\tilde{M}_\epsilon$ for all $\epsilon > 0$.
 Fix $0 < \epsilon < \delta$. A straightforward computation shows that
 \begin{align}
    \label{DefinitionConvolutionOnHMinusOneDomain}
 &\langle M_\epsilon T_\delta  v, \psi \rangle_{H^{-1}(\R^3_+) \times H^1_0(\R^3_+)} 
 = \langle  v, T_{-\delta} \tilde{M}_\epsilon  \psi \rangle_{H^{-1}(\R^3_+) \times H^1_0(\R^3_+)}
\end{align}
for all $ v \in \Ltwoh$ and $\psi \in H^1_0(\R^3_+)$.
 As $T_{-\delta} \tilde{M}_\epsilon$ maps $H^1_0(\R^3_+)$ continuously into itself, the mapping
 $M_\epsilon T_\delta$ continuously extends to 
 an operator on $H^{-1}(\R^3_+)$ via formula~\eqref{DefinitionConvolutionOnHMinusOneDomain}.
  We deduce the identity
\begin{align*}
 \partial_j M_\epsilon T_\delta v = M_\epsilon \partial_j T_\delta v =  M_\epsilon T_\delta  \partial_j v
\end{align*}
by duality for all $j \in \{1,2,3\}$ and $v \in \Ltwoh$ using that the partial derivative commutes with 
$T_{-\delta}$, $\tilde{M}_\epsilon$, and the zero extension on $H^1_0(\R^3_+)$.
We further note that for $A \in W^{1,\infty}(\R^3_+)$ and $v \in \Hh{-1}$ we have
\begin{align}
  \label{EquationTranslationAndCoefficient}
  (T_\delta A) T_\delta v =  T_\delta (A v)
\end{align}
in $\Hh{-1}$.

II) Let $0 < \epsilon < \delta$. In the following, we abbreviate the differential operator 
$L(T_\delta A_0, \ldots, T_\delta A_3, T_\delta D)$ by $L_\delta$ and $\Div(T_\delta A_1, T_\delta A_2, T_\delta A_3)$ by 
$\Div_\delta$. (Recall~\eqref{EquationDefinitionOfGeneralizedDiv}.) We define $\alpha' = \alpha - e_3 \in \N_0^4$ and note that 
$|\alpha'| = m-1$ and $\alpha'_3 = k-1$. In particular, $\partial^{\alpha'} u$ belongs to $\G{0}$ by assumption. 
Due to the mollifier the function $M_\epsilon T_\delta \partial^{\alpha'} u$ is contained in 
$C^1(\clJ, \Hh{1}) \hookrightarrow \G{1}$, $M_\epsilon T_\delta \partial^{\alpha'} u_0$ is an 
element of $\Hh{1}$,  
$L_\delta M_\epsilon T_\delta \partial^{\alpha'} u$ is contained in $\G{0}$, 
and $\Div_\delta L_\delta M_\epsilon T_\delta \partial^{\alpha'} u$ belongs to $\Ltwoa$. 
We want to apply 
estimate~\eqref{EquationFirstOrderFinal} from Proposition~\ref{PropositionCentralEstimateInNormalDirection} with 
differential operator $L_\delta$ 
to differences of functions $M_\epsilon T_\delta \partial^{\alpha'} u$ and show that they form 
a Cauchy sequence in $\Hh{1}$ as $\epsilon$ tends to $0$.
Therefore, we have to study the convergence properties of $L_\delta M_\epsilon T_\delta \partial^{\alpha'} u$ 
and $\Div_\delta L_\delta M_\epsilon T_\delta \partial^{\alpha'} u$ as $\epsilon \rightarrow 0$. 
We focus on the latter as this is the more difficult term.

We fix the functions $\mu_{lj} \in \Fuwlk{\tilde{m}}{cp}{1}$ (independent of time) with
\begin{align*}
 A_j = \sum_{l = 1}^3 A_l^{\operatorname{co}} \mu_{lj}
\end{align*}
for all $j \in \{1, 2, 3\}$ which exist by the definition of $\Fcoeff{\tilde{m}}{cp}$. We set 
 \begin{equation*}
    \tilde{\mu} = \begin{pmatrix}
		  \mu &0 \\
		  0   &\mu
        \end{pmatrix}
 \end{equation*}
  and compute
 \begin{align}
  \label{EquationPreparationToComputeDivOfLMepsu}
  &(T_\delta \tilde{\mu})^T \nabla L_\delta M_\epsilon T_\delta \partial^{\alpha'} u \\
  &= \sum_{j = 0}^2 (T_\delta \tilde{\mu})^T (T_\delta \nabla A_j) \partial_j M_{\epsilon} T_\delta \partial^{\alpha'} u + (T_\delta \tilde{\mu})^T (T_\delta \nabla D) M_\epsilon T_\delta \partial^{\alpha'} u   \nonumber\\
  &\quad \! + T_\delta (\tilde{\mu}^T A_0) \nabla M_\epsilon T_\delta \partial_t \partial^{\alpha'} u + T_\delta (\tilde{\mu}^T D) \nabla M_\epsilon T_\delta \partial^{\alpha'} u  + \! \sum_{j = 1}^3 T_\delta(\tilde{\mu}^T A_j) \nabla \partial_j M_{\epsilon} T_\delta \partial^{\alpha'} u \nonumber \\
  &=: \Lambda^{\delta,\epsilon} + \sum_{j = 1}^3 T_\delta(\tilde{\mu}^T A_j) \nabla \partial_j M_{\epsilon} T_\delta \partial^{\alpha'} u, \nonumber
 \end{align}
 where we exploited the results from step I). The cancellation properties 
 of the $L_\delta$-operator established in~\eqref{EquationTraceSmaller3EqualsZero} and~\eqref{EquationTraceLarger3EqualsZero} show that
 \begin{align*}
 	&\sum_{j = 1}^3  \Big(\sum_{k = 1}^3 (T_\delta(\tilde{\mu}^T A_j) \nabla \partial_j M_{\epsilon} T_\delta \partial^{\alpha'} u)_{kk}, \sum_{k = 1}^3 (T_\delta(\tilde{\mu}^T A_j) \nabla \partial_j M_{\epsilon} T_\delta \partial^{\alpha'} u)_{(k+3)k} \Big) = 0.
 \end{align*}
 We thus obtain from~\eqref{EquationPreparationToComputeDivOfLMepsu} that
 \begin{equation}
 	\label{EquationEstimateForDivFirstResult}
 	\Div_\delta  L_\delta M_\epsilon T_\delta \partial^{\alpha'} u = \Big(\sum_{k = 1}^3 \Lambda_{kk}^{\delta,\epsilon}, \sum_{k = 1}^3 \Lambda_{(k+3)k}^{\delta,\epsilon} \Big).
 \end{equation}
 We rewrite $\Lambda^{\delta, \epsilon}$ in the form
 \begin{align*}
 	\Lambda^{\delta, \epsilon} &= \sum_{j = 0}^2 [T_\delta (\tilde{\mu}^T \nabla A_j), M_\epsilon] \partial_j T_\delta \partial^{\alpha'} u  + [T_\delta(\tilde{\mu}^T \nabla D), M_\epsilon] T_\delta \partial^{\alpha'} u \\
 	&\quad  + [T_\delta (\tilde{\mu}^T A_0), M_\epsilon] \nabla  T_\delta \partial_t \partial^{\alpha'} u  + [T_\delta (\tilde{\mu}^T D), M_\epsilon] \nabla T_\delta \partial^{\alpha'} u\nonumber \\
 	&\quad + M_\epsilon T_\delta \Big( \sum_{j=0}^2 \tilde{\mu}^T \nabla A_j \partial_j \partial^{\alpha'} u + \tilde{\mu}^T \nabla D \partial^{\alpha'} u + \tilde{\mu}^T A_0 \nabla \partial_t \partial^{\alpha'} u + \tilde{\mu}^T D \nabla \partial^{\alpha'} u \Big),
 \end{align*}
 and introduce the function
 \begin{align*}
  \tilde{f}_{\alpha'} &= \sum_{0 < \beta \leq \alpha'} \binom{\alpha'}{\beta} \partial^{\beta}  (\tilde{\mu}^T A_0) \nabla \partial^{\alpha' - \beta} \partial_t u
			  + \sum_{0 < \beta \leq \alpha'} \binom{\alpha'}{\beta} \partial^{\beta}  (\tilde{\mu}^T D) \nabla \partial^{\alpha' - \beta} u \\
		      &\quad + \sum_{j = 0}^2 \sum_{0 < \beta \leq \alpha'} \binom{\alpha'}{\beta} \partial^\beta (\tilde{\mu}^T \nabla A_j) \partial^{\alpha' - \beta} \partial_j u  + \sum_{0 < \beta \leq \alpha'} \binom{\alpha'}{\beta} \partial^\beta (\tilde{\mu}^T \nabla D) \partial^{\alpha' - \beta}  u.
 \end{align*}
 As $u$ and $\partial_t u$ are contained in $C(\clJ, \Hh{m-1})$, Lemma~\ref{LemmaRegularityForA0} implies that 
 the function $\tilde{f}_{\alpha'}$ is an element of $\Ltwoa$. 
 With this definition at hand, we deduce
 	\begin{align*}
 	&\Lambda^{\delta, \epsilon}
 	= \sum_{j = 0}^2 [T_\delta (\tilde{\mu}^T \nabla A_j), M_\epsilon] \partial_j T_\delta \partial^{\alpha'} u  + [T_\delta(\tilde{\mu}^T \nabla D), M_\epsilon] T_\delta \partial^{\alpha'} u \\
 	&\quad + [T_\delta (\tilde{\mu}^T A_0), M_\epsilon] \nabla  T_\delta \partial_t \partial^{\alpha'} u + [T_\delta (\tilde{\mu}^T D), M_\epsilon] \nabla T_\delta \partial^{\alpha'} u + \partial^{\alpha'} M_\epsilon T_\delta (\tilde{\mu}^T \nabla f)  \nonumber \\
 	&\quad  - M_\epsilon T_\delta \tilde{f}_{\alpha'} -\sum_{j = 1}^3 \partial^{\alpha'} M_\epsilon T_\delta (\tilde{\mu}^T A_j \nabla \partial_j  u) \\
 	&=: \tilde{\Lambda}^{\delta, \epsilon}  -\sum_{j = 1}^3 \partial^{\alpha'} M_\epsilon T_\delta (\tilde{\mu}^T A_j \nabla \partial_j  u).
 \end{align*}
 The cancellation properties of the differential operator from~\eqref{EquationTraceSmaller3EqualsZero} and~\eqref{EquationTraceLarger3EqualsZero} imply 
 that
 \begin{align*}
 	\Big( \sum_{k = 1}^3 \Lambda_{kk}^{\delta, \epsilon}, \sum_{k = 1}^3 \Lambda_{(k+3)k}^{\delta, \epsilon} \Big) 
 	= \Big( \sum_{k = 1}^3 \tilde{\Lambda}_{kk}^{\delta, \epsilon}, \sum_{k = 1}^3 \tilde{\Lambda}_{(k+3)k}^{\delta, \epsilon} \Big).
 \end{align*}
 In view of~\eqref{EquationEstimateForDivFirstResult}, we conclude that
 \begin{align}
  \label{EquationDivOfLMepsEdeltau}
  &\Div_\delta L_\delta M_\epsilon T_\delta \partial^{\alpha'} u = \Big( \sum_{k = 1}^3 \tilde{\Lambda}_{kk}^{\delta, \epsilon}, \sum_{k = 1}^3 \tilde{\Lambda}_{(k+3)k}^{\delta, \epsilon} \Big).
 \end{align}
 Since $\partial_j \partial^{\alpha'} u$ and $\partial^{\alpha'} u$ belong to $C(\clJ, \Ltwoh)$ 
 and $\nabla A_j$ and $\nabla D$ are contained in $L^\infty(\Omega)$, we have
 \begin{align}
 \label{EquationConvergenceToZeroCommutatorWithNablaA0}
 	[T_\delta (\tilde{\mu}^T \nabla A_j), M_\epsilon] \partial_j T_\delta \partial^{\alpha'} u  + [T_\delta(\tilde{\mu}^T \nabla D), M_\epsilon] T_\delta \partial^{\alpha'} u \longrightarrow 0
 \end{align}
 in $\Ltwoa$ as $\epsilon$ tends to $0$ for all $j \in \{0,1,2\}$.
 For the remaining commutator terms we employ  estimates for the commutator of a 
 $W^{1,\infty}$-function with a mollifier. Take $j \in \{1,2,3\}$. To satisfy
 the assumptions of these commutator estimates, we extend the function $\tilde{\mu}^T A_0(t)$ 
 by reflection 
 at $\partial \R^3_+$ to a function in $W^{1,\infty}(\R^3)$ which we still denote by $\tilde{\mu}^T A_0(t)$ 
 for all $t \in \clJ$. 
 Theorem~C.14 of~\cite{BenzoniGavage} now yields that $[T_\delta (\tilde{\mu}^T A_0), M_\epsilon] \partial_j  T_\delta \partial_t \partial^{\alpha'} u$
 maps into $L^2(\R^3)$, where we identify as usual the function $\partial_t \partial^{\alpha'} u$ with 
 its zero extension to $\R^3$. In particular, $([T_\delta (\tilde{\mu}^T A_0), M_\epsilon] \partial_j  T_\delta \partial_t \partial^{\alpha'} u)(t)$ 
 is an element of $\Ltwoh$ and Theorem~C.14 of~\cite{BenzoniGavage} further shows that
  \begin{align}
   &\Ltwohn{([T_\delta (\tilde{\mu}^T A_0), M_\epsilon] \partial_j  T_\delta \partial_t \partial^{\alpha'} u)(t)} \label{EquationCommutatorEstimateForAj}\\
   &\leq \|([T_\delta (\tilde{\mu}^T A_0), M_\epsilon] \partial_j  T_\delta \partial_t \partial^{\alpha'} u)(t)\|_{L^2(\R^3)} \nonumber \\
   &\leq C \|T_\delta (\tilde{\mu}^T A_0)(t)\|_{W^{1,\infty}(\R^3)} \| T_\delta \partial_t \partial^{\alpha'} u(t)\|_{L^2(\R^3)} \nonumber \\
    & \leq C  \|(\tilde{\mu}^T A_0)(t)\|_{W^{1,\infty}(\R^3_+)} \, \Ltwohn{\partial_t \partial^{\alpha'} u(t)}, \nonumber  \\
   &\lim_{\epsilon \rightarrow 0} \Ltwohn{([T_\delta (\tilde{\mu}^T A_0), M_\epsilon] \partial_j  T_\delta \partial_t \partial^{\alpha'} u)(t)} = 0    
   \label{EquationCommutatorConvergingToZeroAj}
  \end{align}
  for all $t \in \clJ$. Hence, the theorem of dominated convergence implies
  \begin{equation*}
   \Ltwoanwgamma{[T_\delta (\tilde{\mu}^T A_0), M_\epsilon] \partial_j  T_\delta \partial_t \partial^{\alpha'} u} \longrightarrow 0
  \end{equation*}
  as $\epsilon \rightarrow 0$. In the same way we deduce that also the other remaining 
  commutators in $\tilde{\Lambda}^{\delta,\epsilon}$ converge to $0$ as $\epsilon \rightarrow 0$. 
  As $T_\delta(\tilde{\mu}^T \nabla f)$ belongs to $H^{m-1}(J \times \R^2 \times (-\delta, \infty))$ 
  and $\tilde{f}_{\alpha'}$ to $\Ltwoa$, we arrive at
 \begin{align}
  \label{EquationConvergenceOfDivergence}
  &\Div_\delta L_\delta M_\epsilon T_\delta \partial^{\alpha'} u \\
  &\longrightarrow \sum_{k = 1}^3 T_\delta \Big((\tilde{f}_{\alpha'} + \partial^{\alpha'}(\tilde{\mu}^T \nabla f))_{kk},  (\tilde{f}_{\alpha'} + \partial^{\alpha'}(\tilde{\mu}^T \nabla f))_{(k+3)k} \Big) =: T_\delta f_{\div,\alpha'}  \nonumber
 \end{align}
 in $\Ltwoa$ as $\epsilon \rightarrow 0$. In the same way we infer that $L_\delta M_\epsilon T_\delta \partial^{\alpha'} u$ 
 converges to $T_\delta f_{\alpha'}$ with
 \begin{equation*}
  f_{\alpha'}:= \partial^{\alpha'} f - \sum_{j = 0}^2 \sum_{0 < \beta \leq \alpha'} \binom{\alpha'}{\beta} \partial^\beta A_j \partial_j \partial^{\alpha' - \beta} u 
    - \sum_{0 < \beta \leq \alpha'} \binom{\alpha'}{\beta} \partial^\beta D \partial^{\alpha' - \beta} u
 \end{equation*}
 in $\G{0}$. Here one combines commutator estimates as in~\eqref{EquationCommutatorEstimateForAj} 
 and~\eqref{EquationCommutatorConvergingToZeroAj} with a standard compactness argument to derive 
 that the convergence is also uniform in~$t$, see~\cite[Lemma~4.1]{SpitzDissertation} for details.
 
 Next take $\eta, r > 0$ with $A_0 \geq \eta$, $\|A_i\|_{W^{1,\infty}(\Omega)} \leq r$,
 and $\|D\|_{W^{1,\infty}(\Omega)} \leq r$ for all $i \in \{0,\ldots, 3\}$. Note that 
 in particular 
 $\|A_i(0)\|_{L^\infty(\R^3_+)} \leq r$ and $\|D(0)\|_{L^\infty(\R^3_+)} \leq r$ for $i \in \{0, \ldots, 3\}$. 
 Now let $\delta > 0$ and take $n_\delta \in \N$ with $n_\delta^{-1} < \delta$. 
 Fix a number $\gamma \geq 1$ and 
 define the constant $C' = C'(\eta,r, T)$ by
 \begin{align}
 \label{EquationIntroductionOfCPrime}
  C' = \Big(C_{\ref{PropositionCentralEstimateInNormalDirection}; 1,0}+ T C_{\ref{PropositionCentralEstimateInNormalDirection}; 1} + \frac{C_{\ref{PropositionCentralEstimateInNormalDirection}; 1}}{\gamma} \Big) e^{C_{\ref{PropositionCentralEstimateInNormalDirection}; 1} T}
 \end{align}
 where $C_{\ref{PropositionCentralEstimateInNormalDirection}; 1,0} = C_{\ref{PropositionCentralEstimateInNormalDirection}; 1,0}(\eta,r)$ 
 and $C_{\ref{PropositionCentralEstimateInNormalDirection}; 1} = C_{\ref{PropositionCentralEstimateInNormalDirection}; 1}(\eta,r,T)$ 
 are the corresponding constants from Proposition~\ref{PropositionCentralEstimateInNormalDirection}.
 Observe that $M_\epsilon T_\delta \partial^{\alpha'} u$
 solves the initial value problem~\eqref{IVP} with differential operator $L_\delta$,
 inhomogeneity $L_\delta M_\epsilon T_\delta \partial^{\alpha'} u$ and initial value $M_\epsilon T_\delta u_0$ for each $\epsilon \in (0, \delta)$. Moreover, 
 $\|T_\delta A_i\|_{W^{1,\infty}(\Omega)} \leq r$ and $\|T_\delta A_i(0)\|_{L^\infty(\R^3_+)} \leq r$
 for all $\delta > 0$ and $i \in \{0, \ldots, 3\}$, and the same is true for $D$.
 Proposition~\ref{PropositionCentralEstimateInNormalDirection} thus shows
 \begin{align}
  \label{EquationEstimateForRegularityInNormalDirectionOne}
    &\Ltwohnt{\nabla(M_{\frac{1}{n}} T_\delta \partial^{\alpha'} u - M_{\frac{1}{k}} T_\delta \partial^{\alpha'} u)}^2  \\
    &\leq C' \Big( \sum_{j = 0}^2 \Ltwohnt{(M_{\frac{1}{n}} - M_{\frac{1}{k}}) \partial_j T_\delta \partial^{\alpha'} u}^2 
    +  \Gnorm{0}{L_\delta(M_{\frac{1}{n}} - M_{\frac{1}{k}}) \partial^{\alpha'} u}^2  \nonumber\\
    & \quad + \Ltwoan{\Div_\delta L_\delta(M_{\frac{1}{n}} - M_{\frac{1}{k}}) \partial^{\alpha'} u}^2 + \Hhn{1}{(M_{\frac{1}{n}} - M_{\frac{1}{k}}) T_\delta \partial^{\alpha'} u_0}^2 \Big),  \nonumber
 \end{align}
for all $n, k \in \N$ with $n, k \geq n_\delta$. Since $\partial_j T_\delta \partial^{\alpha'} u$ 
belongs to $C(\clJ, L^2(\R^2 \times (-\delta, \infty)))$ for $j \in \{0,1,2\}$, $\clJ$ is compact, and $T_\delta \partial^{\alpha'} u_0$ 
is contained in $H^1(\R^2 \times (-\delta,\infty))$, the sum and the last term on the right-hand side 
of~\eqref{EquationEstimateForRegularityInNormalDirectionOne} converge to zero as $n,k \rightarrow \infty$.
Exploiting that $L_\delta M_{\frac{1}{n}} T_\delta \partial^{\alpha'} u$ converges to $f_{\alpha'}$ 
and $\Div_\delta L_\delta M_{\frac{1}{n}} T_\delta \partial^{\alpha'} u$ to $T_\delta f_{\operatorname{div},\alpha'}$
in $\G{0}$ respectively $\Ltwoa$ as $n \rightarrow \infty$, we infer that $(\nabla M_{\frac{1}{n}} T_\delta \partial^{\alpha'} u)_{n \geq n_\delta}$
is a Cauchy sequence in $\G{0}$. 
Since $(M_{\frac{1}{n}} T_\delta \partial^{\alpha'} u)_{n \geq n_\delta}$ 
converges to $T_\delta \partial^{\alpha'} u$ in $\G{0}$, we conclude that $T_\delta \partial^{\alpha'} u$ belongs to $C(\clJ, \Hh{1})$ and that 
\begin{align}
 \label{EquationConvergenceOfGradientMEpsEdeltau}
 \Gnorm{0}{\nabla M_{\frac{1}{n}} T_\delta \partial^{\alpha'} u - \nabla T_\delta \partial^{\alpha'} u} \longrightarrow 0
\end{align}
as $n \rightarrow \infty$
for all $\delta > 0$.
Applying Proposition~\ref{PropositionCentralEstimateInNormalDirection} directly to $M_{\frac{1}{n}} T_\delta \partial^{\alpha'} u$ 
and letting $n \rightarrow \infty$, we obtain
\begin{align}
\label{EquationEstimateForRegularityInNormalDirectionOnlyTranslation}
 \Ltwohnt{\nabla T_\delta \partial^{\alpha'} u }^2 
    &\leq C' \Big( \sum_{j = 0}^2 \Ltwohnt{ \partial_j T_\delta \partial^{\alpha'} u}^2 
    +  \Ltwohnt{T_\delta f_{\alpha'} }^2 + \Ltwoan{ T_\delta f_{\div, \alpha'}}^2\nonumber \\
    & \qquad \qquad  + \Hhn{1}{T_\delta \partial^{\alpha'} u_0}^2 \Big),
\end{align}
for $\gamma \geq 1$, cf.~\eqref{EquationEstimateForRegularityInNormalDirectionOne}.

III)
We next show that $\partial^{\alpha'} u(t)$ is an element of $\Hh{1}$ for all $t \in \clJ$. Note that we only have to prove that
$\partial_3 \partial^{\alpha'} u(t)$ belongs to $\Ltwoh$ for this claim. We abbreviate $\R^2 \times (\delta, \infty)$ by $\R^3_\delta$ and 
denote the restriction operator to $\R^3_\delta$ by
$R_\delta$ for all $\delta > 0$. In the next step we show that $R_\delta u(t)$ belongs to $H^1(\R^3_\delta)$ 
for all $\delta > 0$.

Fix $\delta > 0$ and $t \in \clJ$. Let $\phi \in C_c^\infty(\R^3_\delta)$. We compute
\begin{align*}
 &\int_{\R^3_\delta} R_\delta \partial^{\alpha'} u(t,x) \partial_3 \phi(x) dx 
  = \int_{\R^3_+} T_\delta  \partial^{\alpha'} u(t,x) \partial_3 T_\delta \phi(x) dx \\
  &= - \int_{\R^3_+} \partial_3 T_\delta \partial^{\alpha'} u(t,x) T_\delta \phi(x) dx = - \int_{\R^3_\delta} T_{-\delta}  \partial_3 T_\delta \partial^{\alpha'} u(t,x) \phi(x) dx,
\end{align*}
using that $T_\delta \phi \in C_c^\infty(\R^3_+)$.
It follows
\begin{equation}
  \label{EquationDerivativeOfRestrictedSolution}
 \partial_3 R_\delta \partial^{\alpha'} u(t) = T_{-\delta} \partial_3 T_\delta \partial^{\alpha'} u(t) \in L^2(\R^2 \times (\delta, \infty))
\end{equation} 
as $\partial_3 T_\delta \partial^{\alpha'} u(t) \in \Ltwoh$ by step~II).

Next pick $\overline{\delta} > \delta$ and $\phi \in C_c^\infty(\R^3_{\overline{\delta}})$. We compute
\begin{align*}
 &\int_{\R^3_{\overline{\delta}}} R_{\overline{\delta}}  \partial_3 R_\delta \partial^{\alpha'} u(t,x) \phi(x) dx 
 = \int_{\R^3_\delta} \partial_3 R_\delta \partial^{\alpha'} u(t,x)  \phi(x) dx \\
 &= - \int_{\R^3_\delta} R_\delta \partial^{\alpha'} u(t,x)  \partial_3  \phi(x) dx = - \int_{\R^3_{\overline{\delta}}} R_{\overline{\delta}} \partial^{\alpha'} u(t,x)  \partial_3 \phi(x) dx \\
 &= \int_{\R^3_{\overline{\delta}}} \partial_3 R_{\overline{\delta}} \partial^{\alpha'} u(t,x) \phi(x) dx,
\end{align*}
where we exploited that $\supp(\phi) \Subset  \R^3_{\overline{\delta}}$. Since $\phi \in C_c^\infty(\R^3_{\overline{\delta}})$ 
was arbitrary, we conclude $\partial_3 R_{\overline{\delta}} \partial^{\alpha'} u(t) =   \partial_3 R_\delta \partial^{\alpha'} u(t)$ on $\R^3_{\overline{\delta}}$.
In particular, we can define the function $v(t) \in L^2_{\operatorname{loc}}(\R^3_+)$ by setting for all $\delta > 0$
\begin{align*}
 v(t,x) = \partial_3 R_\delta \partial^{\alpha'} u(t,x) \quad \text{for almost all } x \in \R^3_\delta.
\end{align*}

Take $\phi \in C_c^\infty(\R^3_+)$. Fix a number $\tau > 0$ with $\dist( \supp(\phi), \partial \R^3_+) > \tau$, i.e., 
$\supp(\phi) \Subset  \R_\tau^3$. We then deduce
\begin{align*}
 \int_{\R^3_+} \partial^{\alpha'} u(t,x) \partial_3 \phi(x) dx 
 &= \int_{\R^3_\tau} R_\tau \partial^{\alpha'} u(t,x) \partial_3 \phi(x) dx 
 = - \int_{\R^3_\tau} \! \partial_3 R_\tau \partial^{\alpha'} u(t,x) \phi(x) dx \\
 &= - \int_{\R^3_\tau} v(t,x) \phi(x) dx 
 = - \int_{\R^3_+} v(t,x) \phi(x) dx.
\end{align*}
This means that $\partial_3 \partial^{\alpha'} u(t) = v(t)\in L^2_{\operatorname{loc}}(\R^3_+)$.

We further note that $\partial_3 R_\delta \partial^{\alpha'} u(t)$ converges pointwise almost everywhere on $\R^3_+$ 
to $v(t) = \partial_3 \partial^{\alpha'} u(t)$ 
as $\delta \rightarrow 0$. 
Using~\eqref{EquationDerivativeOfRestrictedSolution}, we further infer
\begin{align}
\label{EquationNormOfRestrictionAndTranslation}
 \Ltwohn{\partial_3 R_\delta \partial^{\alpha'} u(t)}^2 &= \int_{\R^3_\delta} |T_{-\delta} \partial_3 T_\delta \partial^{\alpha'} u(t,x)|^2 dx
 = \int_{\R^3_+} |\partial_3 T_\delta \partial^{\alpha'} u(t,x)|^2 dx \nonumber \\
 & = \Ltwohn{ \partial_3 T_\delta \partial^{\alpha'} u(t)}^2.
\end{align}

Let $(\delta_n)_n$ be a null-sequence. Fatou's lemma, \eqref{EquationNormOfRestrictionAndTranslation}, 
and~\eqref{EquationEstimateForRegularityInNormalDirectionOnlyTranslation} then imply
\begin{align*}
 &\int_{\R^3_+} |\partial_3 \partial^{\alpha'} u(t,x)|^2 dx = \int_{\R^3_+} \liminf_{n \rightarrow \infty} |\partial_3 R_{\delta_n} \partial^{\alpha'} u(t,x)| dx \\
 &\leq \liminf_{n \rightarrow \infty} \int_{\R^3_+} |\partial_3 R_{\delta_n} \partial^{\alpha'} u(t,x)|^2 dx 
  = \liminf_{n \rightarrow \infty} \Ltwohn{\partial_3 T_{\delta_n} \partial^{\alpha'} u(t)}^2 \\
  &\leq e^{2 \gamma T} \liminf_{n \rightarrow \infty} \Gnorm{0}{\nabla T_{\delta_n} \partial^{\alpha'} u}^2 \\
  &\leq C'e^{2 \gamma T} \Big( \sum_{j = 0}^2 \Ltwohnt{ \partial_j \partial^{\alpha'} u}^2 +  \Ltwohnt{f_{\alpha'} }^2 + \Ltwoan{f_{\div, \alpha'}}^2 + \Hhn{1}{\partial^{\alpha'} u_0}^2 \Big) \\
  & =: K_u^2 < \infty,
\end{align*}
where we used that $\partial_j \partial^{\alpha'} u, \partial^{\alpha'} u \in C(\clJ, \Ltwoh)$ for $j \in \{0,1,2\}$, $f_{\alpha'} \in  \G{0}$, 
$f_{\div, \alpha'} \in \Ltwoa$, and $\partial^{\alpha'} u_0 \in \Hh{1}$.
We conclude that $\partial_3 \partial^{\alpha'} u(t)$ belongs to $\Ltwoh$ with $\Ltwohn{\partial_3 \partial^{\alpha'} u(t)} \leq K_u$ for all $t \in \clJ$.

We further point out that $R_\delta \partial_3 \partial^{\alpha'} u(t) = R_\delta v(t) = \partial_3 R_\delta \partial^{\alpha'} u(t)$. This fact 
implies that
\begin{align*}
 |\partial_3 R_\delta \partial^{\alpha'} u(t)| \leq |\partial_3 \partial^{\alpha'} u(t)|
\end{align*}
on $\R^3_+$. As $\partial_3 R_\delta \partial^{\alpha'} u(t)$ tends to $\partial_3 \partial^{\alpha'} u(t)$ 
pointwise almost everywhere on $\R^3_+$, we obtain that
\begin{align*}
 \partial_3 R_\delta \partial^{\alpha'}u(t) \longrightarrow \partial_3 \partial^{\alpha'} u(t)
\end{align*}
in $\Ltwoh$ as $\delta \rightarrow 0$.

Since $\partial_3 T_\delta \partial^{\alpha'} u$ belongs to $C(\clJ, \Ltwoh)$ for all $\delta > 0$, one can argue as 
in~\eqref{EquationNormOfRestrictionAndTranslation} to deduce that $\partial_3 R_\delta \partial^{\alpha'} u$ is also 
continuous on $\clJ$ with values in $\Ltwoh$ and thus strongly measurable.
Hence, $\partial_3 \partial^{\alpha'} u$ is the pointwise limit of strongly measurable functions and therefore itself strongly measurable 
on $J$ with values in $\Ltwoh$. As a result, $\partial_3 \partial^{\alpha'} u$ and thus $\nabla \partial^{\alpha'} u$ belong to 
$L^\infty(J, \Ltwoh)$. We then obtain via Proposition~\ref{PropositionCentralEstimateInNormalDirection} that $\partial^{\alpha'} u$ is contained 
in $C(\clJ, \Hh{1})$.
\end{proof}


\begin{cor}
 \label{CorollaryRegularityInNormalDirectionL2InTime}
Let $\eta > 0$, $m \in \N$, and $\tilde{m} =  \max\{m,3\}$. Take $A_0 \in \Fupdwl{\tilde{m}}{cp}{\eta}$, 
$A_1, A_2 \in \Fcoeff{\tilde{m}}{\operatorname{cp}}$, $A_3 = A_3^{\operatorname{co}}$,
and $D \in \Fuwl{\tilde{m}}{cp}$. Pick $f \in \Ha{m}$, 
and $u_0 \in \Hh{m}$. 
Let $u$ be a solution of the initial value problem~\eqref{IVP} with differential operator $L = L(A_0,\ldots, A_3, D)$, inhomogeneity 
$f$, and initial value $u_0$. Assume that $u$ belongs to $\bigcap_{j = 1}^m C^j(\clJ, \Hh{m-j})$.

Take $k \in \{1, \ldots, m\}$ and a multi-index $\alpha \in \N_0^4$ with $|\alpha| = m$, $\alpha_0 = 0$,
and $\alpha_3 = k$. 
Suppose that $\partial^\beta u$ is contained in $\Ltwoa$ for all $\beta \in \N_0^4$ with $|\beta| = m$ and 
$\beta_3 \leq k-1$.
Then $\partial^\alpha u$ is an element of $\Ltwoa$.
\end{cor}

\begin{proof}
 We only have to make small adaptions to the proof of Lemma~\ref{LemmaRegularityInNormalDirection}. In step II) 
 of that proof we replace the a priori estimate~\eqref{EquationFirstOrderFinal} from Proposition~\ref{PropositionCentralEstimateInNormalDirection} 
 by estimate~\eqref{EquationFirstOrderL2}. The arguments from step II) then yield that $T_\delta \partial^{\alpha'} u$
 is an element of $L^2(J, \Hh{1})$. Integrating over the time-space domain in step III) of the proof of Lemma~\ref{LemmaRegularityInNormalDirection}, 
 we derive that $\partial^{\alpha'} u$ belongs to $L^2(J, \Hh{1})$.
\end{proof}


For the regularization in spatial tangential variables, we first introduce the family of  
norms
\begin{align}
 \Hhtan{s}{v}^2 &= \int_{\R_+} \int_{\R^2} (1 + |\xi|^2)^{s}  |(\mathcal{F}_2 v)(\xi,x_3)|^2 d\xi dx_3, \nonumber \\
 \label{EquationWeightedTangentialNorms}
 \Hhtanw{s}{v}{\delta}^2 &= \int_{\R_+} \int_{\R^2} (1 + |\xi|^2)^{s+1} (1 + |\delta \xi|^2)^{-1} |(\mathcal{F}_2 v)(\xi,x_3)|^2 d\xi dx_3
\end{align}
for all $v \in \mathcal{S}'(\overline{\R^3_+})$ with $\mathcal{F}_2 v \in L^2_{\operatorname{loc}}(\R^3_+)$,  $s \in \R$ and $\delta > 0$, where $\mathcal{F}_2$ denotes the Fourier transform in $x_1$- and $x_2$-direction 
and $\mathcal{S}'(\overline{\R^3_+})$ the space of tempered distributions on $\overline{\R^3_+}$,
see Section~1.7 and Section~2.4 in~\cite{Hoermander}. The space $\Hhta{s}$ consists of those  $v$ 
for which $\Hhtan{s}{v}$ is finite. As in the unweighted case we have of course the identity
\begin{equation}
 \label{EquationDerivativeInWeightedNorm}
 \Hhtanw{s+1}{v}{\delta}^2 = \Hhtanw{s}{v}{\delta}^2 + \sum_{j = 1}^2 \Hhtanw{s}{\partial_j v}{\delta}^2
\end{equation}
for all $s \in \R$ and $\delta > 0$. We further note that the definition directly implies
\begin{equation*}
 \Hhtanw{s}{v}{\delta} \leq \Hhtan{s+1}{v}
\end{equation*}
for all $v \in \Hhta{s+1}$, $s \in \R$, and $\delta > 0$.

We further take a function $\chi \in C_c^\infty(\R^2)$ such that $\mathcal{F}_2 \chi(\xi) = O(|\xi|^{m+1})$ as $\xi \rightarrow 0$ 
and $\mathcal{F}_2 \chi(t \xi) = 0$ for all $t \in \R$ implies $\xi = 0$, cf.~\cite{Hoermander}.
As usual we set $\chi_\epsilon(x) = \epsilon^{-2} \chi(x / \epsilon)$ 
for all $x \in \R^2$ and $\epsilon > 0$ and denote the convolution in spatial tangential variables with $\chi_{\epsilon}$ 
by $J_\epsilon$, i.e., 
\begin{equation*}
 J_\epsilon v (x) = \chi_\epsilon *_{\operatorname{ta}} v(x)= \int_{\R^2} \chi(y) v(x_1 - y_1, x_2 - y_2, x_3) dy
\end{equation*}
for all $v \in \Ltwoh$.

One of the advantages to work with the weighted norms from~\eqref{EquationWeightedTangentialNorms}
is that one can reduce the task of showing that 
a function $v$ from $\Hhta{s}$ belongs to $\Hhta{s+1}$ to finding a uniform bound in $\delta > 0$ for the $H^{s}_{\operatorname{ta},\delta}(\R^3_+)$-norms.
The following properties of this family of weighted norms are a consequence of~(2.4.4), Theorem~2.4.1, Theorem~2.4.2, 
Theorem~2.4.5, and Theorem~2.4.6 in~\cite{Hoermander}.
\begin{lem}
 \label{LemmaPropertiesOfWeightedNorms}
 Let $s \in [0,m+1)$, $v \in \Hhta{s-1}$, and let $A \in C^{\infty}(\R^3_+)$ be constant outside of a 
 compact subset of $\overline{\R^3_+}$.
 \begin{enumerate}
 \item \label{ItemBoundedUniformInDeltaImpliesRegPlusOne} Assume that there is a constant $C$, independent of $\delta$, 
 such that
 \begin{equation*}
  \Hhtanw{s-1}{v}{\delta} \leq C
 \end{equation*}
 for all $\delta > 0$ in a neighborhood of $0$. Then $v$ belongs to $\Hhta{s}$.
 
 \item \label{ItemEstimateMollifierFromBelow} There exist constants $c$ and $C$, independent of $\delta$ and $v$, 
 such that
 \begin{align*}
  c \Hhtanw{s-1}{v}{\delta}^2 &\leq \Hhtan{s-1}{v}^2 + \int_0^1 \Ltwohn{J_\epsilon v}^2 \epsilon^{-2s-1} \Big(1 + \frac{\delta^2}{\epsilon^2} \Big)^{-1} d\epsilon \\ 
  &\leq C \Hhtanw{s-1}{v}{\delta}^2
 \end{align*}
 for all $\delta \in (0,1)$.
 
 \item \label{ItemCommutatorEstimateForWeightedNorm} There is a constant $C$, independent of $\delta$ and $v$, such that
 \begin{equation*}
  \int_0^1 \Ltwohn{A J_\epsilon v - J_\epsilon(A v)}^2 \epsilon^{-2s-1} \Big(1 + \frac{\delta^2}{\epsilon^2} \Big)^{-1} d\epsilon 
  \leq C \Hhtanw{s-2}{v}{\delta}^2
 \end{equation*}
 for all $\delta \in (0,1)$.
 \end{enumerate}
\end{lem}
 We note that H{\"o}rmander states the commutator estimate only for coefficients from the 
 Schwartz space.
The proof of Theorem~2.4.2 in~\cite{Hoermander} however also works for smooth coefficients which are 
constant outside of a compact set. 

Employing the family of weighted norms from~\eqref{EquationWeightedTangentialNorms} and 
Lemma~\ref{LemmaPropertiesOfWeightedNorms}, we can now show how regularity in time implies 
regularity in tangential directions. Since we want to apply Lemma~\ref{LemmaPropertiesOfWeightedNorms}, 
we have to assume that the coefficients belong to $C^\infty(\overline{\Omega})$. We will 
return to coefficients in $\F{m}$ with an approximation argument 
below.
\begin{lem}
 \label{LemmaRegularityInSpaceTangential}
 Let $\eta > 0$, $m \in \N$ and $\tilde{m} = \max\{m,3\}$.
 Take coefficients $A_0 \in \Fupdwl{\tilde{m}}{cp}{\eta}$, $A_1, A_2 \in \Fcoeff{\tilde{m}}{cp}$, 
 $A_3 = A_3^{\operatorname{co}}$,
 $D \in \Fuwl{\tilde{m}}{cp}$ and $B = B^{\operatorname{co}}$. We further assume that these 
 coefficients belong 
 to $C^\infty(\overline{\Omega})$.
 Let $u$ be the weak solution
 of~\eqref{IBVP} with differential operator $L = L(A_0,\ldots, A_3, D)$, inhomogeneity
 $f \in \Hata{m}$, boundary value $g \in \E{m}$,  and initial value $u_0 \in \Hhta{m}$. 
 Suppose that $u$ belongs to $\bigcap_{j = 1}^m C^j(\clJ, \Hh{m-j})$. Pick a multi-index $\alpha \in \N_0^4$ with $|\alpha| = m$ 
 and $\alpha_0 = \alpha_3 = 0$. Then $\partial^\alpha u$ is an element of $C(\clJ, \Ltwoh)$.
\end{lem}
\begin{proof}
 I) We will establish the assertion in two steps.
 First we will show that $u$  
 is an element of $L^\infty(J, \Hhta{m})$. To that purpose we will apply
 Lemma~\ref{LemmaPropertiesOfWeightedNorms} and the a priori 
 estimates from Lemma~\ref{LemmaEllerResult}.
 
  Fix a parameter $\delta \in (0,1)$. Let $\gamma > 0$. The generic constants appearing in the following will all be independent 
  of $\delta$ and $\gamma$. We further note that Lemma~\ref{LemmaPropertiesOfWeightedNorms} will be used in almost every step 
  in the following so that we will not cite it every time. Applying the differential 
  operator $L$ to $J_\epsilon u$, 
  we obtain
  \begin{equation}
  \label{EquationCommutingDiffOpAndMollifierWithoutNormalDerivatives}
   L J_\epsilon  u = J_\epsilon f + \sum_{j = 0}^2 [A_j, J_\epsilon] \partial_j u + [D, J_\epsilon] u
  \end{equation}
  for all $\epsilon \in (0,1)$ since $A_3 = A_3^{\operatorname{co}}$. Lemma~\ref{LemmaPropertiesOfWeightedNorms} 
  allows us to estimate
  \begin{align}
  \label{EquationEstimatingTheCommutatorErrorheps}
   \int_J e^{-2\gamma t} &\int_0^1  \Ltwohn{[A_j, J_\epsilon] \partial_j  u(t)}^2 \epsilon^{-2m-1} \Big(1 + \frac{\delta^2}{\epsilon^2} \Big)^{-1} d\epsilon dt \\
   &\leq C \Hhtanwt{m-1}{u}{\delta}^2 + C \Hhtanwt{m-2}{ \partial_t u}{\delta}^2  \nonumber \\
   &\leq C \Hhtanwt{m-1}{u}{\delta}^2 + C \Hangamma{m-1}{\partial_t u}^2 \nonumber
  \end{align}
  for all $j \in \{0,1,2\}$. The commutator $[D, J_\epsilon] u$ is treated analogously. 
  In particular, $L J_\epsilon  u$ is an element of $\Ltwoa$. 
  Identity~\eqref{EquationCommutingDiffOpAndMollifierWithoutNormalDerivatives} 
  further implies that $A_3 \partial_3 J_\epsilon u$ belongs to $\Ltwoa$ so that $A_3 J_\epsilon u$ is an 
  element of $L^2(J, \Hh{1})$. We infer that the trace of $B J_\epsilon  u$ is contained in $L^2(J, H^{1/2}(\partial \R^3_+))$.
  Finally, $J_{\epsilon} u_0$ is an element of $\Ltwoh$ so that we can apply the a priori estimate from 
  Lemma~\ref{LemmaEllerResult} to the function $J_\epsilon u$. Before doing so, we use Lemma~\ref{LemmaPropertiesOfWeightedNorms} to 
  derive
  \begin{align}
  \label{EquationEstimateWeightedNormalphaprimeuAgainstConvolution}
   &\sup_{t \in J} e^{-2 \gamma t} \Hhtanw{m-1}{ u(t)}{\delta}^2 \leq C \sup_{t \in J} e^{-2 \gamma t} \Hhn{m-1}{u(t)}^2 \\
   &\hspace{11em}+ C \sup_{t \in J} e^{-2\gamma t} \int_0^1 \Ltwohn{J_\epsilon  u(t)}^2 \epsilon^{-2m-1} \Big(1 + \frac{\delta^2}{\epsilon^2} \Big)^{-1} d\epsilon  \nonumber\\
   &\leq C \Gnorm{m-1}{u}^2  + C  \int_{0}^1 \Gnorm{0}{J_\epsilon u}^2 \epsilon^{-2m-1} \Big(1 + \frac{\delta^2}{\epsilon^2} \Big)^{-1} d\epsilon \nonumber
  \end{align}
  for all $\gamma > 0$. The a priori estimates from Lemma~\ref{LemmaEllerResult} now show that there is a constant 
  $C_0$ and a number $\gamma_0 > 0$ such that
  \begin{align}
   \label{EquationApplyingAPrioriEstimatesToJepsPartialalphaPrime}
   \Gnorm{0}{J_\epsilon  u}^2 &\leq C_0 \Ltwohn{J_\epsilon u_0}^2 + C_0 \|B J_\epsilon u\|_{L^2_\gamma(J, H^{1/2}(\partial \R^3_+))}^2 + \frac{C_0}{\gamma} \Ltwoan{L J_\epsilon u}^2
  \end{align}
  for all $\gamma \geq \gamma_0$. Fix such a parameter $\gamma$ in the following. We next treat the terms appearing in~\eqref{EquationApplyingAPrioriEstimatesToJepsPartialalphaPrime}.
  Applying identity~\eqref{EquationCommutingDiffOpAndMollifierWithoutNormalDerivatives}, Fubini's theorem, and estimate~\eqref{EquationEstimatingTheCommutatorErrorheps}, we infer
  \begin{align}
  	\label{EquationEstimatingLJepsuFirstStep}
  	 &\int_0^1 \Ltwoan{L J_\epsilon  u}^2  \epsilon^{-2m-1} \Big(1 + \frac{\delta^2}{\epsilon^2} \Big)^{-1} d\epsilon \nonumber \\
  	 &\leq C \int_0^1 \Big(\Ltwoan{J_\epsilon f}^2 + \sum_{j = 0}^2 \Ltwoan{[A_j, J_\epsilon] \partial_j u}^2 \nonumber \\
  	 &\hspace{6 em} + \Ltwoan{[D, J_\epsilon] u}^2 \Big)  \epsilon^{-2m-1} \Big(1 + \frac{\delta^2}{\epsilon^2} \Big)^{-1} d\epsilon \nonumber \\
  	 &= C \int_J e^{-2 \gamma t} \int_0^1 \Big(\Ltwohn{J_\epsilon f(t)}^2 + \sum_{j=0}^2 \Ltwohn{[A_j, J_\epsilon] \partial_j u(t)}^2 \nonumber \\
  	 &\hspace{8 em} + \Ltwohn{[D, J_\epsilon] u(t)}^2 \Big)  \epsilon^{-2m-1} \Big(1 + \frac{\delta^2}{\epsilon^2} \Big)^{-1} d\epsilon \, dt \nonumber \\
  	 &\leq C \Hhtanwt{m-1}{f}{\delta}^2 + C \Hhtanwt{m-1}{u}{\delta}^2 + C \Hangamma{m-1}{\partial_t u}^2 ,
  \end{align}
  where we once again employed Lemma~\ref{LemmaPropertiesOfWeightedNorms} in 
  the last line. Since the matrix $B$ is constant, it commutes with the mollifier 
  $J_\epsilon$ so that $B J_\epsilon u = J_\epsilon g$ for all $\epsilon > 0$.
  We note that the proof of Lemma~\ref{LemmaPropertiesOfWeightedNorms}~\ref{ItemEstimateMollifierFromBelow}, 
  see Theorems~2.4.5 and~2.4.1 in~\cite{Hoermander}, shows that
  \begin{equation*}
   \int_0^1 \|J_\epsilon v\|_{H^{1/2}(\partial \R^3_+)}^2 \epsilon^{-2m-1} \Big(1 + \frac{\delta^2}{\epsilon^2} \Big)^{-1} d\epsilon \leq C \|v\|_{H^{m-1/2}_{\delta}(\partial \R^3_+)}^2
  \end{equation*}
  for all $v \in H^{m-1/2}(\partial \R^3_+)$ and $\delta \in (0,1)$. Here we identify $\partial \R^3_+$ 
  with $\R^2$ and make the natural adaptions in~\eqref{EquationWeightedTangentialNorms} to 
  define $H^{m-1/2}_{\delta}(\partial \R^3_+)$, see also Section~2.4 in~\cite{Hoermander}. 
  Consequently,
  \begin{align}
   \label{EquationEstimateForFirstPartOfgalphaprime}
   &\int_0^1 \|B J_\epsilon u\|_{L^2_\gamma(J, H^{1/2}(\partial \R^3_+))}^2 \epsilon^{-2m-1} \Big(1 + \frac{\delta^2}{\epsilon^2} \Big)^{-1} d\epsilon \\
   &=\int_J e^{- 2\gamma t} \int_0^1 \|J_\epsilon g(t)\|_{H^{1/2}(\partial \R^3_+)}^2 \epsilon^{-2m-1} \Big(1 + \frac{\delta^2}{\epsilon^2} \Big)^{-1} d\epsilon dt \nonumber \\
   &\leq C \int_J e^{-2\gamma t} \|g(t)\|_{H^{m-1/2}_\delta(\partial \R^3_+)}^2 dt
   \nonumber \\
   &\leq C \int_J e^{-2\gamma t} \|g(t)\|_{H^{m+1/2}(\partial \R^3_+)}^2 dt 
   \leq C \Enorm{m}{g}^2. \nonumber
  \end{align}
  For the initial value we note that Lemma~\ref{LemmaPropertiesOfWeightedNorms} directly 
  yields
  \begin{equation}
   \label{EquationEstimateForJepsInitialValue}
   \int_0^1 \Ltwohn{J_\epsilon  u_0}^2  \epsilon^{-2m-1} \Big(1 + \frac{\delta^2}{\epsilon^2} \Big)^{-1} d\epsilon 
   \leq C \Hhtanw{m-1}{u_0}{\delta}^2 \leq  C \Hhn{m}{u_0}^2.
  \end{equation}
  Inserting~\eqref{EquationApplyingAPrioriEstimatesToJepsPartialalphaPrime} to~\eqref{EquationEstimateForJepsInitialValue} 
  into~\eqref{EquationEstimateWeightedNormalphaprimeuAgainstConvolution}, we obtain that
  \begin{align}
  \label{EquationEstimateForpartialalphaprimeInHta0deltaEstimatesForRightHandSideInserted}
   &\sup_{t \in J} e^{-2 \gamma t} \Hhtanw{m-1}{u(t)}{\delta}^2 \\
   &\leq C \Gnorm{m-1}{u}^2  + C \Hhn{m}{u_0}^2 
   +C \Enorm{m}{g}^2 \nonumber \\
   &\qquad + \frac{C}{\gamma} \Big( \|f\|_{L^2_\gamma(J, \Hhta{m})}^2 + C \Hangamma{m-1}{\partial_t u}^2 + C \Hhtanwt{m-1}{u}{\delta}^2  \Big). \nonumber 
  \end{align}
  Choosing a number $\gamma$ large enough, we thus find a constant $K_1$ such that 
  \begin{align}
   \label{EquationEstimateForpartialalphaprimeInHta0deltaGammaChosen}
   \sup_{t \in J} e^{-2 \gamma t} \Hhtanw{m-1}{ u(t)}{\delta}^2  \leq K_1
  \end{align}
  for all $\delta \in (0,1)$.
  Hence, Lemma~\ref{LemmaPropertiesOfWeightedNorms}~\ref{ItemBoundedUniformInDeltaImpliesRegPlusOne} implies that 
  $u(t)$ belongs to $\Hhta{m}$ for all $t \in \clJ$ and that $u$ is 
  contained in $L^{\infty}(J, \Hhta{m})$.
    
  II) Applying Corollary~\ref{CorollaryRegularityInNormalDirectionL2InTime} inductively, we infer that 
  $u$ is an element of $\Ha{m}$. To establish that $u$
  belongs to $\G{m}$, we apply Lemma~\ref{LemmaEllerResult} again. 
  
  Fix a multi-index $\alpha \in \N_0^4$ with $|\alpha| = m$ and $ \alpha_0 = \alpha_3 = 0$.
  Since $u$ is a solution of~\eqref{IBVP} and we already know that 
  $u \in \Ha{m}$, we derive
 \begin{align*}
  L \partial^\alpha u  = \partial^\alpha f 
  - \sum_{j = 0 }^2 \sum_{0 < \beta \leq \alpha} \binom{\alpha}{\beta} \partial^\beta A_j \partial^{\alpha - \beta} \partial_j u 
  - \sum_{0 < \beta \leq \alpha} \binom{\alpha}{\beta} \partial^\beta D \partial^{\alpha - \beta} u
  =: f_\alpha,
 \end{align*}
 where $f_\alpha$ belongs to $\Ltwoa$. 
 Next consider the function $J_{\frac{1}{n}} \partial^\alpha u$, which belongs to $\G{0}$. As in~\eqref{EquationCommutingDiffOpAndMollifierWithoutNormalDerivatives}
  we compute
  \begin{equation*}
   L J_{\frac{1}{n}} \partial^\alpha u = J_{\frac{1}{n}} f_\alpha + \sum_{j = 0}^2 [A_j , J_{\frac{1}{n}}] \partial_j \partial^\alpha u + [D, J_{\frac{1}{n}}] \partial^\alpha u
  \end{equation*}
  for all $n \in \N$.
  As $f_\alpha$ is an element of $\Ltwoa$, we have 
  \begin{equation}
  \label{EquationConvergenceToftildealpha}
   J_{\frac{1}{n}} f_\alpha \longrightarrow f_\alpha
  \end{equation}
  in $\Ltwoa$ as $n \rightarrow \infty$. Arguing as in~\eqref{EquationCommutatorEstimateForAj} and~\eqref{EquationCommutatorConvergingToZeroAj}, 
  we further derive
  \begin{equation}
  \label{EquationConvergenceToZeroCommutatorAjJepsPartialalphau}
   \sum_{j = 0}^2 [A_j , J_{\frac{1}{n}}] \partial_j \partial^\alpha u + [D,J_{\frac{1}{n}}] \partial^\alpha u \longrightarrow 0
  \end{equation}
  in $\Ltwoa$ as $n \rightarrow \infty$ since $u$ belongs to $\Ha{m}$.
  Since $g$ belongs to $\E{m}$ and  $u_0$ to $\Hh{m}$, the functions $B J_{\frac{1}{n}} \partial^{\alpha} u = J_{\frac{1}{n}} \partial^{\alpha} g$ and 
  $J_{\frac{1}{n}} \partial^{\alpha} u_0$ tend to $\partial^{\alpha} g$ in $\E{0}$ respectively to $\partial^\alpha u_0$ in $\Ltwoh$ as $n \rightarrow \infty$.
  Applying Lemma~\ref{LemmaEllerResult}, we get a constant $C_0$ and a number $\gamma > 0$ 
  such that
  \begin{align*}
   &\Gnorm{0}{J_{\frac{1}{n}} \partial^\alpha u - J_{\frac{1}{k}} \partial^\alpha u}^2 \leq C_0 \Ltwohn{J_{\frac{1}{n}} \partial^\alpha u_0 - J_{\frac{1}{k}} \partial^\alpha u_0}^2 \\
   &\quad + C_0 \Enorm{0}{B J_{\frac{1}{n}} \partial^\alpha u  - B J_{\frac{1}{k}} \partial^\alpha u}^2 + \frac{C_0}{\gamma} \Ltwoan{L J_{\frac{1}{n}} \partial^\alpha u - L J_{\frac{1}{k}} \partial^{\alpha} u}^2
  \end{align*}
  for all $n,k \in \N$. We conclude that 
  $(J_{\frac{1}{n}} \partial^\alpha u)_n$ is a Cauchy sequence in $\G{0}$. As $(J_{\frac{1}{n}} \partial^\alpha u)_n$ 
  converges to $\partial^\alpha u$ in $\Ltwoa$, we obtain that 
  $\partial^\alpha u$ is an element of $\G{0}$.
\end{proof}

In the next result we show how to gain one derivative in time. We study the initial boundary value problem 
formally solved by $\partial_t u$. The time integral of the solution of this problem then 
coincides with $u$. Here one sees explicitly where the compatibility conditions are needed.
\begin{lem}
 \label{LemmaBasicRegularityInTime}
 Let $\eta > 0$. Take coefficients $A_0 \in \Fupdwl{3}{cp}{\eta}$, $A_1, A_2 \in \Fcoeff{3}{cp}$, $A_3 = A_3^{\operatorname{co}}$,
 $D \in \Fuwl{3}{cp}$, and $B = B^{\operatorname{co}}$. Choose data $u_{0} \in \Hh{1}$, $g \in  \E{1}$,
 and $f \in \Ha{1}$. Assume that the tuple $(0, A_0,\ldots, A_3, D, B, f, g,  u_0)$ 
 fulfills the compatibility conditions~\eqref{EquationCompatibilityConditionPrecised} of order $1$. 
 Let $u \in C(\clJ, \Ltwoh)$ be the weak solution of~\eqref{IBVP} with differential operator $L=L(A_0,\ldots, A_3, D)$, inhomogeneity $f$,
 boundary value $g$, and initial value $u_{0}$. Assume that $u \in C^1(\overline{J'}, \Ltwoh)$ implies 
 $u \in G_1(J' \times \R^3_+)$ for every open interval $J' \subseteq J$. 
 Then $u$ belongs to $\G{1}$.
\end{lem}
\begin{proof}
 Without loss of generality we assume $J = (0,T)$. Take $r > 0$ such that
 \begin{align}
 \label{EquationBoundingCoefficientsForAllt}
  &\Fnorm{3}{A_i} \leq r, \quad \Fnorm{3}{D} \leq r, \nonumber\\
  &\max \{\Fvarnorm{2}{A_i(t)},\max_{1 \leq j \leq 2} \Hhn{2-j}{\partial_t^j A_0(t)}\} \leq r, \nonumber\\
  &\max \{\Fvarnorm{2}{D(t)},\max_{1 \leq j \leq 2} \Hhn{2-j}{\partial_t^j D(t)}\} \leq r
 \end{align}
 for all $t \in \clJ$ and $i \in \{0, 1, 2\}$.
 Let $\gamma = \gamma(\eta, r,T)$ be defined by 
 \begin{align*}
  \gamma = \max\{\gamma_{\ref{LemmaEllerResult};0}, \gamma_{\ref{TheoremAPrioriEstimates};1} \} \geq 1,
 \end{align*}
 where $\gamma_{\ref{LemmaEllerResult};0} = \gamma_{\ref{LemmaEllerResult};0}(\eta,r)$ and 
 $\gamma_{\ref{TheoremAPrioriEstimates};1} = \gamma_{\ref{TheoremAPrioriEstimates};1}(\eta,r,T)$ are the 
 corresponding constants from Lemma~\ref{LemmaEllerResult} and Theorem~\ref{TheoremAPrioriEstimates} respectively.
 We further introduce the constant $C_0 = C_0(\eta,r,T)$ by 
 \begin{align*}
  C_0 = \max\{C_{\ref{LemmaEllerResult};0,0}, C_{\ref{LemmaEllerResult};0,1}, C_{\ref{LemmaEllerResult};0},  C_{\ref{TheoremAPrioriEstimates};1}, 
      (C_{\ref{TheoremAPrioriEstimates};1,0} + T C_{\ref{TheoremAPrioriEstimates};1}) e^{C_{\ref{TheoremAPrioriEstimates};1} T}, 
      C_{\ref{LemmaEstimatesForHigherOrderInitialValues};1,1}\} \geq 1,
 \end{align*}
 where again $C_{\ref{LemmaEllerResult};0,0} = C_{\ref{LemmaEllerResult};0,0}(\eta,r)$, 
  $C_{\ref{LemmaEllerResult};0} = C_{\ref{LemmaEllerResult};0}(\eta,r)$,
 $C_{\ref{TheoremAPrioriEstimates};1} = C_{\ref{TheoremAPrioriEstimates};1}(\eta,r,T)$,
 and $C_{\ref{LemmaEstimatesForHigherOrderInitialValues};1,1} = C_{\ref{LemmaEstimatesForHigherOrderInitialValues};1,1}(\eta,r)$
 are the corresponding constants from Lemma~\ref{LemmaEllerResult}, Theorem~\ref{TheoremAPrioriEstimates}, 
 and Lemma~\ref{LemmaEstimatesForHigherOrderInitialValues} respectively.
 Finally, we set 
 \begin{align*}
  R_1 = C_0 e^{2 \gamma T} (\Gnorm{0}{f}^2 + \Hangamma{1}{f}^2 + \Enorm{1}{g}^2 + \Hhn{1}{u_0}^2).
 \end{align*}
 
  I) Take  $t_0 \in \clJ$ and assume that $u(t_0) \in \Hh{1}$ with $\Hhn{1}{u(t_0)}^2 \leq R_1$. 
  We show the existence of a time step $T_s > 0$ and a function $v \in C([t_0, T_s'], \Ltwoh)$ satisfying
 \begin{equation}
 \label{EquationForFixedPointRegularityOrder1}
  \left\{ \begin{aligned}
    L_{\partial_t} v   &= \partial_{t} f - \partial_t D \Big(\int_{t_0}^t v(s) ds +  u(t_0) \Big), \quad &&x \in \R^3_+,  && t \in J'; \\
    B v &= \partial_t g, \quad &&x \in \partial \R^3_+,  &&t \in J'; \\
    v(t_0) &= S_{1,1}(t_0, A_0, \ldots, A_3,D, f, u(t_0)), \quad &&x \in \R^3_+;
  \end{aligned} \right.
 \end{equation}
 where we abbreviate
 \begin{equation*} 
   L_{\partial_t} = L(A_0, \ldots, A_3, \partial_t A_0 + D)
 \end{equation*}
 and define $T_s' := \min\{t_0 + T_s, T\}$ and $J' := (t_0, T_s')$.
 Recall that the function $S_{1,1}(t_0, A_0, \ldots, A_3, D, f, u(t_0))$ belongs to $\Ltwoh$ by Lemma~\ref{LemmaEstimatesForHigherOrderInitialValues}.
 
 Take a number $T_s \in (0,T)$ to be fixed below and define $J'$ and $T_s'$ as above. We further set $\Omega' = J' \times \R^3_+$.
 Let $w \in C(\overline{J'}, \Ltwoh)$. Note that $\partial_t A_0 + D$  
 and $\partial_t D$ still belong to 
 $L^\infty(\Omega)$. Hence, the problem
 \begin{equation*}
  \left\{ \begin{aligned}
    L_{\partial_t} v   &= \partial_{t} f - \partial_t D \Big(\int_{t_0}^t w(s) ds +  u(t_0) \Big), \quad &&x \in \R^3_+, \quad &&t \in J'; \\
    B v &= \partial_t g, \quad &&x \in \partial \R^3_+, &&t \in J'; \\
    v(t_0) &= S_{1,1}(t_0, A_0, \ldots, A_3, D, f, u(t_0)), &&x \in \R^3_+,
  \end{aligned} \right.
 \end{equation*}
 has a unique solution $\Phi(w)$ in $C(\overline{J'}, \Ltwoh)$ by Lemma~\ref{LemmaEllerResult}.
 We next define
 \begin{align}
 \label{EquationIntroducingFixedPointSpaceForTimeRegularity}
  B_R = \{v \in C(\overline{J'}, \Ltwoh)\colon  \GnormPrime{0}{v} \leq R\},
 \end{align}
 where $R > 0$ will be fixed below. Equipped with the metric induced by the $\Ggamma{0}$-norm this is a complete 
 metric space. Let $w \in B_R$. Employing 
 H\"older's and Minkowski's inequality, Lemma~\ref{LemmaEllerResult}, and the bound
 \begin{align*}
 \Ltwohn{S_{1,1}(t_0, A_0,\ldots, A_3, D, f, u(t_0))}^2 &\leq 2 C^2_{\ref{LemmaEstimatesForHigherOrderInitialValues};1,1} (\Ltwohn{f(t_0)}^2 + \Hhn{1}{u(t_0)}^2) \\
 &\leq 4 C_0^2 R_1
 \end{align*}
  from Lemma~\ref{LemmaEstimatesForHigherOrderInitialValues}, we estimate
 \begin{align}
 \label{EquationSelfMappingForOrder1}
  &\GnormPrime{0}{\Phi(w)}^2 \leq C_0 \Big\| \partial_t f - \partial_t D  \int_{t_0}^t w(s) ds - \partial_t D\, u(t_0)  \Big\|_{L^2_\gamma(\Omega')}^2  \\
  &\quad + C_0 \|\partial_t g\|_{E_{0,\gamma}(J' \times \partial \R^3_+)}^2 + C_0 \Ltwohn{S_{1,1}(t_0, A_0, \ldots, A_3, D, f, u(t_0))}^2 \nonumber\\
      &\leq 2 C_0 r^2 \int_{t_0}^{T_s'} e^{-2\gamma t} \Big(\int_{t_0}^t \Ltwohn{w(s)} ds + \Ltwohn{u(t_0)} \Big)^2 dt  + 4 (1 + C_0^2) R_1 \nonumber\\
      &\leq 4 C_0 r^2 T_s \GnormPrime{0}{w}^2 + 4 C_0 r^2 T_s R_1 + 8 C_0^2  R_1. \nonumber
 \end{align}
 We now set 
 \begin{align*}
  R = (18 C_0^2 R_1)^{1/2}
 \end{align*}
 in~\eqref{EquationIntroducingFixedPointSpaceForTimeRegularity} and choose $T_s \in (0,T)$ so small that 
 \begin{align*}
  4 C_0 r^2 T_s \leq \frac{1}{2}.
 \end{align*} 
 We point out that $T_s$ is independent of $t_0$. Using~\eqref{EquationIntroducingFixedPointSpaceForTimeRegularity}
 and this choice of $R$ and $T_s$, we obtain from~\eqref{EquationSelfMappingForOrder1}
 \begin{align*}
  \GnormPrime{0}{\Phi(w)}^2 \leq \frac{R^2}{2} + \frac{R^2}{2} = R^2
 \end{align*}
 for all $w \in B_R$, i.e., $\Phi(B_R) \subseteq B_R$. Moreover, 
 Lemma~\ref{LemmaEllerResult} implies that
 \begin{align*}
  &\GnormPrime{0}{\Phi(w_1) - \Phi(w_2)}^2 \leq C_0 \Big\|\partial_t D \int_{t_0}^t (w_1(s) - w_2(s)) ds  \Big\|_{L^2_\gamma(\Omega')}^2 \\
      &\leq C_0 \|\partial_t D\|_{L^\infty(\Omega')}^2 T_s \GnormPrime{0}{w_1 - w_2}^2 \leq C_0 r^2 T_s \GnormPrime{0}{w_1 - w_2}^2 \\
      &\leq \frac{1}{2} \GnormPrime{0}{w_1 - w_2}^2
 \end{align*}
 for all $w_1, w_2 \in B_R$. The contraction mapping principle thus gives a unique $v \in B_R$ with 
 $\Phi(v) = v$ on $J'$, i.e., $v$ is the asserted solution of~\eqref{EquationForFixedPointRegularityOrder1}.
 
 II) In this step we assume that $u(t_0)$ belongs to $\Hh{1}$ with $\Hhn{1}{u(t_0)}^2 \leq R_1$ 
 and that $(t_0, A_0, \ldots, A_3, D, f,g, u(t_0))$ fulfills the compatibility 
 conditions~\eqref{EquationCompatibilityConditionPrecised} of order one; i.e., $\tr (B u(t_0)) = g(t_0)$. 
 
 Let $J'$ be defined as in step I) and let $v$ be the solution 
 of~\eqref{EquationForFixedPointRegularityOrder1} constructed in step I). 
 A straightforward computation shows that $A_0 v$ has a weak time derivative in 
 $L^2(J', \Hh{-1})$ and
  \begin{align*}
   \partial_t(A_0 v) = \partial_t f - \partial_t D \Big(\int_{t_0}^t v(s) ds + u(t_0) \Big) -\sum_{j=1}^3 A_j \partial_j v -D v,
  \end{align*}
  see~\cite[Lemma~4.7]{SpitzDissertation} for details. We set 
  \begin{align*}
   w(t) = u(t_0) + \int_{t_0}^t v(s) ds
  \end{align*}
  for all $t \in \overline{J'}$. Observe that $w$ belongs to $C^1(\overline{J'}, \Ltwoh)$ with $w(t_0) = u(t_0)$. 
  Employing~\eqref{EquationForFixedPointRegularityOrder1} and~\eqref{EquationDefinitionSmp} we then compute in $\Hh{-1}$
  \begin{align*}
   &L w(t) = (A_0 v)(t) + \int_{t_0}^t \Big(\sum_{j = 1}^3 A_j \partial_j v(s) \Big) ds 
   + \sum_{j = 1}^3 A_j \partial_j u(t_0) + (D w)(t) \\
      &= \int_{t_0}^t \Big( \partial_t (A_0 v)(s)+ \sum_{j = 1}^3 A_j \partial_j v(s) \Big) ds + (D w)(t)  + (A_0 v) (t_0) + \sum_{j = 1}^3 A_j \partial_j u(t_0) \\
      &= \int_{t_0}^t (\partial_t f(s) - (\partial_t D w)(s) - (D v)(s))ds + (D w)(t) \\
      &\quad + A_0(t_0) S_{1,1}(t_0,A_0,\ldots, A_3, D,f,u(t_0)) + \sum_{j = 1}^3 A_j \partial_j u(t_0)  \\
      &= f(t) - f(t_0) - \int_{t_0}^t \partial_t (D w)(s) ds + D w(t) + f(t_0) - (D w)(t_0) = f(t)
  \end{align*}
  for all $t \in \overline{J'}$. In particular, $L(A_0,\ldots, A_3, D)w$ belongs to $\Ltwoa$. 
  
  To compute the trace of $B w$ on $\Gamma' = J' \times \partial \R^3_+$, we stress that 
  $\Tr(B v) = \partial_t g$ on $\Gamma'$ by~\eqref{EquationForFixedPointRegularityOrder1}.
  Moreover, the trace operator $\Tr$ commutes with integration in time here, see~\cite[Corollary~2.18]{SpitzDissertation} for the proof.
  Since $(t_0, A_0, \ldots, A_3, D, f,g, u(t_0))$ fulfills the 
  compatibility conditions of order one, we thus infer
  \begin{align*}
   \Tr (B w)(t) &= \Tr \Big( B \int_{t_0}^t v(s) ds + B u(t_0)\Big) =\int_{t_0}^t \Tr (B v)(s) ds + \tr(B u(t_0)) \\
    &= \int_{t_0}^t \partial_t g(s) ds + g(t_0) = g(t)
  \end{align*}
  for all $t \in \overline{J'}$.
  The function $w \in C^1(\overline{J'}, \Ltwoh)$ consequently solves~\eqref{IBVP} on $\Omega'$ with initial value $u(t_0)$ 
  at initial time $t_0$. 
  As $u$ also solves~\eqref{IBVP} on $\Omega'$ with inhomogeneity $f$, boundary value $g$, 
  and initial 
  value $u(t_0)$ in $t_0$, the uniqueness statement in Lemma~\ref{LemmaEllerResult} yields 
  $u = w$ on $\Omega'$. (Here we use the obvious variant of the lemma for the initial time $t_0$.)
  We conclude that $u$ is an element of $C^1(\overline{J'}, \Ltwoh)$.
  The assumptions therefore tell us that $u$ belongs to $\GPrime{1}$.
  
  III) We next consider $t_0 = 0$. Since $u(0) = u_0 \in \Hh{1}$, $\Hhn{1}{u_0}^2 \leq R_1$, and 
  $(0, A_0, \ldots, A_3, D, B, f, g, u_0)$ fulfills the compatibility conditions of first order by assumption, step II) shows
  that $u$ belongs to 
  $G_1((0, T_0) \times \R^3_+)$, where we set $T_0 = \min\{T_s, T\}$.
  If $T_0 = T$ we are done. Otherwise, we apply Theorem~\ref{TheoremAPrioriEstimates} to obtain
  \begin{align*}
   \Gnorm{1}{u}^2 &\leq C_0 ( \Gnorm{0}{f}^2  + \Hhn{1}{u_0}^2 + \Enorm{1}{g}^2 + \frac{1}{\gamma} \Hangamma{1}{f}^2) \\
      &\leq e^{-2\gamma T} R_1.
  \end{align*}
  We conclude that $\Hhn{1}{u(T_0)}^2 \leq R_1$. Moreover,  $(T_0, A_0, \ldots, A_3, D, B, f, g, u(T_0))$ fulfills the compatibility 
  conditions of first order by~\eqref{EquationDifferentiatedBoundaryCondition}
  since $u$ is a solution in $G_1(J' \times \R^3_+)$. We can therefore  apply 
  step II) with $t_0 = T_0$. We see that $u$ belongs to $G_1((T_0, T_1) \times \R^3_+)$, with 
  $T_1 = \min\{T, T_0 + T_s\}$. Since 
  \begin{align*}
   \partial_t u_{|[0,T_0]}(T_0) = S_{1,1}(T_0, A_0, \ldots, A_3, D,f,u(T_0)) = \partial_t u_{|[T_0, T_1]}(T_0),
  \end{align*}
  we infer
  $u \in G_1((0, T_1) \times \R^3_+)$. In this way we iterate. Since the time step $T_s$ does not depend on $t_0$, we 
  are done after finitely many steps. We conclude that $u$ is an element of $G_1((0,T) \times \R^3_+)$.
\end{proof}


We want to iterate the previous result in order to deduce higher order regularity. 
To that purpose we need a relation between the operators $S_{m,p}$ of different order, which is 
stated in the next lemma.
Its assertion follows inductively from the definition of the operators $S_{m,p}$ 
and a straightforward computation. We refer to~\cite[Lemma~4.8]{SpitzDissertation} for the details.
\begin{lem}
 \label{LemmaHigherOrderCompatibilityConditions}
 Let  $\eta > 0$, $m \in \N$ and $\tilde{m} = \max\{m,3\}$. Take $A_0 \in \Fupdwl{\max\{m+1,3\}}{cp}{\eta}$ with 
 $\partial_t A_0 \in \Fuwl{\tilde{m}}{cp}$ and $D \in \Fuwl{\max\{m+1,3\}}{cp}$. Let $A_1$, $A_2 \in \Fcoeff{\max\{m+1,3\}}{cp}$, 
 $A_3 = A_3^{\operatorname{co}}$, and $B = B^{\operatorname{co}}$.
 Choose $t_0 \in \clJ$, 
 $u_0 \in \Hh{m+1}$, $g \in \E{m+1}$, and $f \in \Ha{m+1}$. 
 Assume that $u \in \G{m}$ solves~\eqref{IBVP} with differential operator $L(A_0,\ldots, A_3, D)$, 
 inhomogeneity $f$, boundary value $g$, and initial value $u_0$.
 Set $u_1 = S_{m+1,1}(t_0, A_0, \ldots,A_3, D, f, u_0)$ and $f_1 = \partial_t f - \partial_t D u$.
 Then
 \begin{align*}
  S_{m, p}(0,A_0, \ldots, A_3, \partial_t A_0 + D, f_1, u_1) = S_{m+1, p+1}(0,A_0, \ldots, A_3, D, f,u_0)
 \end{align*}
 for all $p \in \{0, \ldots, m-1\}$.
\end{lem}

The combination of the previous results with an iteration argument then yields the desired 
regularity of the solution $u$ provided the coefficients are additionally elements of 
$C^{\infty}(\overline{\Omega})$.
\begin{prop}
 \label{PropositionRegularityForApproximatingProblem}
 Let $\eta > 0$, $m \in \N$, and $\tilde{m} = \max\{m,3\}$. 
 Choose coefficients 
 $A_0 \in \Fupdwl{\tilde{m}}{cp}{\eta}$, $A_1, A_2 \in \Fcoeff{\tilde{m}}{cp}$, $A_3 = A_3^{\operatorname{co}}$,
  $D \in \Fuwl{\tilde{m}}{cp}$, and $B = B^{\operatorname{co}}$. Assume that these coefficients are  
  contained in $C^\infty(\overline{\Omega})$.
 Take data $f \in \Ha{m}$, $g \in \E{m}$, and $u_0 \in \Hh{m}$ such that the tuple $(0, A_0, \ldots, A_3, D, B, f, g, u_0)$ 
 satisfies the compatibility conditions~\eqref{EquationCompatibilityConditionPrecised} 
 of order $m$, i.e.,
 \begin{align*}
  \Tr (B S_{m,l}(0,A_0,\ldots, A_3, D,f,u_0)) = \partial_t^l g(0) \quad \text{for } 0 \leq l \leq m - 1.
 \end{align*}
 Let $u$ be the weak solution of~\eqref{IBVP} with differential operator $L = L(A_0,\ldots, A_3, D)$, inhomogeneity $f$, 
 boundary value $g$, and initial value $u_0$. 
 Then $u$ belongs to $\G{m}$.
\end{prop}
\begin{proof}
 The assertion is true for $m = 1$ by Lemma~\ref{LemmaBasicRegularityInTime}, Lemma~\ref{LemmaRegularityInSpaceTangential}, 
 and Lemma~\ref{LemmaRegularityInNormalDirection}. Now assume that we have shown 
 the assertion for  a number $m \in \N$. Let all the assumptions be fulfilled for $m+1$. By the induction hypothesis,
 the weak solution $u$ of~\eqref{IBVP} belongs to $\G{m}$. Moreover, $\partial_t u$ 
 solves the initial boundary value problem
 \begin{equation*}
 \left\{\begin{aligned}
    L_{\partial_t} v   &= \partial_{t} f - \partial_t D u, \quad &&x \in \R^3_+, \quad &&t \in J; \\
    B v &= \partial_t g, \quad &&x \in \partial \R^3_+, &&t \in J; \\
    v(0) &= S_{m+1,1}(0, A_0, \ldots, A_3, D, f, u_0),  &&x \in \R^3_+,
  \end{aligned} \right.
\end{equation*}
 where we again write $L_{\partial_t}$ for $L(A_0, \ldots, A_3, \partial_t A_0 + D)$.
 Using the abbreviations $u_1$ for $S_{m+1,1}(0, A_0,\ldots, A_3, D, f, u_0)$ and $f_1$ for $\partial_t f - \partial_t D u$ once more,
 we deduce that $u_1$ is contained in $\Hh{m}$ by Lemma~\ref{LemmaEstimatesForHigherOrderInitialValues}, that $\partial_t g$ belongs
 to $\E{m}$, and that $f_1$ is an element of $\Ha{m}$ by  Lemma~\ref{LemmaRegularityForA0}~\ref{ItemProductHkappa} 
 since $\partial_t D \in \G{\max\{m,2\}}$ and 
 $u \in \G{m}$.
 Lemma~\ref{LemmaHigherOrderCompatibilityConditions} further shows that $(0, A_0, \ldots, A_3, \partial_t A_0 + D, f_1, \partial_t g, u_1)$
 fulfills the compatibility conditions~\eqref{EquationCompatibilityConditionPrecised}
 of order $m$.
 Finally, we have $A_0 \in \Fupdwl{\tilde{m}}{cp}{\eta} \cap C^\infty(\overline{\Omega})$ with $\partial_t A_0 \in \Fuwl{\tilde{m}}{cp}$ and 
 $\partial_t A_0 + D \in \Fuwl{\tilde{m}}{cp} \cap C^\infty(\overline{\Omega})$ so that the induction hypothesis yields
 that  $\partial_t u$ is an element of $\G{m}$, implying that $u$ is an element of
 $\bigcap_{j=1}^{m+1} C^j(\clJ, \Hh{m+1-j})$. 
 By Lemma~\ref{LemmaRegularityInSpaceTangential} and Lemma~\ref{LemmaRegularityInNormalDirection}, 
 the solution $u$ then belongs to $\G{m+1}$.
\end{proof}

It remains to remove the assumption of smooth coefficients. We therefore want to approximate the coefficients 
from $\F{m}$ by smooth ones. However, approximating the coefficients will violate the compatibility conditions 
in general. We overcome this difficulty by not only approximating the coefficients but also the 
initial value in such a way, that the tuple consisting of the approximating coefficients and data 
still satisfies the compatibility conditions up to order $m$.
\begin{lem}
 \label{LemmaExistenceOfApproximatingSequence}
 Let $\eta > 0$, $m \in \N$, and $\tilde{m} = \max\{m,3\}$. Take coefficients $A_0 \in \Fupdwl{\tilde{m}}{cp}{\eta}$, $A_1, A_2 \in \Fcoeff{\tilde{m}}{cp}$,
 $A_3 = A_3^{\operatorname{co}}$, $ D \in \Fuwl{\tilde{m}}{cp}$, and $B = B^{\operatorname{co}}$ and data $f \in \Ha{m}$, $g \in \E{m}$,
 and $u_0 \in \Hh{m}$ which fulfill the
 compatibility conditions~\eqref{EquationCompatibilityConditionPrecised} of order $m$ in $t_0 \in \clJ$, i.e.,
 \begin{align*}
  \Tr B S_{m,l}(t_0, A_0,\ldots, A_3, D, f, u_0) = \partial_t^l g(t_0) \quad \text{for } 0 \leq l \leq m - 1.
 \end{align*}
 Let $\{A_{i,\epsilon}\}_{\epsilon > 0}$ and $\{D_{\epsilon}\}_{\epsilon > 0}$ be the families of functions 
 provided by Lemma~\ref{LemmaApproximationOfCoefficients} for $A_i$ and $D$ respectively for $i \in \{0,1,2\}$.
 Then there exists a number $\epsilon_0 > 0$ and a family  $\{u_{0,\epsilon}\}_{0 < \epsilon <\epsilon_0}$ in $\Hh{m}$ such that the 
 compatibility conditions for $(t_0, A_{0,\epsilon}, A_{1,\epsilon}, A_{2,\epsilon}, A_3, D_\epsilon, B, f,g, u_{0,\epsilon})$ of order $m$ 
 hold; i.e.,
 \begin{align*}
  \Tr B S_{m,l}(t_0, A_{0,\epsilon},A_{1,\epsilon}, A_{2,\epsilon}, A_3, D_\epsilon, f, u_{0,\epsilon}) = \partial_t^l g(t_0) \quad \text{for } 0 \leq l \leq m - 1,
 \end{align*}
  and
 $u_{0,\epsilon} \rightarrow u_0$ in $\Hh{m}$ as $\epsilon \rightarrow 0$.
\end{lem}
\begin{proof}
 Without loss of generality we assume $t_0 = 0$. Note that $A_{1,\epsilon}$ and $A_{2,\epsilon}$ are 
 still time independent for all $\epsilon > 0$. We set $u_{0,\epsilon} = u_0 + h_\epsilon$ and look for $h_\epsilon \in \Hh{m}$ with $h_\epsilon \rightarrow 0$ in 
 $\Hh{m}$  such that the compatibility conditions are fulfilled. Since 
 $B = M A_3$ for a constant matrix $M = M^{\operatorname{co}}$, it is 
 sufficient for that purpose to find $h_\epsilon$ with
 \begin{align*}
  A_3 S_{m,p}(0, A_{0,\epsilon}, A_{1,\epsilon}, A_{2,\epsilon}, A_3, D_\epsilon, f, u_0 + h_\epsilon)  = A_3 S_{m,p}(0, A_0,\ldots, A_3, D, f, u_0)
 \end{align*}
 for all $0 \leq p \leq m-1$ on $\partial \R^3_+$. To simplify the notation, we will drop the dependancy of 
 the operators on $0$, $A_3$ and $f$ in the following since they remain fixed throughout the proof.
 
 I) The definition of the operators $S_{m,k}$ was given inductively. In principle, it is possible 
 to derive an explicit representation of $S_{m,k}$. However, we are satisfied with the representation
 \begin{align}
  \label{EquationClaimedFormOfCC}
  S_{m,p}(A_0, A_1, A_2, D,u_0) &= (-A_{0}(0)^{-1}A_3)^p \partial_3^p u_0 + \sum_{j=0}^{p-1} C_{p,p-j}(A_0,A_1,A_2, D) \partial_3^j u_0 \nonumber\\
		      &\qquad + B_p(A_0, A_1, A_2, D) f ,
 \end{align}
 where 
 $C_{p, p-j}$ is 
 a differential operator which only involves tangential derivatives up to order $p-j$ 
 and which maps $\Hh{m-j}$ continuously into $\Hh{m-p}$ with
 \begin{equation*}
  \|C_{p,p-j}(A_{0,\epsilon}, A_{1,\epsilon}, A_{2,\epsilon}, D_{\epsilon})\|_{\Hh{m-j} \rightarrow \Hh{m-p}} \leq C
 \end{equation*}
 for all $\epsilon \geq 0$ and $j \in \{0, \ldots, p-1\}$. Here $A_{i,0}$ means $A_i$ for $i \in \{0,1,2\}$ 
 and accordingly for $D$.
 Similarly, $B_p$ is a differential operator of order $p - 1$ which maps $\Ha{m}$ 
 continuously into $\Hh{m-p}$.
 
 For the proof of this claim one proceeds by induction with respect to $p$, inserting the 
 representation~\eqref{EquationClaimedFormOfCC} for the lower order terms into the 
 definition of $S_{m,p}$. A very careful analysis of the regularity of 
 the arising coefficients then yields the mapping properties of $C_{p,p-j}$ and $B_p$. 
 Similarly, an induction shows that $S_{m,p}(A_{0,\epsilon}, A_{1,\epsilon}, A_{2,\epsilon}, D_{\epsilon},u_0)$ 
 converges to $S_{m,p}(A_0, A_1, A_2, D,u_0)$ in $\Hh{m-p}$ as $\epsilon \rightarrow 0$ for 
 $0 \leq p \leq m-1$.
 The details can be found in~\cite[Lemma~4.10]{SpitzDissertation}.

 II) Let $h \in \Hh{m}$. By means of~\eqref{EquationClaimedFormOfCC}, we have
 \begin{align}
 \label{EquationConstructionOfCCFormalComputation}
  &S_{m,p}( A_{0,\epsilon}, A_{1,\epsilon}, A_{2,\epsilon},  D_\epsilon, u_0 + h) 
  =  S_{m,p}( A_{0,\epsilon}, A_{1,\epsilon}, A_{2,\epsilon}, D_\epsilon,u_0)  \nonumber\\
		      &\hspace{2em} + (-A_{0,\epsilon}(0)^{-1}A_3)^p \partial_3^p h + \sum_{j=0}^{p-1} C_{p,p-j}(A_{0,\epsilon}, A_{1,\epsilon}, A_{2,\epsilon}, D_\epsilon) \partial_3^j  h. 
 \end{align}
 Set $a_0^\epsilon = 0$. Then $a_0^\epsilon \in \Hh{m}^6$ and 
 \begin{align*}
  S_{m,0}(A_0,A_1,A_2, D,u_0) - S_{m,0}(A_{0,\epsilon},A_{1,\epsilon}, A_{2,\epsilon}, D_\epsilon,u_0) = u_0 - u_0 = 0 = a_0^\epsilon.
 \end{align*} 
 Let $k \in \{0, \ldots, m-2\}$. Assume that we have constructed families of functions $a_p^\epsilon \in \Hh{m-p}^6$ such that 
 \begin{align}
 \label{EquationConstructionOfCCInductionStep}
  &A_3 \,\big((-A_{0,\epsilon}(0)^{-1}A_3)^p  a_p^\epsilon\big) \\
    &= A_3 \Big(S_{m,p}(A_0,A_1, A_2, D,u_0) - S_{m,p}(A_{0,\epsilon}, A_{1,\epsilon}, A_{2,\epsilon}, D_\epsilon, u_0) \Big) \nonumber \\
      &\quad - A_3 \Big(\sum_{j=0}^{p-1} C_{p,p-j}(0, A_{0,\epsilon}, A_{1,\epsilon}, A_{2,\epsilon}, D_\epsilon)   a_j^\epsilon\Big) , \nonumber\\
     &a_p^\epsilon \longrightarrow 0 \quad \text{in } \Hh{m-p}^6 \text{ as } \epsilon \rightarrow 0 \nonumber
 \end{align}
  for every $p \in \{0, \ldots, k\}$.
  Then the functions
  \begin{align*}
   &\sum_{j=0}^{k} C_{k+1,k+1-j}(A_{0,\epsilon}, A_{1,\epsilon}, A_{2,\epsilon}, D_\epsilon)  a_j^\epsilon, \\
   &S_{m,k+1}(A_0,A_1, A_2, D,u_0) - S_{m,k+1}(A_{0,\epsilon},A_{1,\epsilon}, A_{2,\epsilon}, D_\epsilon,u_0)
  \end{align*}
  belong to $\Hh{m-k-1}$ and converge to zero in this space by step I)
  as $\epsilon \rightarrow 0$.  
  Lemma~\ref{LemmaConstructionForCC} below thus gives a number $\epsilon_0 > 0$ and 
  functions $a_{k+1}^\epsilon \in \Hh{m-k-1}^6$ 
  such that
  \begin{align*}
   &A_3\,\big((-A_{0,\epsilon}(0)^{-1}A_3)^{k+1}  a_{k+1}^\epsilon\big) \\
   &= A_3 \Big(S_{m,k+1}(A_0,A_1, A_2, D,u_0) - S_{m,k+1}(A_{0,\epsilon}, A_{1,\epsilon}, A_{2,\epsilon}, D_\epsilon, u_0) \Big) \\
      &\quad - A_3  \Big(\sum_{j=0}^{k} C_{k+1,k+1-j}(A_{0,\epsilon}, A_{1,\epsilon}, A_{2,\epsilon}, D_\epsilon)  a_j^\epsilon \Big) , \\
   &a_{k+1}^\epsilon \longrightarrow 0 \quad \text{in } \Hh{m-k-1}^6 \text{ as } \epsilon \rightarrow 0.
  \end{align*}
  for all $\epsilon \in (0, \epsilon_0)$. The induction is thus finished.
  We next define
  \begin{align*}
   b_p^\epsilon := a_p^\epsilon(\cdot, 0) \in H^{m-p-\frac{1}{2}}(\partial \R^3_+)
  \end{align*}
  for $0 \leq p \leq m-1$. Since the trace operator from $\Hh{m-p}$ into $H^{m-p-\frac{1}{2}}(\partial \R^3_+)$ 
  is continuous, we infer that $b_p^\epsilon \rightarrow 0$ in $H^{m-p-\frac{1}{2}}(\partial \R^3_+)$ as 
  $\epsilon \rightarrow 0$. Theorems~2.5.7 and~2.5.6 in~\cite{Hoermander} now yield functions $h_\epsilon \in \Hh{m}$ 
  with 
  \begin{align*}
   \partial_3^p h_\epsilon(\cdot, 0) = b_p^\epsilon \quad \text{on } \partial \R^3_+
  \end{align*}
  for $0 \leq p \leq m-1$ and $\epsilon \in (0, \epsilon_0)$, which satisfy $h_\epsilon \rightarrow 0$ in $\Hh{m}$ as $\epsilon \rightarrow 0$. 
  
  We set $u_{0,\epsilon} = u_0 + h_\epsilon$ for all $\epsilon > 0$. Then $u_{0,\epsilon}$ tends to $u_0$ in $\Hh{m}$ and
  by construction we have
  \begin{align*}
   &\tr \big( A_3 S_{m,p}(A_{0,\epsilon},A_{1,\epsilon}, A_{2,\epsilon}, D_\epsilon, u_{0,\epsilon})\big)= \tr \big(A_3 S_{m,p}(A_0,A_1,A_2,D,u_0)\big)
  \end{align*}
  for $0 \leq p \leq m-1$. Since $(0,A_0,\ldots, A_3, D,B,f,g,u_0)$ fulfills the compatibility conditions~\eqref{EquationCompatibilityConditionPrecised}
  of order $m$ and $B = M A_3 $, we conclude the assertion.
\end{proof}

In the proof of the previous result we exploited that we can continuously
invert $(-A_{0,\epsilon}(0)^{-1} A_3)^p$ on the range of $A_3$ in a certain sense. We provide the proof 
of this statement in the next lemma.
\begin{lem}
 \label{LemmaConstructionForCC} 
 Let $\eta > 0$ and $m \in \N$ with $m \geq 3$. Take $A_0 \in \Fpdk{m}{\eta}{6}$ and $A_3 = A_3^{\operatorname{co}}$.
 Pick $k\in \N$ with $k \leq m-1$ and $p \in \N_0$. Choose $r > 0$ such that $\Fvarnorm{m-1}{A_0(0)}\leq r$.
 Take an approximating family $\{A_{0,\epsilon}\}_{\epsilon> 0}$  provided 
 by Lemma~\ref{LemmaApproximationOfCoefficients}.
 Let $\{v_{0,\epsilon}\}_{\epsilon > 0}$ be a family of functions in $\Hh{k}^6$. Then there exists a number 
 $\epsilon_0 >0$ and a family of functions $\{v_{p,\epsilon}\}_{0 < \epsilon < \epsilon_0}$ in $\Hh{k}^6$ 
 such that
  \begin{align*}
  A_3 (A_{0,\epsilon}(0)^{-1} A_{3})^p v_{p,\epsilon} = A_3 v_{0,\epsilon}
 \end{align*}
 for all $\epsilon \in (0, \epsilon_0)$
 and a constant $C = C(\eta,r)$ such that
 \begin{align*}
    \Hhn{k}{v_{p,\epsilon}} \leq C \Hhn{k}{v_{0,\epsilon}}
 \end{align*}
 for all $\epsilon \in (0, \epsilon_0)$.
\end{lem}
\begin{proof}
 I) Due to the properties of the approximating family, we find an $\epsilon_0 > 0$ such 
  that
  \begin{align}
  \label{EquationBoundsForApproximatingSequenceCoefficients}
   \Fvarnorm{m-1}{A_{0,\epsilon}(0)} \leq 2 r
  \end{align}
  for all $\epsilon \in (0, \epsilon_0)$. We introduce the invertible matrix
  \begin{equation*}
   Q = \begin{pmatrix}
          0 &0 &0 &0 &1 &0 \\
          0 &0 &0 &-1 &0 &0 \\
          0 &0 &1 &0 &0  &0 \\
          0 &-1 &0 &0 &0  &0 \\
          1 &0 &0 &0 &0  &0 \\
          0 &0 &0 &0 &0  &1
         \end{pmatrix}
         \quad \text{so that} \quad 
    A_3 Q = \begin{pmatrix}
			   1 &0 &0 &0 &0 &0 \\
			   0 &1 &0 &0 &0 &0 \\
			   0 &0 &0 &0 &0  &0 \\
			   0 &0 &0 &1 &0  &0 \\
			   0 &0 &0 &0 &1  &0 \\
			   0 &0 &0 &0 &0  &0
                        \end{pmatrix}.
  \end{equation*}
  We further set
  \begin{align*}
   \Theta_\epsilon = \begin{pmatrix}
       A_{0,\epsilon; 33}  &A_{0,\epsilon;36}  \\
       A_{0,\epsilon;63}   &A_{0,\epsilon;66} 
       \end{pmatrix},
  \end{align*}
  which inherits the positive definiteness from $A_{0,\epsilon}$, i.e., $\Theta_\epsilon \geq \eta$ on $\Omega$.
  In particular, $\Theta_\epsilon$ has an inverse with
  \begin{equation}
   \label{EquationBoundForThetaInverse}
   \Fvarnorm{m-1}{\Theta_\epsilon^{-1}(0)} \leq C(\eta, r)
  \end{equation}
  for all $\epsilon \in (0, \epsilon_0)$.
  
  II) Let $w_0 \in \Hh{k}^6$. Due to the previous step
  we can define scalar functions $h_{1,\epsilon}$ and $h_{2,\epsilon}$ by
  \begin{equation*}
  	(h_{1,\epsilon}, h_{2,\epsilon}) = -\Theta_\epsilon^{-1}(0) (A_{0,\epsilon}(0) w_0)_{(3,6)},
  \end{equation*} 
  where we denote for any vector $\zeta$ from $\R^6$ by $\zeta_{(3,6)}$ the two-dimensional vector $(\zeta_3, \zeta_6)$.
  Note that
  \begin{equation}
  	\label{EquationBoundForh1epsh2eps}
  	\Hhn{k}{(h_{1,\epsilon}, h_{2,\epsilon})} \leq C(\eta,r) \Hhn{k}{w_0}
  \end{equation}
  for all $\epsilon \in (0,\epsilon_0)$ by Lemma~\ref{LemmaRegularityForA0}, \eqref{EquationBoundsForApproximatingSequenceCoefficients}, 
  and~\eqref{EquationBoundForThetaInverse}.
  We next set
  \begin{align}
  	\label{EquationDefinitionOfw1eps}
  	\tilde{w}_{0,\epsilon} &= -A_{0,\epsilon}(0)  \Big( w_0 + h_{1,\epsilon} e_3 + h_{2,\epsilon} e_6 \Big), \nonumber \\
  	\tilde{w}_{1,\epsilon} &= Q \tilde{w}_{0,\epsilon}
  \end{align}
  for all $\epsilon \in (0, \epsilon_0)$. We once more obtain a constant $C(\eta,r)$ such that
  \begin{equation*}
  	\Hhn{k}{\tilde{w}_{1,\epsilon}} \leq C(\eta,r) \Hhn{k}{w_0}
  \end{equation*}
  for all $\epsilon \in (0,\epsilon_0)$ due to Lemma~\ref{LemmaRegularityForA0}, \eqref{EquationBoundsForApproximatingSequenceCoefficients}, 
   and~\eqref{EquationBoundForh1epsh2eps}.
   We further point out that the construction of $h_{1,\epsilon}$, $h_{2,\epsilon}$, and $\tilde{w}_{0,\epsilon}$ yields
   \begin{equation*}
   	(\tilde{w}_{0,\epsilon})_{(3,6)} = (-A_{0,\epsilon}(0)   w_0)_{(3,6)} 
   		- \Theta_\epsilon(0)(h_{1,\epsilon},h_{2,\epsilon}) = 0
   \end{equation*}
   for all $\epsilon \in (0,\epsilon_0)$. In particular,
	\begin{equation*}
		A_3 Q \tilde{w}_{0,\epsilon} = \tilde{w}_{0,\epsilon}
	\end{equation*}
	for all $\epsilon \in (0,\epsilon_0)$. We thus compute
	\begin{align*}
		A_3(-A_{0,\epsilon}(0)^{-1} A_3) \tilde{w}_{1,\epsilon} &=  A_3 (-A_{0,\epsilon}(0)^{-1}) \tilde{w}_{0,\epsilon} = A_3 w_0
	\end{align*}
	for all $\epsilon \in (0,\epsilon_0)$, where we also used that the span of $e_3$ and $e_6$ 
	is the kernel of $A_3$. To sum up, we have shown that for each $w_0 \in \Hh{k}^6$ and $\epsilon \in (0,\epsilon_0)$, there is a function $w_{\epsilon} \in \Hh{k}^6$ such that
	\begin{equation}
		\label{EquationConstructionOfFunctionSummary}
		A_{3}(-A_{0,\epsilon}(0)^{-1} A_3) w_{\epsilon} = A_{3} w_0.
	\end{equation}
	Moreover, there is a constant $C = C(\eta,r)$, in particular independent of  $\epsilon$, such that
	\begin{equation}
		\label{EquationEstimateForConstructedFunctionInSummary}
		\Hhn{k}{w_{\epsilon}} \leq C \Hhn{k}{w_0}
	\end{equation}
	for all $\epsilon \in (0, \epsilon_0)$.
	
	III) To show the actual assertion, we proceed inductively. We claim that for all $p \in \N_0$, $\epsilon \in (0, \epsilon_0)$, 
	and $w \in \Hh{k}^6$ there is a function $w_{p,\epsilon}(w)$ in $\Hh{k}^6$ 
	and a constant $C_p = C_p(\eta,r)$
	such that
	\begin{align}
		\label{EquationInductionClaimForConstructionvdeltaepsilon}
		&A_{3} (-A_{0,\epsilon}(0)^{-1} A_3)^{p} w_{p,\epsilon}(w) = A_{3} w, \\
			\label{EquationInductionClaimForEstimatevdeltaepsilon}
		&\Hhn{k}{w_{p,\epsilon}(w)} \leq C_p \Hhn{k}{w}.
	\end{align}
	Note that there is nothing to show in the case $p = 0$. Now assume that we have proven the claim for a number $p \in \N_0$. Fix $\epsilon \in (0, \epsilon_0)$ and $w \in \Hh{k}^6$. Step II) applied with $w_0 = w$ yields a function 
	$\tilde{w}_{p,  \epsilon} \in \Hh{k}^6$ with
	\begin{align}
		 \label{EquationConstructionFirstStep}
		 A_{3} (-A_{0,\epsilon}(0)^{-1} A_3) \tilde{w}_{p, \epsilon} = A_{3} w
		\hspace{0.5em} \text{and} \hspace{0.5em} \Hhn{k}{ \tilde{w}_{p,  \epsilon}} \leq C(\eta,r) \Hhn{k}{w}.
	\end{align}
	We now define $w_{p+1, \epsilon}(w) = w_{p,\epsilon}(\tilde{w}_{p,  \epsilon})$ for each 
	$\epsilon \in (0, \epsilon_0)$. 
	Then $w_{p+1,  \epsilon}(w)$ is contained in $\Hh{k}^6$ and we compute
	\begin{align*}
		A_{3} (-A_{0,\epsilon}(0)^{-1} A_3)^{p+1} w_{p+1,  \epsilon}(w) &= A_{3}(-A_{0,\epsilon}(0)^{-1}) A_3 (-A_{0,\epsilon}(0)^{-1} A_3)^{p} w_{p, \epsilon}(\tilde{w}_{p,  \epsilon}) \\
		&= A_{3}(-A_{0,\epsilon}(0)^{-1}) A_3 \tilde{w}_{p,  \epsilon} 
		= A_{3} w,
	\end{align*}
	where we employed the induction hypothesis~\eqref{EquationInductionClaimForConstructionvdeltaepsilon} and~\eqref{EquationConstructionFirstStep}.
	Combining~\eqref{EquationInductionClaimForEstimatevdeltaepsilon} with~\eqref{EquationEstimateForConstructedFunctionInSummary}, we further obtain
	\begin{align*}
		\Hhn{k}{w_{p+1,\epsilon}(w)} = \Hhn{k}{w_{p,\epsilon}(\tilde{w}_{p,  \epsilon})} 
		\leq C_p \Hhn{k}{\tilde{w}_{p,  \epsilon}} \leq C \Hhn{k}{w},
	\end{align*}
	where $C = C(\eta,r)$. The claim now follows by induction.
	
	The assertion of the lemma is finally proven by setting $v_{p,\epsilon} = w_{p, \epsilon}(v_{0, \epsilon})$ 
	for all $\epsilon \in (0, \epsilon_0)$ and $p \in \N_0$.
\end{proof}

Applying now Proposition~\ref{PropositionRegularityForApproximatingProblem} to the solutions of 
the approximating initial boundary value problems with coefficients and data from Lemma~\ref{LemmaExistenceOfApproximatingSequence},
we derive the differentiability theorem.
\begin{theorem}
  \label{TheoremRegularityOfSolution}
 Let $\eta > 0$, $m \in \N$, and  $\tilde{m} = \max\{m,3\}$. Take coefficients
 $A_0 \in \Fupdwl{\tilde{m}}{cp}{\eta}$, $A_1, A_2 \in \Fcoeff{\tilde{m}}{cp}$, $A_3 = A_3^{\operatorname{co}}$,
  $D \in \Fuwl{\tilde{m}}{cp}$, and $B = B^{\operatorname{co}}$. Choose data
 $f \in \Ha{m}$, $g \in \E{m}$, and $u_0 \in \Hh{m}$ such that the tuple $(0, A_0, \ldots, A_3, D, B, f,g, u_0)$ 
 satisfies the compatibility conditions~\eqref{EquationCompatibilityConditionPrecised} of order $m$.
 Then the weak solution $u$ of~\eqref{IBVP} belongs to $\G{m}$.
\end{theorem}
\begin{proof}
 I) Let $\{A_{i,\epsilon}\}_{\epsilon > 0}$ and $\{D_\epsilon\}_{\epsilon > 0}$ be the families of functions 
 given by Lemma~\ref{LemmaApproximationOfCoefficients} for $A_i$, $i \in \{0,1,2\}$,
 and $D$ respectively. In particular, the coefficients $A_{0,\epsilon}$, $A_{1,\epsilon}$, 
 $A_{2,\epsilon}$, and $D_\epsilon$ belong to $C^\infty(\overline{\Omega})$
 and $\partial_t A_{0,\epsilon}$ is contained in $\F{\tilde{m}}$ for each $\epsilon > 0$. Moreover, $A_{1,\epsilon}$ and 
 $A_{2,\epsilon}$ are independent of time for all $\epsilon > 0$ as $A_1$ and $A_2$ have this property.
 Lemma~\ref{LemmaExistenceOfApproximatingSequence} provides a parameter $\epsilon_0 > 0$ and a family $\{u_{0,\epsilon}\}_{ 0 < \epsilon < \epsilon_0} \subseteq \Hh{m}$ 
 such that $(0, A_{0,\epsilon}, A_{1,\epsilon}, A_{2,\epsilon}, A_3, D_\epsilon, B, f, g, u_{0,\epsilon})$ fulfill the compatibility conditions~\eqref{EquationCompatibilityConditionPrecised} 
 of order $m$  for all $\epsilon \in (0, \epsilon_0)$ and $u_{0,\epsilon} \rightarrow u_0$ in $\Hh{m}$ as $\epsilon \rightarrow 0$. 
 Let $u_\epsilon$ denote the weak solution of~\eqref{IBVP} with differential operator $L(A_{0,\epsilon},A_{1,\epsilon}, A_{2,\epsilon}, A_3, D_\epsilon)$ 
 and inhomogeneity $f$, boundary value $g$, and initial value $u_{0,\epsilon}$ for each $\epsilon \in (0, \epsilon_0)$.
 By Proposition~\ref{PropositionRegularityForApproximatingProblem}, the function $u_\epsilon$ belongs 
 to $\G{m}$ for all $\epsilon \in (0, \epsilon_0)$. Take $r > 0$ such that 
 \begin{align*}
  \Fnorm{\tilde{m}}{A_i} \leq r \quad \text{and} \quad \Fnorm{\tilde{m}}{D} \leq r
 \end{align*}
 for all $i \in \{0, \ldots, 3\}$. Due to Lemma~\ref{LemmaApproximationOfCoefficients} we then also have
 \begin{align*}
  \Fnorm{\tilde{m}}{A_{i,\epsilon}} \leq C r \quad \text{and} \quad \Fnorm{\tilde{m}}{D_\epsilon} \leq C r
 \end{align*}
 for all $\epsilon \in (0, \epsilon_0)$ and $i \in \{0,1,2\}$. 
 Theorem~\ref{TheoremAPrioriEstimates} then yields a constant $C = C(\eta, r,T)$
 and a number $\gamma = \gamma(\eta, r,T)$ 
 such that
 \begin{align}
 \label{EquationEstimateForApproximatingSolutions}
  \Gnorm{m}{u_\epsilon}^2 \leq C\Big(&\sum_{j = 0}^{m-1} \Hhn{m-1-j}{\partial_t^j f(0)}^2 + \Enorm{m}{g}^2 + \Hhn{m}{u_{0,\epsilon}}^2 \nonumber \\
  &\hspace{3em} + \frac{1}{\gamma} \Hangamma{m}{f}^2 \Big)
 \end{align}
 for all $\epsilon \in (0, \epsilon_0)$.
 Let $(\epsilon_n)$ be a sequence 
 of positive numbers converging to zero. Then~\eqref{EquationEstimateForApproximatingSolutions} and 
 $u_{0,\epsilon} \rightarrow u_0$ in $\Hh{m}$ as $\epsilon \rightarrow 0$ yield that 
 $(\partial^\alpha u_{\epsilon_n})$ is bounded in $L^\infty(J, L^2(\R^3_+)) 
 = (L^1(J, L^2(\R^3_+)))^*$ for each $\alpha \in \N_0^4$ with $|\alpha| \leq m$. 
 Since $L^1(J, L^2(\R^3_+))$ is separable, the Banach-Alaoglu theorem gives a $\sigma^*$-convergent subsequence.
 Taking iteratively subsequences for each $\alpha \in \N_0^4$ with $|\alpha| \leq m$, we obtain a subsequence, 
 denoted by $(u_n)$, such that the $\sigma^*$-limit
 $u_\alpha$ of $\partial^\alpha u_n$ exists for all $\alpha \in \N_0^4$ 
 with $|\alpha| \leq m$.
 Lemma~\ref{LemmaEllerResult} and Lemma~\ref{LemmaApproximationOfCoefficients} imply that
 \begin{align*}
  \Gnorm{0}{u_n - u} &\leq C(\Gnorm{0}{L(A_0,\ldots, A_3, D) u_{n} - f}^2 + \Ltwohn{u_{0,n} - u_0}^2) \\
			&\leq C \Big(\sum_{i = 0}^2 \|A_i - A_{i,n}\|_{L^\infty(\Omega)}^2 \Gnorm{0}{\partial_i u_{n}}^2 \\
			    &\qquad + \|D - D_{n}\|_{L^\infty(\Omega)}^2 \Gnorm{0}{u_n}^2
			    + \Ltwohn{u_{0,n} - u_0}^2\Big) \longrightarrow 0
 \end{align*}
 as $n \rightarrow \infty$, where we also exploited that $f = L(A_{0,n},A_{1,n}, A_{2,n}, A_3, D_{n})u_n$, \eqref{EquationEstimateForApproximatingSolutions}, 
 and that $(u_{0,n})_n$ is bounded in $\Hh{m}$. 
 Consequently, $u$ is equal to $u_{(0,0,0,0)}$. Looking at the distributional derivative, 
 we further deduce
 \begin{align*}
  \langle \phi, \partial^\alpha u \rangle = (-1)^{|\alpha|} \langle \partial^\alpha \phi, u \rangle 
      = (-1)^{|\alpha|} \lim_{n \rightarrow \infty} \langle \partial^\alpha \phi, u_n \rangle = \langle \phi, u_\alpha \rangle
 \end{align*}
 for all $\phi \in C_c^\infty(\Omega)$. We conclude that $\partial^\alpha u \in L^\infty(J, \Ltwoh)$ for all $\alpha 
 \in \N_0^4$ with $|\alpha| \leq m$; i.e., $u \in \Gvar{m}$. It remains to remove the tilde here.
 
 II) Let $0 \leq j \leq m-1$. Differentiating the differential equation and 
 the boundary condition and employing~\eqref{EquationTimeDerivativesOfSolution}, we see that
 $\partial_t^j u$ solves  the initial boundary value problem,
 \begin{align*}
  \left\{\begin{aligned}
   L(A_0, \ldots, A_3, D) v  &= f_j, \quad &&x \in \R^3_+, \quad &t \in J; \\
   B v &= \partial_t^j g, \quad &&x \in \partial \R^3_+, &t \in J; \\
   v(0) &= S_{m,j}(0, A_0, \ldots, A_3, D, f, u_0), \quad &&x \in \R^3_+;
\end{aligned}\right.
 \end{align*}
 where 
 \begin{equation*}
   f_j = \partial_t^j f - \sum_{l = 1}^j \binom{j}{l} 
    \Big(\partial_t^l A_0 \partial_t^{j+1-l} u + \partial_t^l D \partial_t^{j-l} u\Big)
 \end{equation*}
 belongs to $\Ha{m-j}$ by Lemma~\ref{LemmaRegularityForA0}.
 We want to apply Lemma~\ref{LemmaBasicRegularityInTime} to $\partial_t^j u$.
 Therefore, the tuple $(0, A_0, \ldots, A_3, D, B, f_j, \partial_t^j g, u_0^j)$ has to satisfy
 the compatibility conditions~\eqref{EquationCompatibilityConditionPrecised} 
 of order $m - j$, where we abbreviate  $S_{m,j}(0,A_0,\ldots, A_3, D,f,u_0)$ by $u_0^j$ for all $0 \leq j \leq m$. 
  
 Similar to Lemma~\ref{LemmaHigherOrderCompatibilityConditions}, an induction with 
 a straightforward calculation in the induction step yields
  \begin{align}
 \label{EquationConstructionForCCWithSameDiffOp}
  S_{m-m_1,m_2}(0,A_0,\ldots, A_3,D, f_{m_1}, u_0^{m_1}) = S_{m, m_1 + m_2}(0,A_0,\ldots, A_3, D,f, u_0)
 \end{align}
 for all $m_1, m_2 \in \N_0$ with $m_2 \leq m_1$ and $m_1 + m_2 \leq m-1$. We once more refer to~\cite{SpitzDissertation}, 
 see step~III) of the proof of Theorem~4.12 there, 
 for the details.
 Note that this identity implies that
 \begin{align*}
  B S_{m-m_1,m_2} (0,A_0,\ldots, A_3,D, f_{m_1}, u_0^{m_1}) &= B S_{m, m_1 + m_2}(0,A_0,\ldots, A_3, D,f, u_0) \\
  &= \partial_t^{m_1 + m_2} g(0) 
  = \partial_t^{m_2} (\partial_t^{m_1} g)(0)
 \end{align*}
 on $\partial \R^3_+$ for all $m_2 \leq m  - m_1 -1$, as the tuple $(0, A_0, \ldots, A_3, D, B, f, g, u_0)$ fulfills the compatibility conditions 
 of order $m$ by assumption. We infer that the 
 tuple $(0, A_0, \ldots, A_3, D, B, f_{m_1} , \partial_t^{m_1} g, u_0^{m_1})$ fulfills the compatibility conditions~\eqref{EquationCompatibilityConditionPrecised} of order 
 $m - m_1$.
 
  III) Step II) applied with $j = m - 1$ shows that $\partial_t^{m-1} u$ solves~\eqref{IBVP} with inhomogeneity
  $f_{m-1} \in \Ha{1}$, boundary value $\partial^{m-1}_t g \in \E{1}$, 
  and initial value $u_0^{m-1} \in \Hh{1}$. The tuple $(0, A_0, \ldots, A_3, D, B, f_{m-1}, \partial_t^{m-1} g, u_0^{m-1})$ fulfills the compatibility
  conditions~\eqref{EquationCompatibilityConditionPrecised} of order $1$ by step II). Next take an open subinterval 
  $J'$ of $J$. Assume that $\partial_t^{m-1} u$ belongs to $C^1(\overline{J'}, \Ltwoh)$. 
  As we already know that $\partial_t^{m-1} u$ belongs to $L^\infty(J, \Hhta{1})$, 
  we can argue as in step II) of the proof of Lemma~\ref{LemmaRegularityInSpaceTangential} 
  to infer that $\partial_t^{m-1}u$ is an element of $C(\overline{J'}, \Hhta{1})$. (Note 
  that the smoothness of the coefficients is not used in that step of the proof of 
  Lemma~\ref{LemmaRegularityInSpaceTangential}.) Lemma~\ref{LemmaRegularityInNormalDirection} then implies that $\partial_t^{m-1} u$
  is contained in $G_1(J' \times \R^3_+)$.
  Lemma~\ref{LemmaBasicRegularityInTime} thus yields
  that $\partial_t^{m-1} u$ belongs to $C^1(\overline{J},\Ltwoh)$; i.e., $u \in C^m(\overline{J},\Ltwoh)$. 
  The previous arguments applied with $J' = J$ now imply that $\partial_t^{m-1}u$ is an element of 
  $\G{1}$.
  
  Next assume that we have proven that $\partial_t^{m-k} u$ is an element of $\G{k}$ for some $k \in \{1,\ldots,m-1\}$. Then $\partial_t^{m-k-1}u$ 
  belongs to 
\begin{equation*}  
  \bigcap_{l=0}^{k} C^{l+1}(\overline{J}, \Hh{k-l}) = \bigcap_{l = 1}^{k+1} C^l(\overline{J},\Hh{k + 1 - l}).
\end{equation*}
  Observe that $\partial_t^{m-k-1} u$ solves~\eqref{IBVP} with inhomogeneity $f_{m - k - 1} \in \Ha{k + 1}$, boundary value 
  $\partial_t^{m-k-1} g \in \E{k+1}$, and initial 
  value $u_0^{m-k-1} \in \Hh{k + 1}$ by  
  step~II). Arguing as before, i.e., combining Lemma~\ref{LemmaBasicRegularityInTime} 
  with step II) of Lemma~\ref{LemmaRegularityInSpaceTangential} and Lemma~\ref{LemmaRegularityInNormalDirection}, 
  we derive that $\partial_t^{m-k-1} u$ belongs to $C(\clJ,\Hh{k+1})$ and thus to $\G{k+1}$. 
  
  By induction we arrive at $\partial_t^{m-k} u \in \G{k}$ for all $k \in \{0, \ldots, m\}$. With $k = m$ we finally 
  obtain $u \in \G{m}$.
 \end{proof}
 
 \textbf{Proof of Theorem~\ref{TheoremMainResultOnDomain}:} Combining Theorems~\ref{TheoremAPrioriEstimates} and~\ref{TheoremRegularityOfSolution}, we derive
 the assertion of Theorem~\ref{TheoremMainResultOnDomain} for $G = \R^3_+$. The localization 
 procedure from Section~\ref{SectionLocalizationAndFunctionSpaces}, see Remark~\ref{RemarkAboutLocalization}, 
 then yields Theorem~\ref{TheoremMainResultOnDomain} for coefficients constant outside of a compact set. 
 
 Once the regularity theory has been established for coefficients constant outside of a 
   compact set, another approximation procedure extends the results to coefficients $A_0$ and $D$ 
   which merely have a limit as $|(t,x)| \rightarrow \infty$. We refer to~\cite[Theorem~4.13]{SpitzDissertation} for details.
 \hfill $\qed$

 \begin{rem}
  Not only the main result extends to coefficients $A_0$ and $D$ with a limit as $|(t,x)| \rightarrow \infty$, 
  but also all the intermediate results. In particular, Proposition~\ref{PropositionCentralEstimateInNormalDirection}, 
  Theorem~\ref{TheoremAPrioriEstimates}, and Theorem~\ref{TheoremRegularityOfSolution} are still
  true if $A_0$ and $D$ only have a limit as $|(t,x)| \rightarrow \infty$, see~\cite[Theorem~4.13]{SpitzDissertation}.
 \end{rem}
 
 \vspace{2em}
 \textbf{Acknowledgment:}
 I want to thank my advisor Roland Schnaubelt for stimulating discussions concerning this research 
and helpful suggestions during the preparation of this article. Moreover, I gratefully acknowledge 
financial support by the Deutsche Forschungsgemeinschaft (DFG) through CRC 1173.

\vspace{1em} 

\emergencystretch = 1em
\printbibliography
\vspace{2em}
\end{document}